\pdfoutput=1
\documentclass[12pt]{article}
\usepackage[T1]{fontenc}
\usepackage{lmodern}

\usepackage{color}
\usepackage{graphicx}
\usepackage{amsmath,amsfonts,amssymb,graphics,amsthm}
\usepackage{hyperref}
\usepackage{comment}
\usepackage{tabularx}
\usepackage[protrusion=true,expansion=true]{microtype}
\usepackage{enumerate}
\usepackage{bbm}

\usepackage[margin=1.2in]{geometry}

\hypersetup{
    colorlinks=false,
    linktocpage,
    }

\numberwithin{equation}{section}

\newtheorem{theorem}{Theorem}[section]

\newtheorem{prop}[theorem]{Proposition} 
\newtheorem{lem}[theorem]{Lemma} 
\newtheorem{cor}[theorem]{Corollary}

\theoremstyle{definition}
\newtheorem{remark}[theorem]{Remark}

\newtheorem{defn}[theorem]{Definition}
\theoremstyle{remark}\newtheorem{notation}[theorem]{Notation}

\def\@rst #1 #2other{#1}
\newcommand\MR[1]{\relax\ifhmode\unskip\spacefactor3000 \space\fi
  \MRhref{\expandafter\@rst #1 other}{#1}}
\newcommand{\MRhref}[2]{\href{http://www.ams.org/mathscinet-getitem?mr=#1}{MR#2}}

\def\MR#1{\href{http://www.ams.org/mathscinet-getitem?mr=#1}{MR#1}}

\newcommand{\C}{\mathbbm{C}}

\newcommand{\E}{\mathbbm{E}}
\newcommand{\N}{\mathbbm{N}}
\newcommand{\Nz}{\mathbbm{N}_0}
\newcommand{\Q}{\mathbbm{Q}}
\newcommand{\Z}{\mathbbm{Z}}
\newcommand{\R}{\mathbbm{R}}

\newcommand{\eps}{\varepsilon}
\newcommand{\1}{\mathbf{1}}

\DeclareMathOperator{\Var}{Var}
\DeclareMathOperator{\SLE}{SLE}

\def\cZ{\mathcal{Z}}

\def\cX{\mathcal{X}}

\def\cT{\mathcal{T}}

\def\cR{\mathcal{R}}

\def\cO{\mathcal{O}}

\def\cL{\mathcal{L}}

\def\cl{\mathfrak{l}}

\newcommand{\aryb}{\begin{eqnarray*}}
\newcommand{\arye}{\end{eqnarray*}}
\def\alb#1\ale{\begin{align*}#1\end{align*}}
\newcommand{\eqb}{\begin{equation}}
\newcommand{\eqe}{\end{equation}}
\newcommand{\eqbn}{\begin{equation*}}
\newcommand{\eqen}{\end{equation*}}

\newcommand{\BB}{\mathbbm}
\newcommand{\ol}{\overline}

\newcommand{\op}{\operatorname}

\newcommand{\frk}{\mathfrak}
\newcommand{\eqD}{\overset{d}{=}}
\newcommand{\rtaD}{\overset{d}{\rightarrow}}
\newcommand{\ep}{\epsilon}
\newcommand{\rta}{\rightarrow}

\newcommand{\wt}{\widetilde}
\newcommand{\wh}{\widehat} 
\newcommand{\mcl}{\mathcal}

\newcommand{\bdy}{\partial}

\DeclareMathAlphabet{\mathpzc}{OT1}{pzc}{m}{it}

\newcommand{\old}[1]{{}} 
\begin{document}

\author{
\begin{tabular}{c}Ewain Gwynne\\[-5pt]\small  University of Chicago\end{tabular}\;
\begin{tabular}{c}Nina Holden\\[-5pt]\small New York Univeristy\end{tabular}\;
\begin{tabular}{c}Xin Sun\\[-5pt]\small Peking University\end{tabular}}

\title{Joint scaling limit of a bipolar-oriented triangulation and its dual in the peanosphere sense}
\date{}

\maketitle

\begin{abstract} 
Kenyon, Miller, Sheffield, and Wilson (2015) showed how to encode a random bipolar-oriented planar map by means of a random walk with a certain step size distribution. Using this encoding together with the mating-of-trees construction of Liouville quantum gravity (LQG) due to Duplantier, Miller, and Sheffield (2014), they proved that random bipolar-oriented planar maps converge in the scaling limit to a $\sqrt{4/3}$-LQG surface decorated by an independent SLE$_{12}$ in the peanosphere sense, meaning that the height functions of a particular pair of trees on the maps converge in the scaling limit to the correlated planar Brownian motion which encodes the SLE-decorated LQG surface. We improve this convergence result by proving that the pair of height functions for an infinite-volume random bipolar-oriented triangulation and the pair of height functions for its dual map converge jointly in law in the scaling limit to the two planar Brownian motions which encode the same $\sqrt{4/3}$-LQG surface decorated by both an SLE$_{12}$ curve and the ``dual'' SLE$_{12}$ curve which travels in a direction perpendicular (in the sense of imaginary geometry) to the original curve. This confirms a conjecture of Kenyon, Miller, Sheffield, and Wilson (2015). {Our paper is the starting point of recent  works connecting LQG and random permutons such as the Baxter permuton.}
\end{abstract}

\tableofcontents

\parindent 0 pt
\setlength{\parskip}{0.25cm plus1mm minus1mm}

\section{Introduction}
\label{sec:intro}

\subsection{Overview}
\label{sec:overview}

A planar map is a planar graph together with an embedding into $\R^2$, where the embedding is specified up to orientation preserving homeomorphisms of $\BB R^2$. In recent years there has been an interest in understanding the geometry of randomly chosen planar maps, partly motivated by applications in quantum gravity, conformal field theory, and string theory. In this context, a random planar map is a natural model of a discrete random surface.

In this paper, we will consider random \emph{bipolar-oriented planar maps}. A bipolar-oriented map is a pair $(G,\mcl O)$ where $G$ is a graph and $\mcl O$ is an orientation on the edges of $G$ with a unique source and sink. 
We note that a bipolar-oriented planar map cannot have self-loops, but can have multiple edges.

Bipolar orientations, especially on planar maps, have a rich combinatorial structure and numerous applications in algorithms. For an overview of the graph theoretic perspective on bipolar orientations, we refer to~\cite{graph} and references therein. See also~\cite{counting,bijection} and the references therein for enumerative and bijective results for bipolar-orientated planar maps. 
A random bipolar orientation on (a subgraph of) the square lattice $\BB Z^2$ is equivalent to the six-vertex model via a bijection described in~\cite{kmsw-6vertex}.

If $(G,\mcl O)$ is a bipolar-oriented planar map, then there exists an embedding of $G$ into $\BB R^2$ such that each edge points in the northwest direction. 
For such an embedding, the source is the same as the southeast pole (i.e., the southeastern-most vertex); and the sink is the same as the northwest pole (i.e., the northwestern-most vertex). 
It is explained in~\cite{kmsw-bipolar} that a bipolar-oriented planar map is decorated with a canonical \emph{east-going tree} rooted at the southeast pole and a \emph{west-going tree} rooted at the northwest pole. These trees give rise to a natural Peano curve  on $G$ which snakes between the two trees and traverses each edge of $G$ and each face of $G$ exactly once; see Figure~\ref{fig:intro-bipolarmap} for an illustration. 

Let $\mcl Z = (\mcl L , \mcl R) $ be the two-dimensional walk whose coordinates are the height functions of the east-going and west-going trees. Then $\mcl Z$ uniquely determines $G$ via a bijection which is defined in~\cite{kmsw-bipolar} and reviewed in Section~\ref{sec:bipolar-background} below. 
It is shown in~\cite{kmsw-bipolar} that for a large class of probability measures on bipolar-oriented planar maps, including uniformly random bipolar-oriented $k$-angulations for $k\geq 3$ with the source and sink adjacent, the law of the corresponding walk $\mcl Z$ is that of a random walk with iid increments subject to certain conditioning. 
Moreover, the correlation of the two coordinates of the step size distribution is always $-1/2$. Hence in this case $\mcl Z$ converges in the scaling limit to a conditioned correlated two-dimensional Brownian motion with correlation $-1/2$. 

We can remove the conditioning by considering a local limit of bipolar-oriented planar maps in the Benjamini-Schramm topology~\cite{benjamini-schramm-topology} based at a uniformly random edge, so that the source and sink both lie at $\infty$. {This local limit was shown to exist in~\cite[Section 3.3]{ghs-map-dist}.} In this case the corresponding walk $\mcl Z$ is an unconditioned two-sided simple random walk with iid increments and its scaling limit is a two-sided correlated two-dimensional Brownian motion with correlation $-1/2$. 

If $(G,\mcl O)$ is a bipolar-oriented planar map and $\wt G$ is the dual map of $G$, then $\mcl O$ induces a \emph{dual bipolar orientation} $\wt{\mcl O}$ on $\wt G$. This orientation gives rise to \emph{north-going} and \emph{south-going} trees on $\wt G$ and thereby a pair of dual height functions $\wt{\mcl Z} = (\wt{\mcl L} , \wt{\mcl R})$ (see Figure~\ref{fig:intro-dual}). 
The main theorem of this paper is a scaling limit result for the joint law of the pair of random walks $(\mcl Z , \wt{\mcl Z})$ in the case of a uniform-infinite bipolar-oriented triangulation (which is the local limit of uniform bipolar-oriented triangulations). 

The description of the limiting law will be given in terms of the left/right boundary length processers of a pair of whole-plane SLE$_{12}$ curves from $\infty$ to $\infty$, coupled together in the sense of imaginary geometry~\cite{ig1,ig2,ig3,ig4}, on an independent \emph{$\sqrt{4/3}$-Liouville quantum gravity cone}. We will now discuss these objects and their relationship to bipolar-oriented planar maps. 
 
A \emph{Liouville quantum gravity} (LQG) surface is a natural model of a continuum random surface, which is defined as follows. Let $\gamma\in(0,2)$ and let $h$ be an instance of the Gaussian free field (GFF) or a related distribution on a domain $D\subset \BB C$. Formally the $\gamma$-LQG surface associated with $h$ is the Riemannian manifold with metric tensor given by
\begin{equation}
\label{eqn:lqg_metric}
e^{\gamma h(z)} (dx^2 + dy^2) ,
\end{equation}
where $dx^2 + dy^2$ denotes the Euclidean metric tensor on $D$. This expression does not make literal sense since $h$ is a distribution and does not take values at points, but one can make sense of the volume form associated with~\eqref{eqn:lqg_metric} as a random measure~\cite{shef-kpz}. See also~\cite{rhodes-vargas-review} and the references therein for an equivalent formulation. {See~\cite{dddf-lfpp,gm-uniqueness} for a construction of the random distance function associated with~\eqref{eqn:lqg_metric}.}

It is expected that a wide class of random planar map models converge under an appropriate scaling limit to LQG surfaces (with parameter $\gamma$ depending on the model) as the size of the map tends to $\infty$. In many of these models, the random planar map is naturally decorated by a statistical physics model which can be represented by a curve or a collection of loops. Often such curves or collections of loops are expected to converge in the scaling limit to a \emph{Schramm-Loewner evolution (SLE) curve}~\cite{schramm0} or a \emph{conformal loop ensemble (CLE)~\cite{shef-cle}}, independent from the quantum surface, with parameter\footnote{Here and throughout the rest of the paper we use the convention of~\cite{ig1,ig2,ig3,ig4} of writing $\kappa' > 4$ for the $\SLE$ parameter and $\kappa = 16/\kappa'  \in(0,4) $ for the dual parameter.} given by either $\kappa = \gamma^2 \in (0,4)$ or $\kappa '  = 16/\gamma^2 > 4$.
See~\cite{shef-kpz,shef-cle,shef-zipper,shef-burger,wedges,kmsw-bipolar} for rigorous mathematical formulations of conjectures of this form and the references therein for related predictions in the physics literature. 
In particular, a random bipolar-oriented planar map decorated by the Peano curve which snakes between the pair of the east- and west-going trees discussed above corresponds to SLE$_{12}$-decorated $\sqrt{4/3}$-LQG~\cite{kmsw-bipolar}. 

There are various ways to formulate the above scaling limit conjectures. One way is to conformally embed a random planar map into $\BB C$ (e.g.\ via circle packing or Riemann uniformization) and try to prove that the random area measure which assigns a point mass to each vertex converges under an appropriate scaling limit to the LQG measure~\cite{shef-kpz,shef-zipper,wedges} and that the statistical mechanics model, thus embedded, converges in an appropriate sense to SLE or CLE. 
{Such scaling limit results have been established  for uniform triangulations under the Cardy-Smirnov embedding\footnote{It was formerly referred as Cardy embedding.}~\cite{hs-cardy-embedding} and mated-CRT maps under the Tutte embedding~\cite{gms-tutte}.} 
{There are also convergence results for random planar maps as metric spaces.} LeGall and Miermont \cite{legall-uniqueness,miermont-brownian-map} proved that uniform quadrangulations (equipped with a re-scaling of the graph distance) converge in law with respect to the Gromov-Hausdorff topology to a continuum metric space called the \emph{Brownian map}. In the recent  works \cite{sphere-constructions,tbm-characterization,lqg-tbm1,lqg-tbm2,lqg-tbm3} the authors show that there is a metric on a $\gamma$-LQG sphere for $\gamma=\sqrt{8/3}$~\cite{wedges,sphere-constructions} under which it is isometric to the Brownian map. However, metric scaling limit results have not been rigorously established for $\gamma \not=\sqrt{8/3}$.

An alternative notion of convergence for {a broader class of} random planar map models has been established based on the so-called \emph{peanosphere construction} of~\cite{wedges}. This construction encodes a $\gamma$-LQG cone, a particular type of $\gamma$-LQG surface parametrized by $\BB C$, decorated by an independent whole-plane space-filling SLE$_{\kappa'}$ curve from $\infty$ to $\infty$~\cite{ig4} for $\kappa' = 16/\gamma^2$ in terms of correlated two-dimensional Brownian motion $Z= (L,R)$ with correlation $-\cos(\pi\gamma^2/4)$ (see \cite{kappa8-cov} for a calculation of this correlation in the case when $\gamma \in (0,\sqrt 2)$). We will review the details of this construction in Section~\ref{sec:peanosphere-background} below. 

Certain random planar maps, possibly decorated with a statistical physics model, can be encoded by means of two-dimensional random walks via discrete analogues of the peanosphere construction. If one can show that the walk encoding the random planar map converges in the scaling limit to a two-dimensional Brownian motion with correlation $-\cos(\pi\gamma^2/4)$, then one obtains convergence of the random planar map model toward SLE-decorated $\gamma$-LQG in a certain sense, which we call the \emph{peanosphere sense}. 

In particular, the above-mentioned scaling limit result of~\cite{kmsw-bipolar} for the random walk $\mcl Z$ associated with an infinite-volume random bipolar-oriented planar map can be re-phrased as the statement that such a planar map converges in the scaling limit to whole-plane SLE$_{12}$ from $\infty$ to $\infty$ on an independent $\sqrt{4/3}$-LQG cone. 
Peanosphere scaling limit results for other random planar map models are proven in~\cite{shef-burger,gms-burger-cone,gms-burger-local,gms-burger-finite,gkmw-burger}. 
 
It is natural to expect that the joint law of an infinite-volume bipolar-oriented map and its dual, each equipped with the Peano curve which snakes between the associated trees, converges in the scaling limit to a single $\sqrt{4/3}$-LQG cone decorated by an independent \emph{pair} of space-filling SLE$_{12}$ curves $\eta'$ and $\wt{\eta}'$ which (in some sense) travel in an orthogonal direction to each other. In \cite{ig4} space-filling SLE$_{\kappa'}$ is constructed as the Peano curve tracing the interface between the tree of $\theta$-angle imaginary geometry flow lines of a whole-plane Gaussian free field and the tree of $(\pi+\theta)$-angle flow lines of the same field. 
Hence one may expect that the limiting curves $\eta'$ and $\wt{\eta}'$ are constructed from the same Gaussian free field in such a way that the angle between the corresponding trees of flow lines is $\frac{\pi}{2}$. 

In \cite[Conjecture 1]{kmsw-bipolar}, it is conjectured that this is indeed the case, at least in the peanosphere sense. Our main result Theorem~\ref{thm1} below is a conformation of the conjecture in \cite{kmsw-bipolar} in the case of uniform infinite bipolar-oriented triangulations.
In other words, we show that the joint law of the height-function pairs $( \mcl Z, \wt{\mcl Z})$ for a uniform infinite bipolar-oriented triangulation converges in the scaling limit to the two correlated two-dimensional Brownian motions which encode $\eta'$ and $\wt\eta'$ via the peanosphere construction. 
This result constitutes the first non-trivial peanosphere scaling limit result for the \emph{joint} law of two different space-filling curves on (almost) the same planar map. 

In the course of proving our main result we will also obtain a description of the $\theta$-angle flow lines of the Gaussian free field used to construct the space-filling SLE$_{\kappa'}$ in terms of the peanosphere Brownian motion for general $\kappa' > 4$ and $\theta\in (0,\pi)$; see Section~\ref{sec:continuum-decomp}.

\bigskip

\begin{remark}
Since the first version of our paper appeared on  the arXiv, there has been significant further developments concerning LQG as the scaling limit of random combinatorial objects, where the mating-of-trees theory has played a crucial role. See~\cite{ghs-mating-survey} for a survey, and see~\cite{bp-lqg-notes,sheffield-icm} for more updates on the general subject of LQG. In particular, the present paper is the starting point of a substantial body of literature which relates LQG decorated by SLE to the scaling limits of certain pattern-avoiding random permutations. To be more specific, Borga~\cite{borga-skew-permuton}  introduces a class of random permutons called \emph{skew Brownian permutons} and shows, using the results of Section~\ref{sec:continuum-decomp} of the present paper, that they can be described in terms of SLE-decorated LQG. Other works show that the skew Brownian permutons describe the scaling limits of various types of random permutations~\cite{borga-strong-baxter,bm-baxter-permutation}; and use the connection to SLE and LQG to prove properties of random permutations and their scaling limits~\cite{bhsy-baxter-permuton,bgs-meander,bdg-skew-lis}. In addition, analogous results to the ones in the present paper are proven for random planar maps decorated by Schnyder woods in~\cite{lsw-schnyder-wood}. Finally, the recent paper~\cite{asy-p-theta} proves an exact formula for the relationship between the flow line angle $\theta$ and the parameter $p$ in the setting of Proposition~\ref{prop:ppp} below. 
\end{remark}

\noindent \textbf{Acknowledgements}
We thank Scott Sheffield for helpful discussions. E.G.\ was supported by the U.S. Department of Defense via an NDSEG fellowship. N.H.\ was supported by a doctoral research fellowship from the Norwegian Research Council.  {This paper was written when the authors were Ph.D. students at MIT.}

\subsection{Basic notation} 

Here we record some basic notations which we will use throughout this paper.  

\begin{notation} \label{def-discrete-intervals}
	For $a < b \in \BB R$, we define the discrete intervals $[a,b]_{\BB Z} := [a, b]\cap \BB Z$ and $(a,b)_{\BB Z} := (a,b)\cap \BB Z$. 
\end{notation}

\begin{notation}\label{def-asymp}
	If $a$ and $b$ are two quantities, we write $a\preceq b$ (resp.\ $a \succeq b$) if there is a constant $C$ (independent of the parameters of interest) such that $a \leq C b$ (resp.\ $a \geq C b$). We write $a \asymp b$ if $a\preceq b$ and $a \succeq b$. 
\end{notation}

\begin{notation} \label{def-o-notation}
	If $a$ and $b$ are two quantities which depend on a parameter $x$, we write $a = o_x(b)$ (resp.\ $a = O_x(b)$) if $a/b \rta 0$ (resp.\ $a/b$ remains bounded) as $x \rta 0$ (or as $x\rta\infty$, depending on context). 
\end{notation}

Unless otherwise stated, all implicit constants in $\asymp, \preceq$, and $\succeq$ and $O_x(\cdot)$ and $o_x(\cdot)$ errors involved in the proof of a result are required to satisfy the same dependencies as described in the statement of said result.

\subsection{Background on bipolar-oriented maps}
\label{sec:bipolar-background}

In this subsection we review some background on bipolar-oriented planar maps following~\cite{kmsw-bipolar} and give a precise definition of the uniform infinite bipolar-orientated triangulation.

\subsubsection{Bipolar-oriented planar maps}
\label{sec:bipolar-map}

In this paper we work on planar maps with no self-loops, but multiple edges are allowed.  A planar map is called a triangulation if every face is incident to 3 edges. Given a planar map $G$, draw a dual vertex in each face and connect a pair of dual vertices when their corresponding faces are adjacent. This produces a new planar map which is called the \emph{dual map} of $ G$. 

An orientation of a graph is a way of assigning each of its edges a direction. A vertex is called a sink (resp. source) if there are no outgoing
(resp. incoming) edges incident to the vertex.  An orientation $\cO$ on a planar map is \emph{bipolar} if it is acyclic and has a single source and a single sink. An orientation on a planar map is bipolar if and only if the map can be embedded into $\BB C$ in such a way that every edge is oriented upward, which means that the starting point of the edge is lower than the ending point of the edge. See Figure~\ref{fig:intro-bipolarmap} for an illustration. We will always assume that our bipolar-oriented planar maps are embedded in the plane in this manner. 

In order to make the directions of flow lines in the discrete and continuum pictures the same, when we consider an embedding of a bipolar-oriented map we will treat the ``compass'' in $\BB C$ as being rotated clockwise by 45 degrees, so that the upward direction is northwest and the downward direction is southeast (c.f.\ Remark~\ref{rmk:compass}). Hence the edges of an embedded planar map are oriented from southeast to northwest, and it is natural to call the sink the \textit{northwest pole} and the source the \textit{southeast pole}. 

It is also natural to draw an unoriented edge between the two poles that divides the outer face into a \textit{southwest pole face} and a \textit{northeast pole face}.  
The new map including the unoriented edge between the two poles is called the \emph{augmented map} of the original map.

\begin{figure}[ht!]
	\begin{center}
		\includegraphics[scale=0.7]{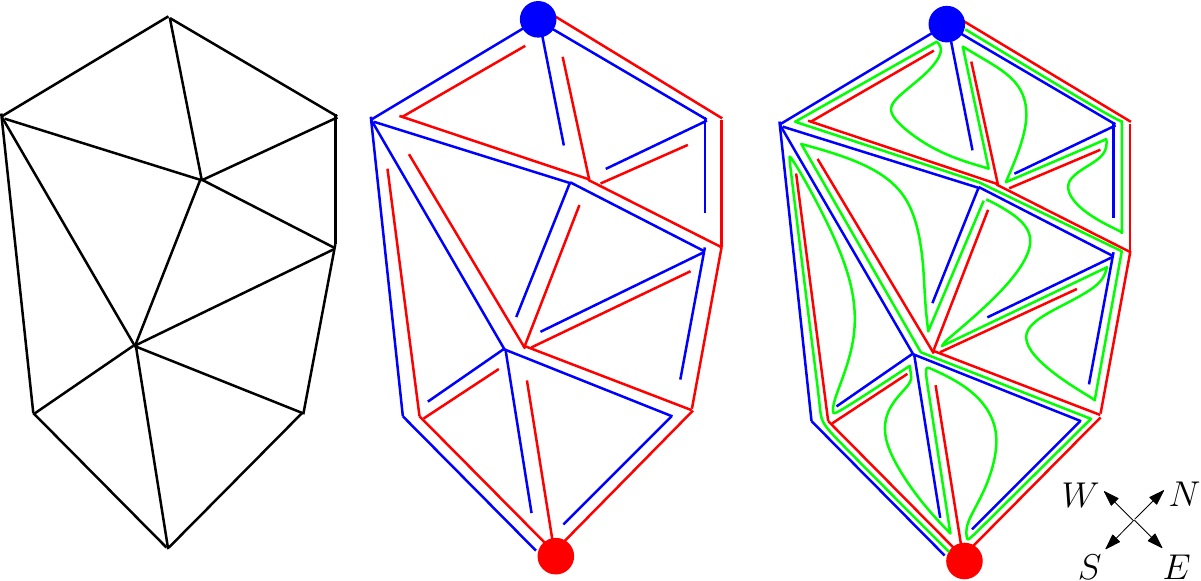}
	\end{center}
	\caption{
		Left: A bipolar-oriented triangulation $(G , \mcl O)$ embedded into the plane in such a way that direction of every edge is upwards. Middle: one can associate an east-going tree $\mcl T^E$ (shown in red) with $(G , \mcl O)$ by cutting each edge other than the leftmost edge leading upward from each vertex at the point where it meets that vertex; and a west-going tree $\mcl T^W$ (shown in blue) by cutting each edge other than the rightmost edge leading downward from each vertex at the point where it meets that vertex. Right: the exploration path $\lambda'$ (shown in green) associated with $(G , \mcl O)$, which traces the interface between the trees, starting at the southeast pole and ending at the northwest pole. 
	} \label{fig:intro-bipolarmap}
\end{figure} 

\begin{figure}[ht!]
	\begin{center}
		\includegraphics[scale=0.7]{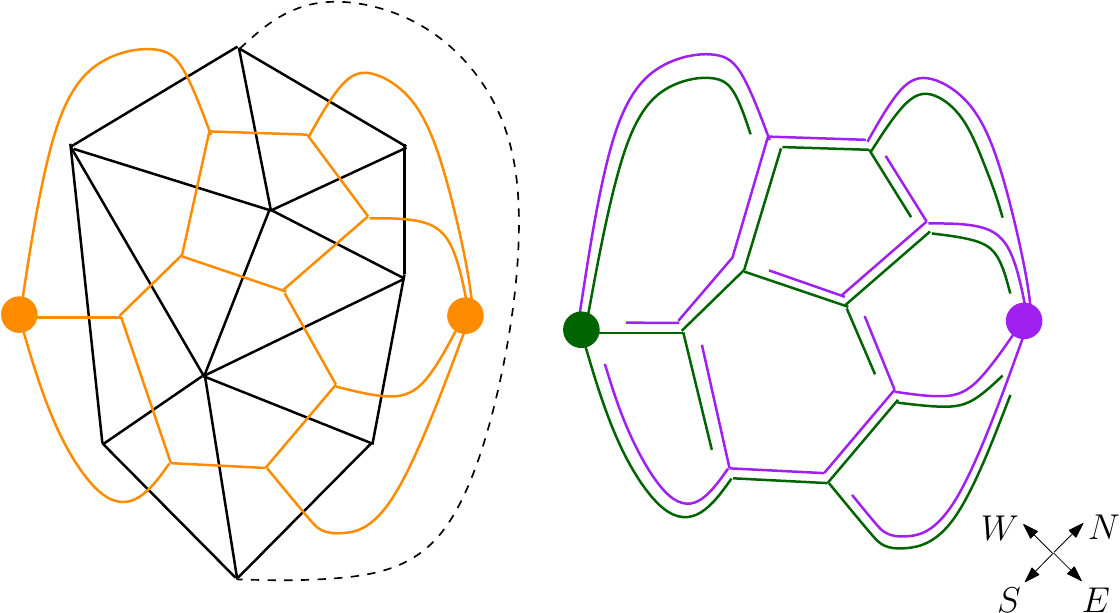}
	\end{center}
	\caption{The dual map $(\wt G  , \wt \cO)$ (orange) associated with a bipolar-oriented triangulation $(G , \cO)$ and its associated trees. Note that the northeast and southwest poles of $(\wt G  ,\wt \cO)$ correspond to the northeast and southwest pole faces in the augmented version of $G$. The extra edge added to $G$ to form the augmented map is shown as a dashed line. The edges of the dual map are oriented leftwards, i.e.\ from northeast to southwest. The trees associated with $(\wt G ,\wt {\cO})$ are the north-going tree (purple) and the south-going tree (green).
	} \label{fig:intro-dual}
\end{figure}

Given a bipolar orientation $\cO$ on a planar map $G$, by reversing the direction of edges, we obtain another bipolar orientation $\ol\cO$ on $G$. Let $\wt G$ be the dual map of the augmented map of $G$. By declaring that each edge of the $\wt G$ travels from northeast to southwest, we obtain a dual orientation $\wt \cO$ on $\wt G$ which is also bipolar. See Figure~\ref{fig:intro-dual}. The northeast (resp. southwest) pole face of $G$ becomes the northeast (resp. southwest) pole vertex while the southeast (resp. northwest) pole of $G$ becomes the southwest (resp. northeast) pole face of $\wt G$. Therefore $\cO$ induces three other bipolar orientation $\ol\cO,\wt\cO,\ol{\wt \cO}$, each of which also determines $\cO$.  

\begin{remark}
	The convention described above of assigning compass directions to the planar map differs from the one used in \cite{kmsw-bipolar} in that our compass directions are rotated by 45 degrees. In~\cite{kmsw-bipolar} the upward direction is defined to be north, so that the edges of $G$ are directed from south to north. The two poles of $G$ are interpreted as a south pole and a north pole, and the poles of the dual map are interpreted as a west pole and an east pole. The convention in \cite{kmsw-bipolar} is more natural in the context of planar maps, but we have chosen the above convention to be consistent with the continuum picture, where we interpret the two trees describing the left and right outer boundaries of the primal space-filling curve as west-going and east-going, respectively, rather than northwest-going and southeast-going. 
	\label{rmk:compass} 
\end{remark}

\subsubsection{Bipolar-oriented map and lattice walk}
\label{sec:bipolar-bijection} 

Suppose $(G,\cO)$ is a bipolar-oriented planar map embedded in the manner shown on the left in Figure~\ref{fig:intro-bipolarmap}, i.e.\ every edge is directed in upward direction. We will explain how to construct a pair of trees on $G$, which are illustrated in the middle part of Figure~\ref{fig:intro-bipolarmap}. Given any vertex $v$ in $G$ other than the southeast pole, let $\{e_1^v , \dots , e_{p^v}^v\}$ be the set of edges adjacent to $v$ that are oriented toward $v$. We assume that $e_1^v,\dots , e_{p^v}^v$ are ordered counterclockwise, such that if $v$ is the northwest pole then $e_1^v$ and $e_{p^v}^v$ lie on the boundary of the outer face, while if $v$ is not the northwest pole then the next edge in the counterclockwise direction after $e_{p^v}^v$ is an out-going edge. Now fix a small $\ep >0$, which we assume to be smaller than the minimal distance between any pair of vertices. For $j\in [2,p^v ]_{\BB Z}$, we shrink the edge $e_j^v$  by $\eps$, i.e.\ we replace it by the subpath of the edge from $u_j^v$ to the first time it hits the circle of radius $\eps$ centered at $u$,
where $u_j^v$ is the other endpoint of $e_j^v$. Since $\cO$ is bipolar, in every cycle at least one edge will be shrunk in this manner, so the resulting graph has no cycles. On the other hand, every edge is still connected to the southeast pole through edges of type $e_1^v$. Therefore this operation produces a plane tree rooted at the southeast pole that has the same number of edges as $G$. We call this tree the \emph{east-going tree}, denoted by $\cT^E$, and note that it is shown in red in Figure~\ref{fig:intro-bipolarmap}. By performing the shrinking operation on $(G,\ol\cO)$, we obtain a plane tree $\cT^W$ rooted at the northwest pole, which we call the \emph{west-going tree}, and which is shown in blue in Figure~\ref{fig:intro-bipolarmap}.

For each edge $e$ of $G$, there is a unique  path in $\mcl T^E$ (resp. $\mcl T^W$) from $e$ to the southeast pole (resp. the northwest pole). We refer to this branch as the \emph{east-going flow line} (resp. \emph{west-going flow line}) started from $e$. 

Suppose now $\ell , r  \in \BB N_0$ with $\ell + r > 0$ and that $G$ has $\ell  $ edges on its southwest boundary, $r $ edges on its northeast boundary 
with $r + \ell > 0$, and $n \in \BB N$ total edges. 
There is a unique path $\lambda'$ on $G$ (a map from $[0,n-1]_{\BB Z}$ to the set of edges of $G$) 
which starts at the leftmost outgoing edge from the southeast pole, ends at the rightmost incoming edge to the northwest pole, and snakes in between the two trees $\cT^E$ and $\cT^W$.
This path hits every edge of $G$ exactly once, and jumps across each face of $G$ exactly once. It is shown in green in Figure~\ref{fig:intro-bipolarmap}.

For $i\in [0,n-1]_{\BB Z}$ let $\cL_i$ (resp. $\cR_i$) be the distance from the top (resp.\ bottom) end point of the edge $\lambda'(i)$ to the northwest (resp.\ southeast) pole in the tree $\cT^W$ (resp.\ $\cT^E$). We define
\eqbn
\mcl Z_i := (\mcl L_i , \mcl R_i) .
\eqen 
As explained in~\cite{kmsw-bipolar}, $(\cZ_i )_{i \in [0,n-1]_{\BB Z} }$ is a lattice walk on $[0,\infty)_{\BB Z} \times [0,\infty)_{\BB Z}$ from $(\ell-1,0)$ to $(0,r-1)$. 
For $i \in [0,n-2]_{\BB Z}$, if $\lambda'(i)$ is the westmost incoming edge to some vertex, then $\mcl Z_{i+1} - \mcl Z_{i } =(-1,1)$. Otherwise, there must be a face $f$ containing $\lambda'(i)$ such that the terminal endpoint of $\lambda'(i)$ is the westmost vertex of $f$, and the edge $\lambda'(i)$ is ``shrunk'' in the construction of $\cT^E$. In this case the next step corresponds to $\lambda'$ traversing $f$. If $f$ has $\ell_f$ edges on the southwest side and  $r_f$ edges on the northeast side, then $\cZ_{i+1} - \cZ_i =(\ell_f-1,1-r_f)$.  
 
As shown in~\cite[Theorem 2]{kmsw-bipolar}, the map $(G , \cO) \mapsto \mcl Z$ is a bijection. We omit the details of the inverse bijection and refer to~\cite[Section 2]{kmsw-bipolar}, where $(G,\cO)$ is dynamically recovered from $\cZ$ by the so-called sewing procedure. Consequently, one can sample a random bipolar-oriented map $(G, \cO)$ with $\ell$ southwest boundary edges, $r$ northeast boundary edges, and $n$ total edges by first sampling a random walk $\cZ$ with certain allowed step sizes which starts $(\ell-1,0)$, at ends at $(0,r-1)$, and stays in the closed first quadrant. We will be particularly interested in the case of a uniformly random bipolar-oriented triangulation. In this case, the corresponding walk can be sampled as follows. Let $\mcl Z'$ be a simple random walk started from $(\ell-1,0)$ with $n$ total steps which are iid samples from the uniform distribution on $\{(1,0) , (0,-1), (-1,1)\}$. Then the law of $\mcl Z$ is the same as the conditional law of $\mcl Z'$ given that it stays in the closed first quadrant and ends at $(0,r-1)$.

\begin{remark}
	Note that the coordinates of the random walk are interchanged as compared to the convention applied in \cite{kmsw-bipolar}, i.e., that paper considers a random walk from $(0,\ell-1)$ to $(r-1,0)$ with iid steps chosen from $\{(0,1),(-1,0),(1,-1)\}$. We choose the above convention to be consistent with our continuum convention, where the first (resp.\ second) coordinate of the walk corresponds to the length of the left (resp.\ right) frontier of the space-filling curve. 
\end{remark}

\subsubsection{The uniform infinite bipolar-oriented triangulation}
\label{sec:bipolar-infinite}

The concept of a uniform infinite bipolar-oriented triangulation is best understood by considering a local limit of finite bipolar-oriented triangulations. Given planar maps $G$ and $G'$ with oriented root edges $e$ and $e'$, respectively, the \emph{Benjamini-Schramm distance}~\cite{benjamini-schramm-topology} from $(G,e)$ to $(G',e')$ is defined by 
\[
d_{\mathrm{loc}} = \inf\{2^{-r}:r\in \N, B_r(e;G)\simeq B_r(e';G')  \} ,
\]
where $B_r(e;G)$ is the ball of radius $r$ centered at the terminal endpoint of $e$ in the graph distance of $G$, and $B_r(e;G)\simeq B_r(e';G')$ means that the two balls are isomorphic. Then $d_{\op{loc}}$ defines a metric on the space of finite rooted planar maps. Let $\mathfrak{M}$ be the Cauchy completion of the space of finite rooted planar maps under this metric. Elements in $\mathfrak M$ which are not finite rooted maps are called infinite rooted planar maps. For more background on infinite rooted planar maps, we refer to the foundational paper~\cite{benjamini-schramm-topology}. 

For $n\in\BB N$ let $(G_n,\cO_n)$ be sampled uniformly from the set of all bipolar-oriented triangulations for which $G_n$ has 1 northeast boundary edge, 2 southwest boundary edges, and $n$ total edges. Let $\cZ^n$ be the lattice walk obtained when applying the bijection described in Section~\ref{sec:bipolar-bijection} to $(G_n,\cO_n)$. Then $\cZ^n$ is uniform among all lattice walks in $[0,\infty)^2_{\BB Z}$ of duration $n$ from from $(1,0)$ to $(0,0)$, with steps in $\{(1,0),(0,-1),(-1,1)\}$. According to \cite[Corollary 7]{kmsw-bipolar}, the number of such walks is  $c(1+o_n(1))n^{-4}\cdot 3^{3n}$. 

{It was shown in~\cite[Section 3.3]{ghs-map-dist} that}  if we uniformly choose an edge $e_n$ from $G_n$, whose direction is inherited from $\cO_n$, then $(G_n,e_n)$ weakly converges to a probability measure on $\mathfrak M$. 
The limiting object $(G,\cO,\BB e)$ is called the \emph{uniform infinite bipolar-oriented triangulation}. We note that $\cO$ is almost surely an acyclic orientation with no sinks and sources, since the two poles are at $\infty$. The orientation $\cO$ gives rise to a pair of trees on $G$, in the same manner as in the finite-volume case (see Figure~\ref{fig:intro-bipolarmap}). We define the discrete east-going and west-going flow lines of $(G , \cO)$ started from any edge $e$ of $G$ to be the branches of these two trees. These flow lines are paths from $e$ to $\infty$. There is a bi-infinite space-filling exploration path $\lambda'$ (a map from $\BB Z$ to the edge set of $G$) which snakes between the two trees, in direct analogy to the finite-volume case. The path $\lambda'$ satisfies $\lambda'(0) = \BB e$ and its left and right outer boundaries at the time it hits any edge $e$ of $G$ are given by the east-going and west-going flow lines started from $e$. 

By clockwise rotation of $90^\circ$, we obtain a dual bipolar-oriented map $(\wt G,\wt \cO ,\wt{\BB e})$ which is the Benjamini-Schramm limit of the finite-volume dual rooted bipolar-oriented maps $(\wt G_n,\wt\cO_n , \wt e_n)$. Here $\wt{\BB e}$ (resp. $\wt e_n$) is the edge of the dual map which crosses $\BB e$ (resp. $e_n$). The bipolar-oriented map $(\wt G,\wt \cO,\wt{\BB e})$ is associated with a north-going and south-going tree (which give rise to north-going and south-going flow lines and an exploration path $\wt\lambda'$) as well as a bi-infinite lattice walk $\wt\cZ = (\wt\cL , \wt\cR)$. Although $\wt\cZ$ is not a simple random walk, in Section~\ref{section:tight} we will show that $\wt\cZ$ converges to a Brownian motion in the scaling limit. The remainder of the paper will be devoted to proving that $(\cZ , \wt\cZ)$ jointly converges after rescaling.

\subsection{Background on SLE and LQG}
\label{sec:lqg-background}

In this subsection we provides some background on the theory of SLE and LQG, which is needed to describe the limiting objects in our main results.

\subsubsection{Liouville quantum gravity and the quantum cone}

Let $\gamma \in (0,2)$. A \emph{Lioville quantum gravity (LQG) surface} \cite{shef-kpz,shef-zipper,wedges} is an equivalence class of pairs $(D,h)$, where $D$ is a planar domain and $h$ is an instance of some variant of the Gaussian free field (GFF) on $D$. See~\cite{shef-gff,ss-contour,ig1,wedges,ig4} for more on the GFF. We say that two such pairs $(D,h)$ and $(\wt D , \wt h)$ are equivalent if there is a conformal map $\varphi:\wt D\to D$ such that
\eqb
\wt h = h\circ\varphi + Q\log|\varphi'| \quad \op{where} \quad Q=\gamma/2+2/\gamma .
\label{eqn:background1}
\eqe
One can also consider quantum surfaces with $k\in\BB N$ marked points $z_1,\dots , z_k \in D\cup \bdy D$. In this case we say that two such quantum surfaces $(D,h,z_1,...,z_k)$ and $(\wt D,\wt h,\wt z_1,...,\wt z_2)$ are equivalent if the condition $\varphi(\wt z_i)=z_i$, $\forall i=1,...,k$, is satisfied in addition to~\eqref{eqn:background1}.

It is proven in \cite{shef-kpz} that given an instance $h$ of a GFF one can a.s.\ define an area measure $\mu_h$ as the weak limit of $\ep^{\gamma^2/2}e^{\gamma h_\ep(z)}dz$ as $\ep\rta 0$, where $dz$ is Lebesgue measure on $D$, and $h_\ep(z)$ is the mean value of $h$ on the circle $\partial B_\ep(z)$  (as defined in \cite{shef-kpz}). The measure $\mu_h$ should be thought of as the Riemannian volume form associated with the metric tensor $e^{\gamma h} \, dx\otimes dy$. One may similarly define a length measure $\nu_h$ on certain curves in $D\cup\bdy D$ including $\partial D$ and $\op{SLE}_\kappa$ curves for $\kappa = \gamma^2$~\cite{shef-zipper, ps-elementary}. See \cite{rhodes-vargas-review} for an alternative, more general approach to defining $\mu_h$ and $\nu_h$. 
If $h$ and $\wt h$ are related as in~\eqref{eqn:background1}, then $\varphi$ pushes forward $\mu_{\wt h}$ to $\mu_h$ and $\nu_{\wt h}$ to $\nu_h$ (in fact this holds a.s.\ for all choices of $\varphi$ simultaneously~\cite{shef-wang-lqg-coord}), so these measures are well-defined functional of the quantum surface. 
 
In this paper we will be especially interested in a particular type of quantum surface called the $\alpha$-quantum cone for $\alpha<Q$. This is a doubly marked quantum surface $(\BB C ,h , 0, \infty)$ which has infinite total mass but finite mass in every neighborhood of 0. The quantum cone has a special interpretation for $\alpha=\gamma$, in which case it is the limiting law on quantum surfaces obtained by zooming in near a quantum typical point of a $\gamma$-LQG surface; see~\cite[Proposition 4.13(ii)]{wedges}.

\subsubsection{Space-filling SLE and imaginary geometry}
\label{sec:sle-background}

The Schramm-Loewner evolution \cite{schramm0} is a conformally invariant family of random fractal curves parameterized by $\kappa>0$, which has three well-known phases. For $\kappa\in(0,4]$ the curve is simple, for $\kappa'\in(4,8)$ the curve hits itself and creates ``bubbles'' but its trace has Lebesgue measure zero, and for $\kappa'\geq 8$ the curve fills space. In \cite[Sections 1.2.3 and 4.3]{ig4} the authors construct a space-filling version of SLE$_{\kappa'}$ for $\kappa'>4$, which agrees with ordinary SLE$_{\kappa'}$ for $\kappa'\geq 8$. Roughly speaking, space-filling SLE$_{\kappa'}$ for $\kappa'\in(4,8)$ is the same as ordinary SLE$_{\kappa'}$, except that the curve enters each ``bubble'' right after it is disconnected from the target point and fills it with a continuous space-filling loop before it continues the exploration towards the target point. 

The authors of \cite{ig4} construct space-filling SLE$_{\kappa'}$ by means of flow lines of a Gaussian free field \cite{ig1,ig2,ig3,ig4}. We will only describe the construction of whole-plane space-filling SLE$_{\kappa'}$ from $\infty$ to $\infty$, but only minor modifications are needed to describe space-filling SLE$_{\kappa'}$ in general domains. Let $\kappa = 16/\kappa' \in (0,4)$, 
\eqb \label{eqn:ig-chi}
\chi = 2/\sqrt{\kappa} - \sqrt{\kappa}/2,
\eqe
$\theta\in[0,2\pi)$, and let $\wh h$ be an instance of a whole-plane Gaussian free field modulo a global additive multiple of $2\pi\chi$. A flow line $\eta$ of $\wh h$ is an $\SLE_{\kappa}$-like curve for $\kappa\in(0,4)$ which is interpreted as a flow line of the vector field $e^{i(\wh h/\chi+\theta)}$, where $\chi$ and $\kappa$ are related as in \ref{eqn:ig-chi}.

Flow lines of $\wh h$ of the same angle $\theta$ started at different points in $\BB Q^2$ merge into each other upon intersecting and form a tree \cite[Theorem 1.9]{ig4}. The \emph{space-filling SLE$_{\kappa'}$ counterflow line of $\wh h$ of angle $\theta$} is the Peano curve of this tree, i.e., it is the curve which visits the points of $\BB C$ in chronological order, where a point $x\in\BB C$ is visited before a point $y\in\BB C$ if the flow line of angle $\theta$ from $x$ merges into the flow line of angle $\theta$ from $y$ on the left side. This curve is proven to be continuous when parametrized by Lebesgue measure in~\cite[Section 4.3]{ig4}.

For any point $z\in\BB C$ it holds a.s.\ that the left (resp.\ right), outer boundary of $\eta'$ when it first hits $z$ is given by the flow line of $\wh h$ started from $z$ of angle $\theta$ (resp.\ $\theta-\pi$). We refer to $\theta=0$ (resp.\ $\theta=\frac{\pi}{2},\theta=\pi,\theta=\frac{3\pi}{2}$) as north (resp.\ west, south, east) direction. The space-filling SLE$_{\kappa'}$ generated from flow lines of angle $\theta=\frac{\pi}{2}$ in the above construction has left and right boundaries given by west-going and east-going, respectively, flow lines, and travels in north direction. The space-filling SLE$_{\kappa'}$ generated from flow lines of angle $\theta=\pi$ has left and right boundaries given by south-going and north-going, respectively, flow lines, and travels in west direction, i.e., in a direction orthogonal to the north-going SLE$_{\kappa'}$.

Space-filling SLE$_{\kappa'}$ has very a different nature for $\kappa'\in(4,8)$ and $\kappa'\geq 8$. For $t\in\R$ and $\kappa'\geq 8$ the domain $\eta'((-\infty,t])$ is a.s.\ homeomorphic to the upper half-plane, while for $\kappa'\in(4,8)$ the domain $\eta'((-\infty,t])$ is an infinite chain of bounded simply connected domains. The reason for this difference is that the flow lines of angle $\theta$ and $\pi-\theta$ from any given point $z\in \BB C$ a.s.\ hit each other (resp.\ intersect only at $z$) for $\kappa'\in(4,8)$ (resp.\ $\kappa'\geq 8$) \cite[Theorem 1.7]{ig4}. If $(\eta'(t))_{t\in\R}$ is a space-filling SLE$_{\kappa'}$, $\tau_z$ is the first time $\eta'$ hits $z\in\BB C$, and we condition on $\eta'((-\infty,\tau_z])$, the conditional law of $\eta'|_{[\tau_z,\infty)}$ for $\kappa'\geq 8$ is an ordinary chordal SLE$_{\kappa'}$ in $\C\setminus \eta'((-\infty,\tau_z])$ from $\eta'(\tau_z)$ to $\infty$. If $\kappa'\in(4,8)$ the conditional law of $\eta'|_{[\tau_z,\infty)}$ is that of a concatenation of independent chordal space-filling SLE$_{\kappa'}$ curves in each complementary connected component of $\eta'((-\infty,t])$ \cite[Footnote 4]{wedges}.

It is proven in \cite[Theorem 1.10]{ig4} that $\wh h$ (viewed as a distribution modulo a global additive constant in $2\pi\chi\Z$) and the space-filling SLE$_{\kappa'}$ flow line $\eta'$ of $\wh h$ with a given angle $\theta$ a.s.\ determine each other. Hence, an instance of the north-going space-filling SLE$_{\kappa'}$ $\eta'$ a.s.\ determines the imaginary geometry field $\wh h$ and thereby also the west-going space-filling SLE$_{\kappa'}$ curve $\wt\eta'$. 

\subsubsection{The peanosphere}
\label{sec:peanosphere-background}

Let $\gamma\in(0,2)$, and consider a $\gamma$-quantum cone $(\C,h,0,\infty)$ decorated with an independent whole-plane space-filling SLE$_{\kappa'}$ curve $\eta'$ from $\infty$ to $\infty$, $\kappa'=16/\gamma^2$. In \cite[Theorem 1.9]{wedges} it is proven that $(h,\eta')$ can be encoded in terms of a two-dimensional Brownian motion $Z=(L_t,R_t)_{t\in\R}$ with the following variances and covariances:
\eqb
\op{Var}(L_t ) =  \alpha |t|, \quad 
\op{Var}(R_t) =  \alpha |t|, \quad 
\op{Cov}(L_t, R_t ) = -\alpha \cos\theta |t|,\quad 
\theta=\frac{4\pi}{\kappa'}, 
\label{eqn:correl} 
\eqe
where $\alpha$ is a deterministic constant depending only on $\kappa'$. {(The constant $\alpha$ was recently explicitly computed in~\cite{ars-fzz}.)}
The Brownian motion $Z$ is constructed from $(h,\eta')$ in the following manner. Suppose we parametrize $\eta'$ according to $\gamma$-quantum area with respect to $h$, i.e., for any $s,t\in\R$ satisfying $s<t$ we have $\mu_h(\eta'([s,t]))=t-s$, and we normalize to that $\eta'(0) = 0$. 
Then $Z$ is the left/right boundary length process of $\eta'$ with respect to $h$. That is, for $t \in \BB R$ the process $L_t$ (resp.\ $R_t$) gives the net change in quantum boundary length of the left (resp.\ right) outer boundary of $\eta'((-\infty,t])$ relative to time 0.

By \cite[Theorem 1.11]{wedges} the pair $(L,R)$ almost surely determines the pair $(h,\eta')$ (up to a rigid rotation of the complex plane about the origin), i.e., there is a measurable a.s.\ defined map which associates a realization of the pair $(L,R)$ with a $\gamma$-quantum cone decorated with an independent space-filling $\SLE_{\kappa'}$.

The relationship between $(h,\eta')$ is referred to as the \emph{peanosphere construction} since it shows that $(h,\eta')$ is an embedding into $\BB C$ of a random curve-decorated topological measure space constructed from $Z$ called a \emph{peanosphere}. 
See Figure~\ref{fig-peano} for an illustration of this construction.

\begin{figure}[ht!]
	\begin{center}
		\includegraphics[scale=0.92]{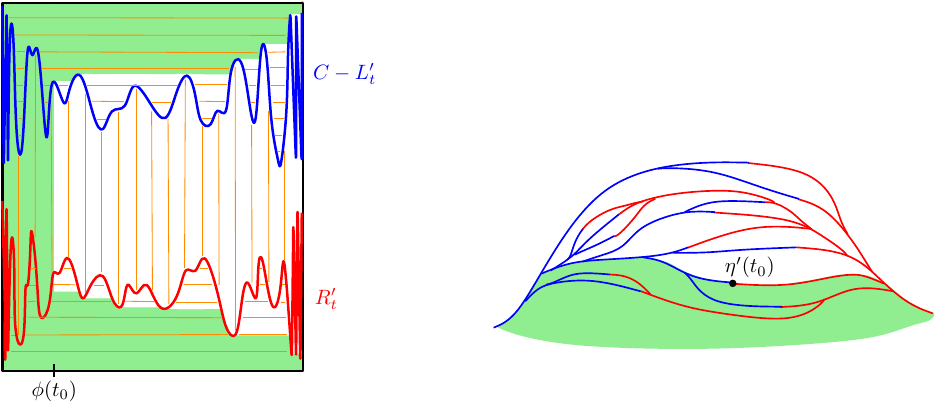}
	\end{center}
	\caption{\label{fig-peano} 
		An illustration of the peanosphere construction of \cite{wedges}. Let $Z=(L_t,R_t)_{t\in\R}$ be a correlated two-dimensional Brownian motion. Let $\phi:\R\to(0,1)$ be an increasing, continuous and bijective function, and for any $t\in(0,1)$ define $ L_t'  :=\phi(L_{\phi^{-1}(t)})$ and $ R_t' :=\phi(L_{\phi^{-1}(t)})$. The left figure shows $ R'$ and $C- L'$, where $C$ is a constant chosen so large that the two graphs do not intersect.
		We draw horizontal lines above the graph of $C- L'$ and below the graph of $ R'$, in addition to vertical lines between the two graphs, and then we identify points which lie on the same horizontal or vertical line segment.  We also identify all points on the boundary of the square.  As explained in~\cite[Proposition 1.7]{wedges}, it follows from  Moore's theorem \cite{moo25} the resulting object is a topological sphere. The sphere is decorated with a space-filling path $\eta'$ 
		where $\eta'(t)$ for $t \in \R$ is the equivalence class of $(\phi(t),L'_t)$. The pushforward of Lebesgue measure on $\R$ induces an area measure $\mu$ on the sphere. The resulting structure, i.e. the topological sphere with the curve $\eta'$ and the measure $\mu$, is called a \emph{peanosphere}. It is shown in \cite{wedges} that the peanosphere has a canonical embedding into $\C$ where the pushforward of $\mu$ encodes a LQG surface known as the $\gamma$-quantum cone and $\eta'$ is an independent space-filling $\SLE_{\kappa'}$, $\kappa'=16/\gamma^2$. The right part of the figure shows a subset of the SLE-decorated LQG surface, where the green region corresponds to points that are visited by $\eta'$ before some time $t_0$. The two trees are embeddings of the trees with contour functions $L$ and $R$, respectively, such that $L$ (resp.\ $R$) encode the quantum boundary length of the left (resp.\ right) frontier of $\eta'$. If $Z=(L_t,R_t)_{t \in [0,1]}$ was a Brownian excursion we would obtain a finite volume LQG surface decorated with an independent space-filling SLE by a similar procedure \cite{wedges,sphere-constructions}.
	}
\end{figure}

We now describe the limiting object in our main theorem. 
Let $\wh h$ be the instance of the whole-plane Gaussian free field (modulo a global additive multiple of $2\pi\chi$, with $\chi$ as in \eqref{eqn:ig-chi}) such that $\eta'$ is the north-going space-filling SLE$_{\kappa'}$ flow line of $\wh h$ (recall Section~\ref{sec:sle-background}). Let $\wt\eta'$ be the west-going space-filling SLE$_{\kappa'}$ flow line of $\wh h$, parameterized according to $\gamma$-LQG area with respect to $h$ and satisfying $\wt\eta'(0)=0$. 
Then $\wt\eta'$ (viewed modulo parametrization) is a.s.\ determined by $\eta'$ (viewed modulo parametrization), so is independent from $h$. Furthermore, $\wt\eta' \eqD \eta'$ so if we let 
$\wt Z = (\wt L_t , \wt R_t)_{t\in \R}$ be the left/right boundary length process for $\wt\eta'$ then $\wt Z \eqD Z$. 
This gives us a coupling of two correlated planar Brownian motions $Z$ and $\wt Z$ which a.s.\ determine each other.

\subsection{Statement of main result}
\label{sec:mainres}

Consider a uniform infinite bipolar-oriented triangulation $(G ,\cO,\BB e)$, 
and recall that it can be encoded by a random walk $\cZ=(\cL_n,\cR_n)_{n\in\Z}$ with iid 
increments in $\{(1,0),(0,-1),(-1,1)\}$. Let $(\wt G , \wt \cO, \wt{\BB e})$ 
be the dual bipolar-oriented map of $(G , \cO,\BB e)$, and 
let $\wt{\cZ}:=(\wt {\cL}_n,\wt {\cR}_n)_{n\in\Z}$ 
be the random walk encoding $(\wt G , \wt \cO, \wt{\BB e})$. 
Extend $\mcl L$, $\mcl R$, $\wt\cL$, and $\wt\cR$ to $\BB R$ by linear interpolation. For $m\in\BB N$ and $t\in\BB R$ define
\eqb \label{eqn:rescaled-walk}
\begin{split}
	L^m_t := (3\alpha /2)^{1/2}  m^{-1/2} \cL_{t m} ,\quad R^m_t :=  (3\alpha/2)^{1/2}  m^{-1/2} \cR_{t m} , \quad   Z^m_t:= (L^m_t , R^m_t), \\
	\wt L^m_t := (3\alpha/2)^{1/2} m^{-1/2} \wt \cL_{t m} ,\quad \wt R^m_t :=  (3\alpha/2)^{1/2}  m^{-1/2} \wt\cR_{t m} , \quad   \wt Z^m_t:= (\wt L^m_t , \wt R^m_t) ,
\end{split}
\eqe 
where $\alpha$ is as in~\eqref{eqn:correl} for $\gamma = \sqrt{4/3}$. The reason for including the factor $(3\alpha/2)^{1/2}$ is so that $Z^m$ converges in law to the Brownian motion in~\eqref{eqn:correl}. 

Let $Z=(L,R)$ be a two-dimensional correlated Brownian motion with correlation $-\frac 12$ and variance $\alpha$, where $\alpha>0$ is again as in \eqref{eqn:correl}. 
Recall from Section~\ref{sec:peanosphere-background} that $Z$ is the left/right boundary length process of a $\sqrt{4/3}$-quantum cone $(\C,h,0,\infty)$ decorated with an independent space-filling SLE$_{12}$ by the peanosphere construction. 
Let $\wt Z = (\wt L , \wt R)$ be the left/right boundary length process of the west-going space-filling SLE$_{12}$ determined by $\eta'$, as in Section~\ref{sec:peanosphere-background}.

\begin{theorem}
	Let $Z^m,\wt Z^m$ and $Z,\wt Z$ be as defined above. Then the following convergence holds in law for the topology of uniform convergence on compact sets:
	\eqbn
	(Z^m,\wt Z^m) \rtaD (Z,3\wt Z).
	\eqen
	\label{thm1}
\end{theorem}

\begin{remark} \label{remark:primal-tree} 
	In Theorem~\ref{thm1}, the second (west-going) Peano curve under consideration is defined in terms of the north-going and south-going trees on the dual bipolar-oriented map $(\wt G  , \wt \cO, \wt{\BB e})$. There is also another natural way to define a west-going Peano curve associated with a bipolar-oriented map $(G , \cO,\BB e)$. Namely, in Figure~\ref{fig:intro-bipolarmap} we define a new red tree (the \emph{primal north-going tree}) rooted at the source by cutting each edge other than the \emph{right}most edge leading upward from each vertex and define a new blue tree (the \emph{primal south-going tree}) rooted at the sink by cutting each edge except for the \emph{left}most edge leading downward from each vertex. Let $\dot\lambda'$ be the Peano curve associated with these two new trees (defined in an analogous manner to the original Peano curve $\lambda'$). Then $\dot\lambda'$ travels in the west direction, rather than the north direction. We can encode $(G , \cO,\BB e)$ by a different random walk $\dot{\mcl Z} = (\dot\cL ,\dot\cR)$ by using these two new trees instead of the original east-going and west-going trees. If we perform this construction for the uniform infinite bipolar-oriented triangulation, then the walk $\dot{\mcl Z}$ is a bi-infinite random walk with the same law as the original walk $\mcl Z$. It will be shown in Proposition~\ref{prop:twoNtrees} that if we define $\dot Z^m$ as in~\eqref{eqn:rescaled-walk}, then the scaling limit of $\dot Z^m$ is the same as the scaling limit of the re-scaled walk $\wt Z^m$ of Theorem~\ref{thm1}, up to a constant factor. In particular, one has
	\eqbn
	(Z^m,\dot Z^m) \rtaD (Z, \wt Z) ,
	\eqen
	with $(Z , \wt Z)$ as in Theorem~\ref{thm1}. This is the version of Theorem~\ref{thm1} which is conjectured in~\cite{kmsw-bipolar}. 
\end{remark}

\subsection{Outline}
\label{sec:outline}

The remainder of this article is structured as follows. 
In Section~\ref{sec:one-dual-tree}, we prove some properties of the dual bipolar-oriented map $(\wt G , \wt\cO, \wt{\BB e})$ of a uniform infinite bipolar-oriented triangulation using combinatorial techniques. In particular, we describe the discrete north-going and south-going flow lines (i.e.\ the branches of the two trees which decorate the dual map) in terms of the primal walk $\cZ$; and we prove that the dual walk $\wt\cZ$ converges in law to a correlated two-dimensional Brownian motion in the scaling limit. 

In Section~\ref{sec:continuum-decomp}, we describe the joint law of the peanosphere Brownian motion $Z$ and the set of times when the space-filling SLE$_{\kappa'}$ curve $\eta'$ crosses the $\theta$-angle flow line of the whole-plane GFF used to construct $\eta'$ in terms of a certain Poisson point process of conditioned Brownian motions. This is done for general values of $\kappa' >4$ and $\theta \in (-\pi/2,\pi/2)$, although we only need the case when $\kappa'= 12$ and $\theta =0$ for the proof of Theorem~\ref{thm1}. 
In the case when $\kappa' = 12$, this description is an exact continuum analogue of the description of the times when the north-going Peano curve on a uniform infinite bipolar-oriented triangulation crosses the north-going discrete flow line in terms of the associated walk $\mcl Z$. 

In Section~\ref{sec:flowline-conv} we prove that the joint law of $\cZ$ and the random walk $\cX$ which encodes the excursions of $\cZ$ away from a north-going discrete flow line, appropriately rescaled, converges to the joint law of their continuum counterparts. In Section~\ref{sec:peano-conv}, we conclude the proof of Theorem~\ref{thm1}.

\section{Properties of the dual map}
\label{sec:one-dual-tree}

\subsection{Discrete decomposition of contour functions}
\label{sec:excursions}

Let $\Nz:=\N\cup\{0\}$. In this section we define random one-dimensional paths $\cX = (\cX_n)_{n\in\Nz}$ and $\cl = (\cl_n)_{n\in\Nz}$ which will be discrete processes adapted to the filtration generated by the random walk $\mcl Z = (\cL,\cR)$ which encodes a uniform infinite bipolar-oriented triangulation $(G , \cO , \BB e)$. Let $\lambda'$ be the space-filling exploration path of $(G , \cO,\BB e)$ and let $\lambda^N$ be the north-going discrete flow line started from $\lambda'(0)$ (i.e.\ the infinite branch of the dual north-going tree which begins at the edge of the dual map $\wt G$ which crosses $\lambda'(0)$); recall Section~\ref{sec:bipolar-bijection}. 

For each time $n\in\Nz$ the path $\lambda'$ is either in an \emph{east} excursion or in a \emph{west} excursion. We are in a west excursion at time $n$ if both end-points of the edge $\lambda'(n)$ edge are on the west side of $\lambda^N$, and we are in an east excursion at time $n$ if at least one end-point of the corresponding edge is on the east side of $\lambda^N$. Note that this description is not symmetric in east excursions and west excursions. We have chosen this definition of east and west excursions since it makes the description of the times we are in an east (resp.\ west) excursion particularly simple in terms of $(\cL,\cR)$. We define $N_k^W$ (resp.\ $N_k^E$), $k\in\N$, to be the times at which a west (resp.\ east) excursion starts. More precisely, we define $N_k^W$ (resp.\ $N_k^E$) to be the first time after time $N_{k}^E$ (resp.\ $N_{k-1}^W$) for which $\lambda'$ is at an edge strictly on the west side of $\lambda^N$ (resp.\ at an edge which crosses $\lambda^N$). See Figure \ref{fig1}. For any $n\in\BB N$ let $\cl_n$ be the number of edges on the north-going flow line $\lambda^N$ which are crossed by $\lambda'$ during the time interval $[0,n]_{\BB Z}$. The random walk $\cX = (\cX_n)_{n\in\Nz}$ is defined by $\cX_0=0$ and
\eqb
\cX_n-\cX_{n-1}=
\begin{cases}
	\cL_{n-1} - \cL_{n} &\mathrm{if}\,\, \exists k\in\N \mathrm{\,s.t.\,\,} N_k^E     < n\leq N_k^W,\\
	\cR_{n} - \cR_{n-1} &\mathrm{if}\,\, \exists k\in\N \mathrm{\,s.t.\,\,} N_{k-1}^W < n\leq N_k^E.
\end{cases}
\label{eqn:Xm}
\eqe
The following lemma implies that $\cX$ is a random walk with iid increments, since $N_k^W,N_k^E$, $k\in\N$, are stopping times for $\cX$. In fact, $\mcl X$ has the same marginal law as $\mcl L$ and $\mcl R$.
\begin{lem}
	The times $N_k^W,N_k^E$, $k\in\N$, are stopping times for $\cX$ (hence also for $\cZ$), and are given by
	\eqb
	\begin{split}
		N_1^E &= 0,\\ 
		N_k^W &= \inf\{ n\geq N_k^E     \,:\,\cL_n=\cL_{N_k^E}-1 \}
		= \inf\{ n\geq N_k^E     \,:\,\cX_n=1 \},\\
		N_k^E &= \inf \{ n\geq N_{k-1}^W\,:\,\cR_n= \cR_{N_{k-1}^W}-1\}
		= \inf \{ n\geq N_{k-1}^W\,:\,\cX_n=0\}.
	\end{split}
	\label{descrdec-eq3}
	\eqe
	Moreover, the edge $\lambda'(n)$ crosses the north-going flow line $\lambda^N$ if and only if $\cX_n=0$. In particular, $\cl_n= \#\{0\le k\le n: \cX_k=0 \}$, which is the local time at 0 for $\cX$. 
\end{lem}
\label{prop:discrete-excursions}

\begin{proof}
	Consider an edge of the bipolar map which is crossed by $\lambda^N$. We say that a time $n\in\BB N$ is of type (a) (resp.\ type (b)) if the edge $\lambda'(n)$ crosses the north-going flow line $\lambda^N$ and the edge $\lambda'(n+1)$ is on the west side of $\lambda^N$ (resp.\ either crosses $\lambda^N$ or is on the east side of $\lambda^N$). See Figure \ref{fig2}. 
	
	For any $k\in\N$ it follows by the definition of $N_k^W$ that $N^W_k-1$ is a time of type (a). By definition of $N_{k+1}^E$ and $\lambda'$, $\lambda'(N_{k+1}^E)$ is the edge which follows the edge $N_k^W-1$ along the path $\lambda^N$. During the time interval $[N_k^W,N_{k+1}^E)$ the space-filling path $\lambda'$ is traversing the subtree of the east-going tree which is rooted at the edge $\lambda'(N_k^W-1)$. In particular, $\cR_{N_k^W}=\cR_{N_k^W-1}+1$. Since the easternmost vertex is identical for the edges $\lambda'(N_{k}^W-1)$ and $\lambda'(N_{k+1}^E)$ it holds that $\cR_{N_{k}^W-1}=\cR_{N_{k+1}^E}$. See Figure \ref{fig2}(a). It follows that $N^E_{k+1}$ is given by \eqref{descrdec-eq3}, i.e., $N^E_{k+1}$ is the first time after $N_k^W$ at which the height in the east-going tree has decreased by one, or equivalently the first time $\cX$ hits 0. Furthermore, no edges visited by $\lambda'$ during $[N^W_k,N^E_{k+1})$ are crossed by $\lambda^N$, which is consistent with the formula for $\cl$ in the statement of the lemma since $\cX>0$ throughout the interval.
	
	By definition of $N_k^W$ and the space-filling path it holds that given $N_k^E$ for some $k\in\N$ we have $N_k^W=n+1$, where $n$ is the first time after $N_k^E$ of type (a). If $n'\in[N_k^E,N_k^W)$ is a time of type (b), then the next edge visited by $\lambda^N$ after $\lambda'(n')$ has the same height in the west-going tree as $\lambda'(n')$, since the westernmost vertex of these two edges is the same. See Figure \ref{fig2}(b). Hence it holds that $\cL_{n'}=\cL_{N_k^E}$ and $\cX_{n'}=\cX_{N_k^E}=0$ for all $n'<[N_k^E,N_k^W)$ for which $\lambda'(n)$ crosses $\lambda^N$. In particular, $\cL_{N_k^W-1}=\cL_{N_k^E}$ and $\cX_{N_k^W-1}=\cX_{N_k^E}=0$, since $\lambda'(N_k^W-1)$ crosses $\lambda^N$ by definition of $N_k^W$. Furthermore, any $n'\in [N_k^E,N_k^W)$ for which $\lambda'(n')$ does not cross $\lambda^N$ satisfies $\cL_{n'}>\cL_{N_k^E}$, since $\lambda'(n')$ is an edge in a subtree of the east-going tree which is rooted at an edge which crosses $\lambda^N$, hence $\cX_{n'}<\cX_{N_k^E}=0$. These observations imply our formula for the local time $\cl$. Since $ N_k^W-1 $ is a time of type (a), it holds that $\cL_{N_k^W}=\cL_{N_k^W-1}-1$ and $\cX_{N_k^W}=\cX_{N_k^W-1}+1=1$. Combined with the above observations this implies that $N^W_k$ is given by \eqref{descrdec-eq3}.
\end{proof}

\begin{figure}[ht!]
	\begin{center}
		\includegraphics[scale=1.1]{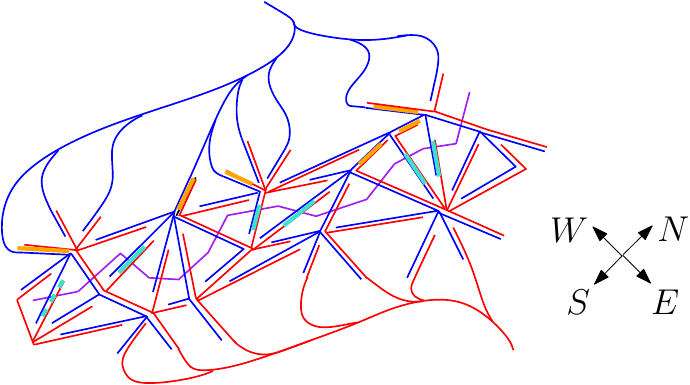}
	\end{center}
	\caption{The purple curve shows the north-going flow line $\lambda^N$. The process $\cL$ (resp.\ $\cR$) gives the height in the blue west-going (resp.\ red east-going) tree. The light blue (resp. orange) thick lines represent the stopping times $N_k^E$ (resp. $N_k^W$), and the time $N_1^E=0$ is shown in dashed light blue.} 
	\label{fig1}
\end{figure}

\begin{figure}[ht!]
	\begin{center}
		\includegraphics[scale=1.3]{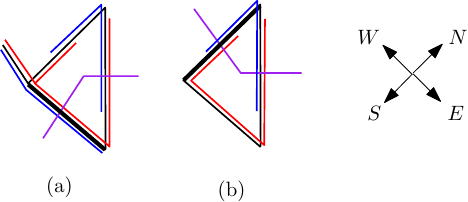}
	\end{center}
	\caption{The thick black edges correspond to times of type (a) and (b), respectively, as defined in the proof of Lemma \ref{prop:discrete-excursions}. 
	The north-going flow line $\lambda^N$ is shown in purple, and the red (resp.\ blue) edges are part of the east-going (resp.\ west-going) tree. } 
	\label{fig2}
\end{figure}

\begin{lem}
	Fix $\ep>0$.	For any $M \in \BB N$, 
	\eqbn
	\BB P\left[\sup_{0\leq k\leq M} |k-\frac{1}{3}\frk l_{N_k^W}|> M^{1/2+\ep}\right] = o_M(M^{-p}) ,\quad \forall p > 0 ,
	\eqen
	at a rate depending only on $\ep$. 
	\label{prop:primal-dual-length-same}
\end{lem}

\begin{proof}
	By Lemma \ref{prop:discrete-excursions} the random variables $Y_k:=\frk l_{N^W_{k+1}}-\frk l_{N^W_{k}}$ are iid geometric random variables with a success probability $1/3$, i.e., for any $p\in\N$ it holds with probability $(2/3)^{p-1}(1/3)$ that $Y_k=p$. In particular, $\E[Y_k]=3$ for each $k\in\BB N$. The lemma now follows by a union bound and elementary estimates for sums of iid geometric random variables (see, e.g., \cite[Theorems 2.1 and 3.1]{janson-tail}).
\end{proof}

\subsection{Equivalence of west-going Peano curve in primal and dual map}
\label{section:tight}

There are two natural west-going Peano curves associated with an instance of an infinite volume bipolar-oriented map: one Peano curve $\wt\lambda'$ tracing the interface between the north-going and south-going tree in the dual map $(\wt G,\wt{\mcl O},\wt{\BB e} )$, and one Peano curve $\dot\lambda'$ tracing the interface between the north-going and south-going tree in the primal map $(G,\mcl O,\BB e)$. See Remark \ref{remark:primal-tree}. In this section we will prove that the two discrete random walks $\wt{\mcl Z}$ and $\dot{\mcl Z}$ encoding each of these Peano curves converge jointly to the same scaling limit (up to a constant). Once we have proved our main result Theorem~\ref{thm1}, this will imply a version of Theorem \ref{thm1} with the random walk $\dot{\mcl Z}$ instead of $\wt{\mcl Z}$. The random walk $\wt{\mcl Z}=(\wt{\mcl L},\wt{\mcl R})$ was defined in Section \ref{sec:mainres}, and encodes the west-going Peano curve $\wt\lambda':\BB Z\to E(\wt G)$ in the dual map. The random walk $\dot{\mcl Z}=(\dot{\mcl L},\dot{\mcl R})$ is the corresponding random walk for $\dot{\lambda}'$, i.e., $\dot{\mcl L}$ (resp.\ $\dot{\mcl R}$) encodes the height in the south-going (resp.\ north-going) tree in the primal map relative to time 0. Assume $\dot\lambda'$ is normalized such that $\wt\lambda'(0)=\wt{\BB e}$ and $\dot\lambda'(0)=\BB e$. The random walk $\dot{\mcl Z}$ has iid increments in $\{(-1,1),(1,0),(0,-1)\}$ by symmetry with the north-going Peano curve $\lambda'$ in $(G,\mcl O,\BB e)$. Extend $\dot{\mcl Z}$ to a function on $\BB R$ by linear interpolation, and for each $m\in\BB Z$ and $t\in\BB R$ define the renormalized version of $\dot{\mcl Z}$ by
\eqb \label{eqn:rescaled-walk'}
\begin{split}
	\dot L^m_t := (3\alpha /2)^{1/2}  m^{-1/2} \dot{\cL}_{t m} ,\quad \dot R^m_t :=  (3\alpha/2)^{1/2}  m^{-1/2} \dot{\cR}_{t m} , \quad   \dot Z^m_t= (\dot L^m_t , \dot R^m_t),
\end{split}
\eqe 
where $\alpha$ is as in~\eqref{eqn:correl} for $\gamma = \sqrt{4/3}$. The main result of this section is the following proposition.

\begin{prop}
	The pair $(\wt Z^m,\dot Z^m)$ converges in law as $m\rta\infty$ to $(3Z,Z)$, where $Z$ is the two-dimensional correlated Brownian motion considered in Theorem \ref{thm1}.	
	\label{prop:twoNtrees}	
\end{prop}

The proof of the proposition will be based on several basic lemmas. Recall that each edge around a face in the primal map is on either the left (equivalently, southwest) side of the face or on the right (equivalently, northeast) side of the face. Therefore, given a face of the primal map, the lowest left edge and the highest right edge around the face are well-defined.  Similarly, the edges around a face in the dual map are on either the upper (equivalently, northwest) or lower (equivalently, southeast) side of the face, and the leftmost lower edge and the rightmost upper edge around the face are both well-defined.

We state the following lemma for general infinite volume bipolar oriented maps, i.e., we do not restrict ourselves to triangulations as in the remainder of the paper. Note that the lemma also holds for finite maps.
\begin{lem}
	Consider an infinite volume bipolar-oriented map $(G,\mcl O,\BB e)$ and its dual $(\wt G,\wt{\mcl O},\wt{\BB e})$, and let $\dot{\lambda}'$ and $\wt{\lambda}'$ (resp.\ $\dot{\mcl Z}$ and $\wt{\mcl Z}$) denote the corresponding west-going Peano curves (resp.\ height function of south-going and north-going trees) in the primal and dual map. Then the following holds: 
	\begin{itemize}
		\item[(i)] The path $\dot\lambda'$ (resp.\ $\wt\lambda'$) always traverses an edge following its orientation, i.e. from southeast to northwest or, equivalently, in upwards direction (resp.\ from northeast to southwest or, equivalently, from right to left). 
		\item[(ii)] The path $\dot\lambda'$ always crosses a face from top to bottom, and the last (resp.\ first) edge visited before (resp.\ after) the face is crossed is the lowest (resp.\ highest) edge on the right (resp.\ left) side of the face. The path $\wt\lambda'$ always crosses a face from left to right, and the last (resp.\ first) edge visited before (resp.\ after) the face is crossed is the leftmost (resp.\ rightmost) edge on the lower (resp.\ upper) side of the face.
		\item[(iii)] The Peano curves $\dot\lambda'$ and $\wt\lambda'$ always have the north-going (resp.\ south-going) tree on their right (resp.\ left) side.
		\item[(iv)] A step of the walk $\dot {\mcl Z}$ (resp.\ $\wt {\mcl Z}$) is equal to $(-1,1)$ iff the Peano curve does not cross a face of the primal (resp.\ dual) map in this step.
	\end{itemize}
	\label{prop:cross-direction}
\end{lem}

\begin{proof}
	The lemma is immediate from the bijection described in \cite[Sections 2.1 and 2.2]{kmsw-bipolar}.
\end{proof}

The following lemma says that the west-going Peano curves $\dot{\lambda}'$ and $\wt{\lambda}'$ visit the edges of the map in the same order. We state the result for general primal and dual bipolar-oriented maps, i.e., we do not restrict ourselves to the case when the primal map is a triangulation. Recall that we can identify  $E(G)$ and $E(\wt G)$, and therefore we may interpret $\wt\lambda'$ and $\dot\lambda'$ to have the same range.

\begin{lem}
	Consider the setting described in Lemma \ref{prop:cross-direction}. For any $n\in\BB Z$ it holds that $\dot\lambda'(n)=\wt\lambda'(n)$.
	\label{prop:primal-dual-order}
\end{lem}

\begin{proof}
	By induction and recentering of the map it is sufficient to prove that $\wt\lambda'(1)=\dot\lambda'(1)$. We consider two cases separately: (a) $Z_1\neq(-1,1)$, (b) $Z_1=(-1,1)$.
	
	First assume case (a) occurs. By Lemma \ref{prop:cross-direction}(iv) the Peano curve $\dot{\lambda}'$ crosses a face $f$ in the first step, so by Lemma \ref{prop:cross-direction}(ii) $\wt\lambda'(1)$ (resp.\ $\wt\lambda'(0)$) is the lowest (resp.\ highest) edge on the left (resp.\ right) side of $f$. By definition of the dual north-going (resp.\ south-going) tree this implies that there is a branch in this tree for which $\dot\lambda'(0)$ is the direct predecessor (resp.\ successor) of $\dot\lambda'(1)$ when following the branch towards its root. Hence, by definition of the Peano curve and Lemma \ref{prop:cross-direction}(iii), the dual Peano curve $\wt\lambda'$ will visit $\dot\lambda'(1)$ right after $\dot\lambda'(0)$, so $\wt\lambda'(1)=\dot\lambda'(1)$.
	
	Now assume case (b) occurs. By Lemma \ref{prop:cross-direction}(iv) the primal Peano curve $\dot\lambda'$ crosses a vertex $v\in V(G)$ in the first time step. The first step of our proof will be to show that the dual Peano curve $\wt\lambda'$ crosses the face in the dual graph corresponding to $v$ in this time step. Since (b) occurs there is a branch in the north-going (resp.\ south-going) primal tree for which $\dot\lambda'(0)$ is a direct predecessor (resp.\ successor) of $\dot\lambda'(1)$. By definition of the north-going (resp.\ south-going) tree this implies that the edge $e_1$ (resp.\ $e_0$) next to $\dot\lambda'(1)$ (resp.\ $\dot\lambda'(0)$) in clockwise order around $v$ has $v$ as its northernmost (resp.\ southernmost) end-point. Since $\dot\lambda'(0)$ (resp.\ $\dot\lambda'(1)$) has $v$ as its northernmost (resp.\ southernmost) end-point this implies that the vertex in $\wt G$ which is an end-point of both $\dot\lambda'(0)$ and $e_0$ (resp.\ $\dot\lambda'(1)$ and $e_1$), is the leftmost (resp.\ rightmost) vertex on the face of $\wt G$ corresponding to $v$. This implies that $\wt\lambda'$ will traverse this face right after visiting $\wt\lambda'(0)=\dot\lambda'(0)=0$. Furthermore, $\wt\lambda'(1)$ is the rightmost edge on the upper side of this face, i.e., $\wt\lambda'(1)=\dot\lambda'(1)$.
\end{proof}

The following lemma states that the north-going trees in the primal and dual map are similar, in the sense that the primal and dual branches started from a given edge trace each other closely, and that the branches started from two different edges merge approximately simultaneously for the primal and dual tree. Contrary to the lemmas above, we assume $G$ is a triangulation as in the remainder of the paper. Let $\lambda^{N}_n:\BB N_0\to E(\wt G)$ be the north-going flow line in the dual map started from $\wt\lambda'(n)$, i.e., it is defined exactly as the flow line $\lambda^N$ in Section \ref{sec:excursions}, but for the map with marked edge $\wt\lambda'(n)$ instead of $\wt{\BB e} =\wt\lambda'(0)$. Let $\dot\lambda^{N}_n:\BB N_0\to E(G)$ be the north-going flow line in the primal map started from $\dot\lambda'(n)=\wt\lambda'(n)$. 
In other words, $\dot\lambda^{N}_n$ is such that $\dot\lambda^{N}_n(0)=\dot\lambda'(n)$, 
for any $n\in\BB N$ the edges $\dot\lambda^{N}_n(k-1)$ and $\dot\lambda^{N}_n(k)$ share an end-point $v_k\in V(G)$, and for any $n,k\in\BB N$ the edge $\dot\lambda^{N}_n(k)$ has $v_k$ as its lowest end-point and it is the edge pointing in rightmost (i.e., northeast) direction with this property. 
\begin{lem}
	Let $n_1,n_2\in\BB Z$ satisfy $n_1<n_2$, and consider the primal north-going branches $\dot\lambda^N_{n_1}$ and $\dot\lambda^N_{n_2}$, and the dual north-going branches $\lambda^N_{n_1}$ and $\lambda^N_{n_2}$. 
	\begin{itemize}
		\item[(i)] The branches $\lambda^N_{n_1}$ and $\dot\lambda^N_{n_1}$ are parallel in the following sense. For any $k\in\BB N$ consider the face $f$ on the right side of $\dot\lambda^N_{n_1}(k)$ when following the branch $\dot\lambda^N_{n_1}$ towards the root of the primal north-going tree. The dual branch $\lambda^N_{n_1}$ contains the edge $e'$ of $f$ which is adjacent to $\dot\lambda^N_{n_1}(k)$ in clockwise order around $f$.
		\item[(ii)] The primal branches $\lambda^N_{n_1}$ and $\lambda^N_{n_2}$ merge simultaneously as the dual branches $\dot\lambda^N_{n_1}$ and $\dot\lambda^N_{n_2}$ in the following sense. Let $e\in E(G)$ be the last edge on $\dot\lambda^N_{n_1}$ which is not part of $\dot\lambda^N_{n_2}$ when following the branch $\dot\lambda^N_{n_1}$ towards the root of the primal north-going tree. Consider the face $f$ on the right side of $e$ when following the branch $\dot\lambda^N_{n_1}$ towards the root of the tree. Let $e'\in E(G)$ be the edge of $f$ which is adjacent to $e$ in clockwise order around $f$. Then $e'$ in the first edge which is on the path of both $\lambda^N_{n_1}$ and $\lambda^N_{n_2}$ when following the branches towards the root of the dual north-going tree.
	\end{itemize}
	\label{prop:pd-trees-same}
\end{lem}

\begin{proof}
	Part (i) is immediate by induction, geometric considerations, and the definition of dual and primal north-going flow lines. Part (ii) is immediate by part (i), geometric considerations, and the definition of dual and primal north-going flow lines. See Figure~ \ref{fig-primaldualbranches}.
\end{proof}

\begin{figure}[ht!]
	\begin{center}
		\includegraphics[scale=1.6]{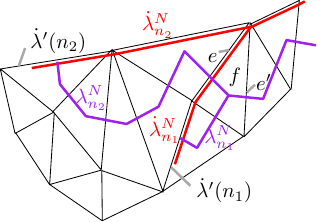}
	\end{center}\label{fig-primaldualbranches}
	\caption{An illustration of the statement and proof of Lemma \ref{prop:pd-trees-same}. The primal branches $\dot{\lambda}^N_{n_1}$ and $\dot{\lambda}^N_{n_2}$ are shown in red, and the corresponding dual branches ${\lambda}^N_{n_1}$ and ${\lambda}^N_{n_2}$ are shown in purple. To prove (i) we first note that for $k=1$ the dual branch $\lambda^N$ enters the face $f$ described in the statement of the lemma, which follows by definition of $\dot{\lambda}^N_{n_1}$ and $\lambda^N_{n_1}$; the dual branch ${\lambda}^N_{n_1}$ will visit all triangles in counterclockwise order which are between the two edges $\dot{\lambda}^N_{n_1}(0)$ and $\dot{\lambda}^N_{n_1}(1)$. We then use induction on $k$ to show that, conditioned on $\lambda^N_{n_1}$ entering the face $f$ in the statement of the lemma, $\lambda^N_{n_1}$ leaves the face through the edge $e'$; furthermore, by the definition of $\dot{\lambda}^N_{n_1}$ and $\lambda^N_{n_1}$ the dual branch $\lambda^N_{n_1}$ will enter the face which is on the right side of $\lambda^N_{n_1}(k+1)$ when following the branch $\lambda^N_{n_1}$ towards the root of the tree. To prove (ii) we observe that by (i), the definition of $e$ and the definition of $\lambda^N_{n_2}$, the branch $\lambda_{n_2}^N$ goes through the edge $e$, so the branch $\lambda_{n_2}^N$ enters the face $f$. By (i) the branch $\lambda_{n_1}^N$ also visits the face $f$, and by definition of the north-going dual branches both $\lambda_{n_1}^N$ and $\lambda_{n_2}^N$ will leave $f$ through the edge $e'$.
	}
\end{figure}

\begin{proof}[Proof of Proposition \ref{prop:twoNtrees}]
	The renormalized random walk $\dot Z^m$ converges in law to $Z$ by Donsker's theorem and symmetry between the west-going and north-going Peano curve in the primal map. By Lemma \ref{prop:primal-dual-order} and re-rooting invariance in law of $(G , \cO , \BB e)$, the law of $(\dot Z,\wt Z)$ is invariant under recentering. Furthermore, since the time reversal of $(\dot L,\wt L)$ describes the north-going flow line in the map encoded by the time reversal of $(R,L)$, which is equal in distribution to $(L,R)$, it follows that the time reversal of $(\dot L,\wt L)$ is equal in law to $(\dot R,\wt R)$. Hence it is sufficient to prove that for any $T,\ep>0$,
	\eqbn
	\lim_{m\rta\infty}\BB P\left[\sup_{0\leq n\leq Tm}|\dot{\mcl R_{n}}-\frac 13 \wt{\mcl R}_n|>\ep m^{1/2}\right]=0.
	\eqen
	For any fixed $n\in\BB N$ let $k_1$ (resp.\ $k_2$) be the length of $\dot\lambda_n^N$ (resp.\ $\dot\lambda_0^N$) until the two branches $\dot\lambda_n^N$ and $\dot\lambda_0^N$ merge, i.e.\ the number of edges along $\dot\lambda_n^N$ (resp.\ $\dot\lambda_0^N$) which are not contained in $\dot\lambda_0^N$ (resp.\ $\dot\lambda_n^N$). Then $k_1,k_2< M:=1+\sup_{0\leq \wt n\leq n} \dot{\mcl R}_{\wt n} - \inf_{0\leq \wt n\leq n} \dot{\mcl R}_{\wt n}$, and since $\dot{\mcl R}$ gives the height in the north-going tree, $\dot{\mcl R}_n=k_1-k_2$. For any $k,n\in\BB N$ let $N_k^{n,W}$ (resp.\ $\frk l^n$) be defined exactly as the stopping time $N_k^W$ (resp.\ the increasing process $\frk l$) in Lemma \ref{prop:discrete-excursions}, but for the rerooted map $(G,\mcl O,\dot\lambda'(n))$. 
	By Lemma \ref{prop:pd-trees-same}(ii), $\wt{\mcl R}_n=\frk l^n_{N_{k_2+1}^{n,W}}-\frk l_{N_{k_2+1}^{W}}$, since the lemma implies that the length of the north-going flow line in the dual map started from $\lambda'(0)$ (resp.\ $\lambda'(n)$) until the flow line reaches $\lambda'(N_{k_2+1}^{n,W})$ is given by $l_{N_{k_2+1}^{n,W}}$ (resp.\ $l^n_{N_{k_2+1}^{n,W}}$). It follows that
	\eqbn
	\begin{split}
		|\dot{\mcl R}_{n}-\frac 13 \wt{\mcl R}_n|
		& =\left|\left((k_1+1)-\frac 13 \frk l^n_{N_{k_1+1}^{n,W}}\right)-\left((k_2+1)-\frac 13 \frk l_{N_{k_2+1}^{W}}\right)\right|\\
		&\leq \sup_{0\leq k\leq M}
		\left(\left| k-\frac 13\frk l^{n}_{N_{k}^{n,W}} \right|+
		\left| k-\frac 13\frk l_{N_{k}^{W}} \right| \right).
	\end{split}
	\eqen
	We obtain our wanted result by applying a union bound, the fact that $(k-\frk l^{n}_{N_{k}^{n,W}})\eqD(k-\frk l_{N_{k}^{W}})$ for any $n,k\in\BB N$, and Lemma~\ref{prop:primal-dual-length-same} (with $M =m^{1/2 + 1/100}$):
	\eqbn
	\begin{split}
		\BB P&\left[\sup_{0\leq n\leq Tm}|
		\dot{\mcl R}_{n}-\frac 13 \wt{\mcl R}_n|>\ep m^{1/2}\right]\\
		&\,\,\,\,\,\,\leq 
		\BB P\left[\sup_{0\leq n\leq Tm} \dot{\mcl R}_n - \inf_{0\leq n\leq Tm} \dot{\mcl R}_n\geq m^{1/2+1/100}\right] \\
		&\,\,\,\,\,\,\,\,\,\,\,\,+ Tm \sup_{0\leq n\leq Tm} \BB P\left[\sup_{0\leq k\leq m^{1/2+1/100}}| k-\frac 13 \frk l_{N_k^{n,W}}|>\frac 12 \ep m^{1/2}\right]\\
		&\,\,\,\,\,\,= o_m(m^{-p}) ,\quad \forall p > 0. \qedhere
	\end{split}
	\eqen
\end{proof}

\section{Continuum decomposition of contour functions}
\label{sec:continuum-decomp}

Let $\kappa'  > 4$. Throughout this section we define
\eqb  \label{eqn:ig-parameter} 
\kappa := \frac{16}{\kappa'} , \quad 
\lambda' := \frac{\pi}{\sqrt{\kappa'}} ,\quad 
\chi := \frac{2}{\sqrt\kappa} - \frac{\sqrt\kappa}{2}   ,
\eqe 
as in~\cite{ig1,ig2,ig3,ig4}. 
Let $\wh h$ be a whole-plane GFF, viewed modulo a global additive multiple of $2\pi\chi$, as in~\cite{ig4}. Let $\eta'$ be the north-going whole-plane space-filling $\op{SLE}_{\kappa'}$ curve from $\infty$ to $\infty$ which is constructed from the $\pi/2$-angle flow lines of $\wh h$ as in Section~\ref{sec:sle-background}. Recall that a.s.\ for each $w \in \BB C$ the left and right outer boundaries of $\eta'$ stopped at the first time it hits $w$ are equal to the images of the west- and east-going flow lines $\eta_w^E$ and $\eta_w^W$ of $\wh h$ started from $w$, respectively (i.e.\ with angles $-\pi/2$ and $\pi/2$).

Let $(\BB C , h ,  0 , \infty)$ be a $\gamma$-quantum cone independent from $\wh h$ (equivalently, independent from $\eta'$). We take $\eta'$ to be parameterized by $\gamma$-quantum mass with respect to $h$ and to be normalized so that $\eta'(0) = 0$. For $t\in \BB R$, let $L_t$ (resp.\ $R_t$) be the $\gamma$-quantum length of the left (resp.\ right) outer boundary of $\eta'((-\infty, t] )$ with respect to $h$. Let $Z_t := (L_t , R_t)$, so that by~\cite[Theorem 1.9]{wedges}, $Z$ is a Brownian motion with covariance matrix
\eqb \label{eqn:bm-cov-matrix}
\Sigma = \alpha \left(   \begin{array}{cc}
	1 & -\cos(4\pi/\kappa') \\ 
	-\cos(4\pi/\kappa') & 1
\end{array}    \right) ,
\eqe 
where $\alpha > 0$ is the constant appearing in~\eqref{eqn:correl}. 

Fix $\theta \in (-\pi/2 , \pi/2)$ and let 
$\eta^\theta$ be the flow line of $\wh h$ started from 0 with angle $\theta$.
In subsequent sections we will only be interested in the case when $\kappa' = 12$ and $\theta = 0$, but the general case requires only slightly more work so we treat it here.
For $t\geq 0$, we say that $\eta'$ \emph{crosses} $\eta^\theta$ at time $t$ if for each $\ep > 0$, there exists $s_1 , s_2 \in (t-\ep , t+\ep)$ such that $\eta'(s_1)$ and $\eta'(s_2)$ lie on opposite sides of $\eta^\theta$. We define 
\eqbn
\mcl A := \left\{ t \geq 0\,:\, \text{$\eta'$ crosses $\eta^\theta$ at time $t$} \right\}   .
\eqen  

\begin{remark}
	If $\kappa'$ and $\theta$ are such that $\eta'$ a.s.\ crosses $\eta^\theta$ whenever it hits $\eta^\theta$ (which holds in particular when $\kappa'\geq 12$ and $\theta = 0$~\cite[Theorem 1.7]{ig4}) then $\mcl A = (\eta')^{-1}(\eta^\theta)$. In general, however, it is possible for $\eta'$ to hit $\eta^\theta$ without crossing it. Since the left and right boundaries of $\eta'$ stopped when it hits any $z\in\BB C$ are given by the flow lines of $\wt h$ angle $\pm\pi/2$, this will be the case if and only if $\kappa'$ and $\theta$ are such that the $\theta$-angle flow lines of $\wt h$ can hit either the $\pi/2$ or $-\pi/2$-angle flow lines of $\wt h$. As explained in~\cite[Section 3.6]{ig4}, this is the case if and only if $|\theta - \pi/2|  \wedge |\theta + \pi/2| <  \pi\kappa/(4-\kappa)$, where $\kappa$ is as in~\eqref{eqn:ig-parameter}. 
\end{remark}

The goal of this section is to describe the joint law of the pair $(Z , \mcl A)$. In the case when $\kappa' = 12$ and $\gamma =\sqrt{4/3}$, this description will be a continuum analogue of Proposition~\ref{prop:discrete-excursions}.
 
We know from~\cite[Theorem 1.11]{wedges} that $Z$ a.s.\ determines $h$ and $\eta'$ (modulo rotation) and from~\cite[Theorems 1.2 and 1.16]{ig4} that $\eta'$ a.s.\ determines $\eta^\theta$. Hence the set $\mcl A$ is a.s.\ equal to some deterministic functional of $Z$. We are not able to describe this functional explicitly. Instead, we will give an indirect description of the joint law of $(Z , \mcl A)$ by constructing this pair simultaneously from a certain Poisson point process of conditioned Brownian paths. See Figure~\ref{fig:flowline-decomp} for an illustration.

To describe the excursions of $Z$ away from the time set $\mcl A$, we need the following definition. 

\begin{defn} \label{def:bm-excursion}
Let $A : \BB R^2 \rta \BB R^2$ be a linear transformation which fixes the horizontal axis and maps the two coordinates $L$ and $R$ of the Brownian motion $Z$ to a pair of independent Brownian motions with the same variances as $L$ and $R$. Suppose given $T>0$. Let $\wh L$ be a Brownian bridge from 0 to $A x$ in time $T$ and let $\wh R$ be a one-dimensional Brownian excursion of time length $T$ (i.e.\ a Brownian bridge from 0 to 0 in time $T$ conditioned to stay positive). Let $\wh Z = (\wh L , \wh R)$. A \emph{Brownian excursion in the upper half-plane in time $T$ with covariance matrix $\Sigma$} is the random path $\dot Z = A^{-1} \wh Z$. 
A \emph{Brownian excursion in the right half-plane in time $T$ with covariance matrix $\Sigma$} is obtained by pre-composing $\wh Z$ with a rotation by angle $\pi/2$. 
\end{defn}

\begin{prop} \label{prop:ppp}
	For $s \geq 0$, let $\tau_s$ (resp.\ $\wt\tau_s$) be the smallest $t\geq s$ (resp.\ the largest $t \leq s$) for which $t \in \mcl A$. Also let $E_s$ be the event that $\eta'(s)$ lies to the left of $\eta^\theta$. For $s \geq 0$, let
	\begin{align} \label{eqn:signed-bm}
	X_s  := (R_s - R_{\wt\tau_s}) \BB 1_{E_s} - (L_s - L_{\wt\tau_s}) \BB 1_{E_s^c} ,\qquad s \geq 0 . 
	\end{align}
	There is a deterministic constant $p  \in (0,1)$, depending only on $\theta$ and $\kappa'$ and equal to $1/2$ for $\theta = 0$ such that the following is true. 
	The function $|X|$ has the law of $\alpha^{1/2}$ times a standard reflected Brownian motion (where $\alpha$ is as in~\eqref{eqn:bm-cov-matrix}) and 
	\eqbn
	\mcl A = X^{-1}(0) = \left\{\tau_s , \wt\tau_s \,:\, s \geq 0\right\}.
	\eqen
	Furthermore, if we condition on $\mcl A$, then the conditional law of the excursions $\{ (Z-Z_{\wt\tau_s})|_{[\wt\tau_s , \tau_s]} \,:\, s \geq 0\}$ is that of a collection of independent random paths, each of which is independently equal with probability $p$ to a Brownian excursion in the upper half-plane in time $\tau_s - \wt\tau_s$ with covariance matrix $\Sigma$; and is equal with probability $1-p$ to a Brownian excursion in the right half-plane in time $\tau_s - \wt\tau_s$ with covariance matrix $\Sigma$. 
\end{prop}

The process $X$ of~\eqref{eqn:signed-bm} is a continuum analogue of the process $\mcl X$ defined in Section~\ref{sec:excursions}.
Proposition~\ref{prop:ppp} implies that in the case when $\theta = 0$ (so $p = 1/2$), the process $X$ is obtained from $|X|$ by independently multiplying by $\pm 1$ with equal probability on each of its excursions away from 0. Therefore $ X$ has the law of $\alpha^{1/2}$ times a standard linear Brownian motion in this case. In general, $X$ has the law of a skew Brownian motion with parameter $p$~\cite{lejay-skew-bm}. 

We do not know the value of the constant $p$ except in the case when $\theta=0$. \textit{Update:} The value of $p$ for general $\gamma\in (0,2)$ and $\theta \in (-\pi/2,\pi/2)$ is computed in~\cite{asy-p-theta}.

\begin{figure}[ht!]
	\begin{center}
		\includegraphics[scale=1.4]{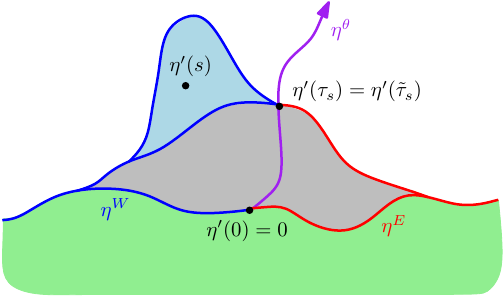}
	\end{center}
	\caption{
		An illustration of the objects involved in the statement of Proposition~\ref{prop:ppp}. The curve $\eta'$ traces the green region, then the gray region, then the blue region. For $s > 0$, $\wt\tau_s$ is the last time that $\eta'$ crosses the flow line $\eta^\theta$ before time $s$, and $\tau_s$ is the next time that $\eta'$ crosses $\eta^\theta$ after time $s$, equivalently (by Lemma~\ref{prop:chronological}) the next time after $s$ at which $\eta'$ hits $\eta'(\wt\tau_s)$. Proposition~\ref{prop:ppp} tells us that the set $\mcl A$ of times $\tau_s$ and $\wt\tau_s$ for $s\geq 0$ has the same law as the zero set of a standard linear Brownian motion and that the excursions of $Z$ away from this set are certain conditioned correlated Brownian motions.
	} \label{fig:flowline-decomp}
\end{figure}

In light of Proposition~\ref{prop:ppp}, one can define the \emph{local time of $Z$ at $\mcl A$}, which we call $\{\ell_t\}_{t\geq 0}$, to be $1/2$ times the local time of the reflected Brownian motion $|X|$ as in~\eqref{eqn:signed-bm} at zero (the reason for the factor of $1/2$ is so that $\ell_t$ is the local time of $X$ at 0 when $\theta=0$, in which case $X$ is a Brownian motion). This local time has a natural interpretation in terms of $h$ and $\eta^\theta$.  

\begin{prop} \label{prop:local-time}
	Let $\{\ell_t\}_{t\geq0 }$ be $1/2$ times the local time of $|X|$ at 0, as above There is a deterministic constant $c>0$ (possibly depending on $\theta$ and $\kappa'$) such that a.s.\ for each $t\geq0$, $\ell_t = c \nu_h(\eta^\theta \cap \eta'([0,t]) )$. 
\end{prop} 

We will eventually show via a rather indirect argument that the constant $c$ of Proposition~\ref{prop:local-time} is equal to $2$ when $\kappa' =12$ and $\theta = 0$ (see Corollary~\ref{cor:c-value} below). It remains an open problem to compute $c$ for other $(\kappa',\theta)$ pairs. \textit{Update:} The value of $c$ for general $(\kappa',\theta)$ pairs is computed in~\cite{asy-p-theta}.

It is not immediately obvious that Proposition~\ref{prop:ppp} uniquely characterizes the joint law of the pair $(Z , \mcl A)$. What is obvious is that Proposition~\ref{prop:ppp} gives a complete description of the law of the process
\eqb \label{eqn:excursion-process}
W_s = (W_s^L , W_s^R) := Z_s - Z_{\wt\tau_s} , \quad \forall s \geq 0 ,
\eqe 
which encodes the excursions of $Z$ away from $\mcl A$. 
Note that $W$ is a.s.\ continuous except at the times $\tau_s$ for $s\geq 0$ (where it has jump discontinuities in one coordinate). It turns out that $W$ actually determines $Z|_{[0,\infty)}$.

\begin{prop} \label{prop:bm-determined}
	In the setting of Proposition~\ref{prop:ppp}, it holds for each $t \in [0,\infty]$ that the Brownian motion $Z|_{[0,t]}$ (equivalently the quantum surface obtained by restricting $h$ to $\eta'([0,t])$) is a.s.\ determined by the process $W|_{[0,t]}$. 
\end{prop}

The remainder of this section is structured as follows. In Section~\ref{sec:ig-lemma}, we prove some lemmas based on the theory of imaginary geometry~\cite{ig1,ig2,ig3,ig4} which describe the interaction of $\wh h$ and the curves $\eta'$ and $\eta^\theta$. In Section~\ref{sec:lqg-lemma}, we prove (using the results of the preceding subsection and of~\cite{wedges}) that the times $\tau_s$ of Proposition~\ref{prop:ppp} are renewal times for $Z$, in the sense that each $\tau_s$ is a stopping time for $Z$ and the joint conditional law of $(Z-Z_{\tau_s})|_{[\tau_s,\infty)}$ and the part of $\eta^\theta$ not hit by $\eta'$ before time $\tau_s$ (the latter viewed as a curve on a certain quantum surface) is the same as the joint law of $Z|_{[0,\infty)}$ and $\eta^\theta$. In Section~\ref{sec:cont-proofs}, we conclude the proof of Propositions~\ref{prop:ppp} and~\ref{prop:local-time}. In Section~\ref{sec:bm-determined}, we prove Proposition~\ref{prop:bm-determined} via an argument which will also be relevant in Section~\ref{sec:flowline-conv}.

\subsection{Imaginary geometry lemmas}
\label{sec:ig-lemma}

In this subsection we will prove some lemmas about the interactions between $\wh h$, $\eta'$, and $\eta^\theta$ which are straightforward consequences of the results of~\cite{ig1,ig2,ig3,ig4}. These lemmas will eventually be used for the proofs of the main results of this section. 
In order to give precise statements of our results in the regime when $\eta'$ can hit $\eta^\theta$ without crossing it, we need the following definition.

\begin{defn} \label{def:covered}
	We say that a point $a\in \eta^\theta$ is \emph{covered} by $\eta'$ at time $t$ if there is a sequence of points $a_n \in \eta^\theta$ such that $a_n \rta a$ and each $a_n$ lies in the interior of $\eta'([0,t])$. 
\end{defn}

The set of points of $\eta^\theta$ which are covered by $\eta'$ at time $t$ is a segment of the curve $\eta^\theta$, as the following lemma demonstrates.

\begin{lem} \label{prop:chronological}
	For $a\in \eta^\theta$, let $t_a$ be the infimum of the times $t \geq 0$ at which $a$ is covered by $\eta'$. Then a.s., for each $a,b \in \eta^\theta$, we have $t_a < t_b$ if and only if $\eta^\theta$ hits $a$ before $b$. 
\end{lem}
\begin{proof}
	For $w\in \BB C$, let $\tau_w$ be the first time that $\eta'$ hits $w$. Also let  $\eta_w^E$ and $\eta_w^W$ be the flow lines of $\wh h$ started from $w$ with angles $-\pi/2$ and $\pi/2$, as above. 
	
	Suppose by way of contradiction that with positive probability, there exist points $a ,b \in \eta^\theta$ such that $\eta^\theta$ hits $a$ before $b$ but $t_b < t_a$. Then we can find a deterministic $w \in \BB Q^2$ such that with positive probability, there exists such an $a$ and $b$ satisfying $t_b < \tau_w < t_a$ (so in particular $b$ is covered by $\eta'([0 , \tau_w])$). Let $E_w$ be the event that this is the case.
	
	Let $\ell^L$ (resp.\ $\ell^R$) be the left (resp.\ right) boundary of $\eta'([0 , \tau_w])$, so that $\ell^L$ (resp.\ $\ell^R$) is the concatenation of the segments of the flow lines $\eta_0^W$ and $\eta_w^W$  (resp.\ $\eta_0^E$ and $\eta_w^E$) traced before these flow lines merge. By Definition~\ref{def:covered}, on $E_w$ we can find $b'  \in \eta^\theta$ which lies in the interior of $\eta'([0,\tau_w])$ and $a \in\eta^\theta$ with $t_a < \tau_w$ such that $a$ is hit by $\eta^\theta$ before $b'$. The curves $\ell^L$ and $\ell^R$ do not cross one another, share the same endpoints, disconnect $b'$ from $\infty$, and do not disconnect $a$ from $\infty$. It follows that $\eta^\theta$ must cross one of these two curves at least twice. Since $\eta^\theta$ does not cross $\eta_0^W$ or $\eta_0^E$, it follows that $\eta^\theta$ must cross either $\eta_w^W$ or $\eta_w^E$ at least twice. By~\cite[Theorem 1.9]{ig4}, two flow lines of $\wh h$ with different angles a.s.\ do not cross more than once, so we obtain the desired contradiction.
\end{proof}

In the next several lemmas, we will consider the filtration
\eqb \label{eqn:wpsf-filtration}
\mcl F_t := \sigma(\eta'|_{(-\infty ,t]} , \, \wh h|_{\eta'((-\infty , t])} , \, h) ,\qquad \forall t\in\BB R .
\eqe 
In the above definition we view $\eta'|_{(-\infty ,t]}$ as a curve embedded in the complex plane with a certain parameterization of time, and $\wh h|_{\eta'((-\infty , t])}$ as a random distribution restricted to a given domain of the complex plane (by contrast, in Section \ref{sec:lqg-lemma} we will often view the restriction of $\eta'$ as a curve on top of a quantum surface, hence we view it as a curve modulo conformal transformations).
We will show in Lemma~\ref{prop:wpsf-determined} below that the definition of $\mcl F_t$ is unchanged if we remove one of $\eta'|_{(-\infty ,t]} $ or $ \wh h|_{\eta'((-\infty , t])}$, but until this is established we need to include both in the definition of $\mcl F_t$. 

In what follows, we recall that if $A \subset \BB C$ is a random closed set coupled with $\wh h$, then $A$ is said to be a \emph{local set} for $\wh h$ if the conditional law of $\wh h|_{\BB C\setminus A}$ given $A$ and $\wh h|_A$ is that of a zero-boundary GFF on $\BB C\setminus A$ plus a harmonic function which depends only on $A$ and $\wh h|_A$. There are several other equivalent definitions of a local set which are given in~\cite[Lemma 3.9]{ss-contour}.

The following lemma describes the conditional law of $\wh h|_{\eta'([T,\infty))}$ given $\mcl F_T$ for a stopping time $T$.

\begin{lem}  \label{prop:ig-stopping}
	Let $T$ be a stopping time for the filtration $(\mcl F_t)$ of~\eqref{eqn:wpsf-filtration}. If $\kappa' \geq 8$, then the conditional law of $\wh h|_{ \eta'([T,\infty))}$ given $\mcl F_T$ is that of a zero-boundary GFF on $\eta'([T,\infty))$ plus the harmonic function which equals $\lambda' + \chi \op{arg} \phi'$ on the left boundary of $\eta'((-\infty , T])$ and $-\lambda' + \chi \op{arg} \phi'$ on the right boundary of $\eta'((-\infty , T])$, where $\lambda'$, $\chi$ are as in~\eqref{eqn:ig-parameter} and $\phi : \eta'([T,\infty)) \rta \BB H$ is a conformal map which takes $\eta'(T)$ to 0 and $\infty$ to $\infty$, and we define $\op{arg}\phi'$ such that it varies continuously. Furthermore, the conditional law of $\eta'|_{[T,\infty)}$ given $\mcl F_T$ is that of the 0-angle space-filling counterflow line of $\wh h |_{\eta'([T,\infty))}$ from $\eta'(T)$ to $\infty$. 
	
	If $\kappa' \in (4,8)$, let $\mcl U$ be the set of connected components of the interior of $\eta'([T,\infty))$. For $U\in \mcl U$, let $x_U$ (resp.\ $y_U$) be the last (resp.\ first) point of $\bdy U$ hit by $\eta' |_{(-\infty,0]} $. Then for $U\in\mcl U$, the conditional law of $\wh h|_{ U}$ given $\mcl F_T$ is that of a zero-boundary GFF on $U$ plus the harmonic function which equals $\lambda' + \chi \op{arg} \phi'$ on the left boundary of $U$ and $-\lambda' + \chi \op{arg} \phi'$ on the right boundary of $U$, where $\phi : U \rta \BB H$ is a conformal map which takes $x_U$ to 0 and $y_U$ to $\infty$ and the left (resp.\ right) boundary of $U$ is defined as the intersection of $\bdy U$ with the left (resp.\ right) boundary of $\eta'((-\infty , T])$. 
	The restrictions of $\wh h$ to different choices of $U \in \mcl U$ are conditionally independent given $\mcl F_T$. Furthermore, the conditional law of $\eta' \cap U$ given $\mcl F_T$ is that of the 0-angle space-filling counterflow line of $\wh h |_{U}$ from $x_U$ to $y_U$. 
\end{lem}
\begin{proof} 
	For $z\in\BB C$, let $T_z$ be the first time that $\eta'$ hits $z$. By the construction of space-filling $\op{SLE}_{\kappa'}$ in~\cite[Section 1.2.3]{ig4}, it is a.s.\ the case that for each $z\in \BB Q^2$, the left and right boundaries of $\eta'|_{(-\infty , T_z]}$ are given by the angle $\pm \pi/2$-flow lines $\eta^W$ and $\eta^E$ of $\wh h$ started from $z$. By~\cite[Theorem 1.1]{ig4}, it follows that the conditional law of $\wh h|_{\eta'([T_z , \infty))}$ given $\mcl F_{T_z}$ is as described in the statement of the lemma with $T_z$ in place of $T$. Furthermore, it follows from the construction of space-filling $\op{SLE}_{\kappa'}$ in~\cite[Section 1.2.3]{ig4} that the conditional law of $\eta'$ given $\mcl F_{T_z}$ is as described in the lemma with $T_z$ in place of $T$. From this, we obtain the statement of the lemma for a general stopping time $T$ for $\eta'$ which is a.s.\ equal to $T_z$ for some $z\in\BB Q^2$. We obtain the statement of the lemma when $T$ is a general stopping time for $(\mcl F_t)$ via a straightforward limiting argument using  continuity of $\eta'$ and the fact that the local set property is well-behaved under Hausdorff distance limits~\cite[Lemma 6.8]{MS-duke}.
\end{proof}

We know from~\cite[Theorem 1.16]{ig4} that $\eta'$ and $\wh h$ a.s.\ determine each other. This determination is local, in the following sense.

\begin{lem} \label{prop:wpsf-determined}
	Let $T$ be a stopping time for the filtration $(\mcl F_t)$ of~\eqref{eqn:wpsf-filtration}. Then $\eta'|_{(-\infty ,T]}$ is a.s.\ determined by $h$, $\eta'((-\infty , T])$, and $\wh h|_{\eta'((-\infty , T])}$. Furthermore, $\wh h|_{\eta'((-\infty , T])}$ is a.s.\ determined by $\eta'|_{(-\infty , T]}$. 
\end{lem}
\begin{proof}
	By Lemma~\ref{prop:ig-stopping}, the conditional law of $(\eta'|_{[T,\infty)} , \wh h|_{\eta'([T,\infty))} )$ given $\mcl F_T$ depends only on $\eta'((-\infty ,T])$. 
	Hence $(\eta'|_{[T,\infty)} , \wh h|_{\eta'([T,\infty))} )$ is conditionally independent from $\mcl F_T$ given $h$ and $\eta'((-\infty ,T])$, so $(\eta'|_{(-\infty , T]} , \wh h|_{\eta'((-\infty , T])})$ and $(\eta'|_{[T,\infty)} , \wh h|_{\eta'([T,\infty))})$ are conditionally independent given $h$ and $\eta'((-\infty , T])$. 
	Since $h$ and $\wh h$ a.s.\ determine $\eta'$~\cite[Theorem 1.16]{ig4}, it follows that $\eta'|_{(-\infty , T]}$ is a.s.\ determined by $h$, $\eta'((-\infty , T])$ and $\wh h|_{\eta'((-\infty , T])}$. 
	Furthermore, since $\eta'$, viewed modulo monotone re-parameterization, a.s.\ determines $\wh h$~\cite[Theorem 1.16]{ig4}, it follows that $\wh h|_{\eta'((-\infty , T])}$ is a.s.\ determined by $\eta'|_{(-\infty , T]}$. 
\end{proof}

Our next two lemmas concern the probabilistic properties of $\eta^\theta$.

\begin{lem}  \label{prop:curve-future}
	Let $T$ be a stopping time for $(\mcl F_t)$ such that $\eta'(T) \in \eta^\theta$ a.s. If $\kappa' \geq 8$, then if we condition on $\mcl F_T$, the curve (recall Lemma~\ref{prop:chronological}) $\eta^\theta \cap \eta'([T,\infty))$ is the $\theta$-angle flow line of $\wh h|_{\eta'([T,\infty))}$ started from $\eta'(T)$. 
	
	If $\kappa' \in (4,8)$, then with $\mcl U$ and $(x_U  , y_U)$ for $U\in\mcl U$ as in Lemma~\ref{prop:ig-stopping}, if we condition on $\mcl F_T$ the curve $\eta^\theta \cap \eta'([T,\infty))$ is the concatenation of the $\theta$-angle flow lines of $\wh h|_U$ started from $x_U$ in the order that the sets $U\in\mcl U$ are filled in by $\eta'$. 
\end{lem}  
\begin{proof} 
	Let $\eta^\theta$ be parameterized by $\gamma$-quantum length with respect to $h$ and let $\sigma$ be a stopping time for $\eta^\theta$, conditioned on $h$. Also let 
	\eqbn
	A :=  \eta'((-\infty , T]) \cup \eta^\theta([0,\sigma]) .
	\eqen  
	Each of $\eta'((-\infty , T])$ and $\eta^\theta([0,\sigma])$ are local sets for the conditional law of $\wh h$ given $h$, and both are a.s.\ determined by $h$ and $\wh h$~\cite[Theorem 1.2]{ig1},~\cite[Theorem 1.16]{ig4}. By~\cite[Lemmas 3.10 and 3.11]{ss-contour}, $A$ is a local set for $\wh h$ given $h$, and the conditional law of $\wh h|_{\BB C\setminus A}$ given $A$, $\wh h|_A$, and $h$ is given by an independent GFF in each complementary connected component of $\BB C\setminus A$, with the boundary data we would expect if $\eta^\theta \cap \eta'([T,\infty))$ were a flow line for the conditional law of $\wh h$ given $\mcl F_T$. By Lemma~\ref{prop:wpsf-determined}, $\mcl F_T$ is the same as the $\sigma$-algebra generated by $\wh h|_{\eta'((-\infty , T])}$ and $h$. Therefore, the conditional law of $\wh h|_{\BB C\setminus A}$ given $\mcl F_T$, $\eta^\theta([0,\sigma])$, and $\wh h|_{\eta^\theta([0,\sigma])}$ is given by an independent GFF in each complementary connected component of $\BB C\setminus A$, with the boundary data we would expect if $\eta^\theta \cap \eta'([T,\infty))$ were a flow line for $\wh h |_{\eta'([T,\infty))}$. By approximating a general stopping time for the conditional law of $\eta^\theta$ given $\mcl F_T$ by stopping times which take on only countably many possible values and applying~\cite[Lemma 6.8]{MS-duke}, we find that the same holds with $\sigma$ replaced by a stopping time for the conditional law of $\eta^\theta$ given $\mcl F_T$. The statement of the lemma now follows by conformally mapping $\eta'([T,\infty))$ (if $\kappa' \geq 8$) or a given $U\in\mcl U$ (if $\kappa'\in (4,8)$) to $\BB H$ and applying~\cite[Theorems 1.1 and 2.4]{ig1}. The condition that the image of $\wh\eta^\theta\cap \ol U$ under this conformal map has a continuous Loewner driving function follows from~\cite[Proposition 6.12]{ig1} and the fact that $\eta^\theta$ cannot trace an east-going or west-going flow line of $\wh h$ (which form the left and right boundaries of $\eta'$) along a non-trivial arc~\cite[Theorem 1.9]{ig4}.  
\end{proof}

\begin{lem} \label{prop:north-msrblty} 
	Let $t \geq 0$ and let $\eta^\theta_t$ be the segment of $\eta^\theta$ which is covered (Definition~\ref{def:covered}) by $\eta'$ at time $t$. Then $\eta^\theta_t$ is a.s.\ determined by $  \eta'|_{[0,t]}$. 
\end{lem}
\begin{proof}
	Let $D $ be the interior of $\eta'([0,\infty))$ and let $\ol\eta':= \eta'(-\cdot) $ be the time reversal of $\eta'$. Lemma~\ref{prop:ig-stopping} tells us that the conditional law of $\wh h|_D$ given $\mcl F_0$ is that of an independent zero-boundary GFF in each connected component of $D$ plus a certain harmonic function which depends only on $ D$. Furthermore, conditioned on $\mcl F_0$, it holds that $\ol\eta'|_{(-\infty , 0]}$ is the 0-angle space-filling counterflow line of $\wh h|_{\eta'([t,\infty))}$ from $\infty$ to 0 and $\eta^\theta $ is the $\theta$-angle flow line of $\wh h|_{\eta'([0,\infty))}$ from 0 to $\infty$. 
	
	Observe that $-t$ is the smallest time $s < 0$ for which $D \setminus \ol\eta' ((-\infty,s])$ has $\mu_h$-mass at most $t$, so $-t$ is a stopping time for $\ol\eta' $, conditioned on $\mcl F_0$. 
	By~\cite[Proposition 6.1]{ig1}, the conditional law of $\wh h |_{D\setminus \eta'([t,\infty))}$ given $\eta' ([t,\infty) ) $ and $\mcl F_0$ is that of an independent zero-boundary GFF in each connected component of $D\setminus \eta' ([t,\infty))$ plus a certain harmonic function which depends only on $\eta'([0,t])$. Furthermore, $\ol\eta'|_{[-t,0]}$ is the 0-angle space-filling counterflow line of this field started from $\eta' (t)$, and $\eta^\theta_t$ is the $\theta$-angle flow line of this field started from 0 (or the concatenation of the $\theta$-angle space-filling counterflow lines and flow lines in each component of $\eta'([0,t])$ if $\kappa'\in (4,8)$). In particular, $\wh h |_{D\setminus \eta' ([t,\infty) )} $, $ \eta'|_{[0,t]}$, and $\eta^\theta_t$ are conditionally independent from $\mcl F_0$ given $ \eta' ([0,t])$.
	
	By~\cite[Theorem 1.2]{ig1} and~\cite[Theorem 1.16]{ig4}, it follows that if we condition on $\eta'([0,t])$, then $\eta^\theta_t$ is a.s.\ determined by $\eta' |_{[0,t]}$. Hence $\eta^\theta_t$ is a.s.\ determined by $ \eta'|_{[0,t]}$. 
\end{proof}

\subsection{Liouville quantum gravity lemmas}
\label{sec:lqg-lemma}

Although the Brownian motion $Z$ a.s.\ determines $h$ and $\eta'$ modulo a complex affine transformation, this Brownian motion does not determine $\eta'$ and $h$ locally, in the sense that if $a,b\in\BB R$ with $a<b$, then the ``shape'' of $\eta'([a,b])$ (i.e.\ its embedding into $\BB C$, modulo complex affine transformations) is \emph{not} determined by $(Z-Z_a)|_{[a,b]}$. However, $(Z-Z_a)|_{[a,b]}$ a.s.\ determines $(h|_{\eta'([a,b])} , \eta'|_{[a,b]})$ modulo the choice of embedding into $\BB C$ (see Lemma~\ref{prop:local-msrblty} just below). To make this idea precise, we will use the following notation.

\begin{defn} \label{def:sle-surface}
	For $a,b\in \BB R \cup \{-\infty , \infty\}$ with $a<b$, let $\mcl S_{a,b} = (\eta'([a,b]) , h|_{\eta'([a,b])})$ be the quantum surface obtained by restricting $h$ to $\eta'([a,b])$. Also let $\wh\eta'_{a,b}$ be the curve $\eta'|_{[a,b]}$, viewed as a curve on the surface $\mcl S_{a,b}$ (still parameterized by $\gamma$-quantum mass).
\end{defn}

\begin{lem} \label{prop:local-msrblty}
	For $a , b \in \BB R \cup \{-\infty , \infty\}$ with $a<b$, the pairs $(\mcl S_{a,b} , \wh\eta'_{a,b})$ and $(Z-Z_a)|_{[a,b]}$ a.s.\ determine each other. 
\end{lem}
\begin{proof}  
	By~\cite[Theorem 1.9]{wedges}, for each $a \in \BB R$, the surfaces $\mcl S_{-\infty , a}$ and $\mcl S_{a,\infty}$ are independent quantum wedges of weight $2-\gamma^2/2$. By the construction of space-filling $\op{SLE}_{\kappa'}$ described in~\cite[Footnote 4]{wedges} and the conformal invariance of SLE, we find that the pairs $(\mcl S_{-\infty , a} , \wh\eta'_{-\infty , a})$ and $(\mcl S_{ a ,  \infty } , \wh\eta'_{ a ,  \infty })$ are independent. 
	
	It is clear that the pair $(\mcl S_{-\infty , a} , \wh\eta'_{-\infty , a})$ (resp.\ $(\mcl S_{ a ,  \infty } , \wh\eta'_{ a ,  \infty })$) a.s.\ determines $(Z-Z_a)|_{(-\infty , a]}$ (resp.\ $(Z-Z_a)|_{[a,\infty)}$).  Since $Z$ a.s.\ determines $(\mcl S_{-\infty , a} , \wh\eta'_{-\infty , a})$ and $(\mcl S_{ a ,  \infty } , \wh\eta'_{ a ,  \infty })$~\cite[Theorem 1.11]{wedges}, it follows that $(Z-Z_a)|_{(-\infty ,a]}$ (resp.\ $(Z-Z_a)|_{[a,\infty)}$) a.s.\ determines $(\mcl S_{-\infty , a} , \wh\eta'_{-\infty , a})$ (resp.\ $(\mcl S_{ a ,  \infty } , \wh\eta'_{ a ,  \infty })$). This completes the proof in the case when either $a = -\infty$ or $b = \infty$. 
	
	For $a,b\in\BB R$ with $a<b$, the pair $(\mcl S_{a,b} , \wh\eta'_{a,b})$ is a deterministic function of each of the pairs $(\mcl S_{-\infty , b} , \wh\eta'_{-\infty , b})$ and $(\mcl S_{a ,  \infty } , \wh\eta'_{ a ,  \infty })$. Consequently, the discussion in the first paragraph implies that $(\mcl S_{a,b} , \wh\eta'_{a,b})$ is a measurable function of each of $(Z-Z_b)|_{(-\infty ,b]}$ and $(Z-Z_a)|_{[a,\infty)}$. The intersection of the $\sigma$-algebra generated by $(Z-Z_b)|_{(-\infty ,b]}$ and $(Z-Z_a)|_{[a,\infty)}$ is the $\sigma$-algebra generated by $(Z-Z_a)|_{[a,b]}$. Consequently, $(\mcl S_{a,b} , \wh\eta'_{a,b})$ is a.s.\ determined by $(Z-Z_a)|_{[a,b]}$. It is clear that $(\mcl S_{a,b} , \wh\eta'_{a,b})$ a.s.\ determines $(Z-Z_a)|_{[a,b]}$, and we conclude. 
\end{proof}

There is also a variant of Lemma~\ref{prop:local-msrblty} with a stopping time for $Z$ in place of a deterministic time. 

\begin{lem} \label{prop:local-msrblty-stopping}
	Let $T$ be a stopping time for $Z$ and let $\mcl S_{T , \infty}$ and $\wh\eta'_{T,\infty}$ be as in Definition~\ref{def:sle-surface} (with $a = T$ and $b = \infty$). Then the conditional law of $(\mcl S_{T,\infty}, \wh\eta'_{T,\infty} )$ given $Z|_{(-\infty ,T]}$ is the same as the law of $(\mcl S_{0,\infty}, \wh\eta'_{0,\infty} )$, and is that of a $2-\gamma^2/2$-quantum wedge decorated by an independent space-filling $\op{SLE}_{\kappa'}$ from 0 to $\infty$ (or a concatenation of independent space-filling $\op{SLE}_{\kappa'}$'s in the bubbles of the surface if $\kappa' \in (4,8)$). Furthermore, $(\mcl S_{T,\infty}, \wh\eta'_{T,\infty} )$ and $(Z-Z_T)|_{[T,\infty)}$ a.s.\ determine each other.
\end{lem}
\begin{proof}
	By Lemma~\ref{prop:local-msrblty} and the invariance of the law of $(h,\eta')$ under translating time for $\eta'$ (\cite[Theorem 1.9]{wedges}), the statement of the lemma is true for deterministic $T$. It follows that the first assertion of the lemma (regarding the conditional law of $(\mcl S_{T,\infty}, \wh\eta'_{T,\infty} )$) is true if $T$ takes on only countably many possible values. For a general stopping time $T$, we choose a sequence of stopping times $T_k$ such that $T_k$ decreases to $T$ a.s.\ and each $T_k$ takes on only countably many possible values. By embedding the surface-curve pairs into a common domain, it is easy to see that $(\mcl S_{T_k,\infty}, \wh\eta'_{T_k,\infty} ) \rta (\mcl S_{T ,\infty}, \wh\eta'_{T ,\infty} )$ a.s.\ in the topology of quantum surfaces (in the sense defined in~\cite{shef-zipper,wedges}, i.e.\ the corresponding quantum area measures converge weakly if we embed the surfaces in the same way) in the first coordinate and the topology of uniform convergence on compact surfaces in the second coordinate. By the backward martingale convergence theorem, we infer that the conditional law of $(\mcl S_{T,\infty}, \wh\eta'_{T,\infty} )$ given $Z|_{(-\infty ,T]}$ is as described in the statement of the lemma. This conditional law does not depend on the realization of $Z|_{(-\infty ,T]}$, so $(\mcl S_{T,\infty}, \wh\eta'_{T,\infty} )$ is independent from $Z|_{(-\infty ,T]}$. Since $Z$ a.s.\ determines $(\mcl S_{T,\infty}, \wh\eta'_{T,\infty} )$, we infer that $(\mcl S_{T,\infty}, \wh\eta'_{T,\infty} )$ is a.s.\ determined by $(Z-Z_T)|_{[T,\infty)}$. It is clear that $(\mcl S_{T,\infty}, \wh\eta'_{T,\infty} )$ a.s.\ determines $(Z-Z_T)|_{[T,\infty)}$. 
\end{proof}

Our next two lemmas are the key inputs in the proofs of Propositions~\ref{prop:ppp} and~\ref{prop:local-time}. 

\begin{lem} \label{prop:stopping}
	For $s \geq 0$, let $\tau_s$ be as in Proposition~\ref{prop:ppp}. Then $\tau_s$ is a stopping time for $Z$. Furthermore, $\wt\tau_s$ is a.s.\ determined by $Z|_{[0,s]}$.   
\end{lem}
\begin{proof} 
	For $s\geq 0$, let $\mcl S_{0,s}$ and $\wh\eta'_{0,s}$ be as in Definition~\ref{def:sle-surface}. Also let $\eta^\theta_s$ be the segment of $\eta^\theta$ which is covered by $\eta'$ at time $s$, as in Lemma~\ref{prop:north-msrblty}; and let $\wh\eta^\theta_{0,s}$ be given by $\eta^\theta_s$, viewed as a curve on the surface $\mcl S_{0,s}$.
	
	By Lemma~\ref{prop:north-msrblty}, the curve $\eta^\theta_s$ is a.s.\ determined by $  \eta'|_{[0,s]}$. By forgetting the embedding of $\eta'|_{[0,s]}$, $h|_{\eta'([0,s])}$, and $\eta^\theta_s$ into $\BB C$, we obtain that (in the notation of Lemma~\ref{prop:local-msrblty}), $\wh\eta^\theta_{0,s }$ is a.s.\ determined by $(\mcl S_{0,s} , \wh\eta'_{0,s})$. By Lemma~\ref{prop:local-msrblty}, we find that the triple $(\mcl S_{0,s} , \wh\eta'_{0,s} , \wh\eta^\theta_{0,s})$ is a.s.\ determined by $Z|_{[0,s]}$. From this, it is clear that $\wt\tau_s$ is a.s.\ determined by $Z|_{[0,s]}$. 
	
	To check that $\tau_s$ is a stopping time for $Z$, let $b$ be the tip of $\eta^\theta_s$, so that $b$ is a point in the outer boundary of $\eta'((-\infty  ,s])$. Also let $\sigma$ be the first time after $s$ at which $\eta'$ hits $b$. 
	Since $\wh\eta^\theta_{0,s }$ is a.s.\ determined by $Z|_{[0,s]}$, it follows that $\sigma$ is a stopping time for $Z$. 
	We claim that $\tau_s = \sigma$ a.s.  
	Indeed, Lemma~\ref{prop:chronological} implies that $\eta'$ cannot cross $\eta^\theta$ between time $s$ and time $\sigma$, for otherwise it would a.s.\ cover a point of $\eta^\theta$ which $\eta^\theta$ hits after $b$. 
	The Markov property of Lemma~\ref{prop:ig-stopping} implies that $\eta'$ a.s.\ hits points on the outer boundary of $\eta'([0,t])$ lying on either side of $b$ in every open interval of times containing $\sigma$. Hence $\eta'$ crosses $\eta^\theta$ at time $\sigma$.
	
	Let $r$ be the quantum length of the segment of the outer boundary of $\eta'((-\infty , s])$ between $\eta'(s)$ and $b$. 
	By the first paragraph, $r$ is a.s.\ determined by $Z|_{[0,s]}$. By the peanosphere construction and since $\tau_s = \sigma$, we find that $\tau_s$ is a.s.\ equal to the smallest $t\geq s$ for which $R_t  - R_s = -r$ (resp.\ $L_t -L_s =- r$) if $b$ lies to the right (resp.\ left) of $\eta'(s)$ on the outer boundary of $\eta'((-\infty , s])$. Hence $\tau_s$ is a stopping time for $Z$.  
\end{proof}

\begin{lem} \label{prop:hit-future}
	Let $T$ be a stopping time for $Z$ such that $\eta'(T) \in \eta^\theta$ a.s. Let $\mcl S_{T , \infty}$ and $\wh\eta'_{T,\infty}$ be as in Definition~\ref{def:sle-surface}. Also let $\wh\eta^\theta_{T , \infty}$ be given by $\eta^\theta \cap \eta'([T,\infty))$ viewed as a curve on the surface $\mcl S_{T,\infty}$. Then the conditional law of the triple $(\mcl S_{T,\infty}, \wh\eta'_{T,\infty}, \wh\eta^\theta_{T,\infty})$ given $Z|_{(-\infty ,T]}$ is the same as the law of $(\mcl S_{0,\infty}, \wh\eta'_{0,\infty}, \wh\eta^\theta_{0,\infty})$. 
\end{lem}
\begin{proof}
	Since $Z|_{(-\infty ,t]}$ is measurable with respect to the $\sigma$-algebra $\mcl F_t$ of~\eqref{eqn:wpsf-filtration} for each $t > 0$, we infer that $T$ is also a stopping time for this latter filtration.  
	By Lemma~\ref{prop:curve-future}, the conditional joint law of the pair $(\eta'|_{[T,\infty)} , \eta^\theta \cap \eta'([T,\infty)))$ given $\mcl F_T$ is the same as the unconditional law of the pair $(\eta'|_{[0,\infty)} , \eta^\theta  )$, modulo a conformal map. By Lemma~\ref{prop:wpsf-determined}, the conditional joint law of the pair $(\eta'|_{[T,\infty)} , \eta^\theta_t)$ given $\eta'([T,\infty))$, defined as in~\eqref{eqn:wpsf-filtration}, is the same as the unconditional law of the pair $(\eta'|_{[0,\infty)} , \eta^\theta  )$, viewed as curves modulo time parameterization and modulo a conformal map. By~\cite[Theorem 1.2]{ig1} and~\cite[Theorem 1.16]{ig4}, $\eta'|_{[T,\infty)}$ a.s.\ determines $\eta^\theta\cap \eta'([T,\infty))$. 
	We conclude by forgetting the embedding of the pair $(\eta'|_{[T,\infty)} , \eta^\theta \cap \eta'([T,\infty)) $ into $\BB C$ and applying Lemma~\ref{prop:local-msrblty-stopping}
\end{proof}

\subsection{Proof of Propositions~\ref{prop:ppp} and~\ref{prop:local-time}}
\label{sec:cont-proofs}

We are now ready to complete the proofs of the first two main results of this section. 

\begin{proof}[Proof of Proposition~\ref{prop:ppp}] 
	By Lemma~\ref{prop:chronological}, for $s \geq 0$, $\tau_s$ is the smallest $t  >s$ for which $\eta'(t) = \eta'(\wt\tau_s)$. By the peanosphere construction (c.f.\ the proof of Lemma~\ref{prop:stopping}), if $E_s$ occurs, then $R_t \geq R_{\wt\tau_s}$ for $t\in [\wt\tau_s , \tau_s]$ and $\tau_s$ is the smallest $t > s$ for which $R_t = R_{\tau_s}$. By Lemma~\ref{prop:stopping}, the conditional law of $Z|_{[s , \tau_s]}$ given $Z|_{[0 , s]}$, $(Z-Z_{\tau_s})|_{[\tau_s,\infty)}$ and $E_s$ is the same as the law of a Brownian motion with covariance matrix $\Sigma$ started from $Z_s$ conditioned on the event that its second coordinate hits $R_{\tau_s}$ for the first time at time $\tau_s$. For any $t \in (\wt\tau_s , \tau_s)$, we have $\wt\tau_t = \wt\tau_s$ and $\tau_t = \tau_s$. By taking a limit as $t\rta  \wt\tau_s$ and applying the above conditioning result, we find that the conditional law of $(Z - Z_{\wt \tau_s})|_{[ \wt\tau_s , \tau_s]}$ given $ Z_{[0,\wt\tau_s]} $, $(Z-Z_{\tau_s})|_{[\tau_s,\infty)}$, and $E_s$ is that of a Brownian excursion in the upper half-plane in time $\tau_s -\wt\tau_s$ with covariance matrix $\Sigma$.
	If $E_s$ does not occur, the same holds with $L$ in place of $R$ and the right half-plane in place of the upper half-plane.  
	
	Hence if we condition on $\sigma(\mcl A , \, E_s\,:\, s \geq 0)$, then the conditional law of the collection of excursions $\{ (Z-Z_{\wt\tau_s})|_{[\wt\tau_s , \tau_s]} \,:\, s \geq 0\}$ is that of a collection of independent random paths which each have the law of a Brownian excursion in time $\tau_s -  \wt\tau_s$ with covariance matrix $\Sigma$ in the upper (resp.\ right) half-plane if $E_s$ (resp.\ $E_s^c$) occurs.
	
	By Lemmas~\ref{prop:stopping} and~\ref{prop:hit-future}, the set $\mcl A$ is a regenerative set, so it is the range of a subordinator. 
	We will now argue that this subordinator is in fact a $1/2$-stable subordinator using a scaling argument.
	For $C>0$, let $h^C = h + \frac{1}{\gamma} \log C$ and let $\eta'_C$ be equal to $\eta'$, parameterized by $\gamma$-quantum mass with respect to $h^C$. By the scale invariance property of quantum cones~\cite[Proposition 4.13(i)]{wedges} and independence of $h$ and $\eta^\theta$, the triple $(h^C , \eta'_C , \eta^\theta)$ agrees in law with $(h, \eta' , \eta^\theta)$. In particular, $C \mcl A \eqD \mcl A$ for each $C> 0$, which implies by \cite[Lemma 1.11 and Theorem 3.2]{bertoin-sub} that $\mcl A$ has the law of the range of a $\beta$-stable subordinator for some $\beta \in (0,1)$.
	
	To identify $\beta$, we observe that if we condition on $Z|_{[0,s]}$ for some $s> 0$, then the conditional law of $\tau_s-s$ is the same as the law of the first time a Brownian motion started from $|X_s |$ reaches 0. Hence for $r>0$, 
	\eqbn
	\BB P\left[ \tau_s - s > r \,\big |\, Z|_{[0,s]} \right] \asymp  |X_s|^{1/2} (r \vee |X_s|)^{-1/2}  .
	\eqen
	We have $|X_s| \leq \sup_{t\in [0,s]} |Z_t|$, and with positive probability we have $|X_s| \geq s^{1/2}$. Hence we can average over all possible values of $|X_s|$ to get that for $r > s^{1/2}$, 
	\eqbn
	\BB P\left[ \tau_s - s > r   \right] \asymp s^{1/2} r^{-1/2} .
	\eqen
	Therefore $\beta = 1/2$. 
	
	Since we know the conditional law of $|X|$ given $\mcl A$ is the same as the conditional law of a reflected Brownian motion given its zero set (which is the range of a $1/2$-stable subordinator) we find that $|X|$ is a reflected Brownian motion. 
	It remains only to prove that there is a $p \in (0,1)$ (depending only on $\kappa'$ and $\theta$) such that if we condition on $\mcl A$, then each of the excursions of $Z$ away from $\mcl A$ lies in the upper (resp.\ right) half-plane with probability $p$ (resp.\ $1-p$). 
	
	To this end, for $\ep > 0$ let $\tau_0^\ep = \wt\tau_0^\ep =0$ and inductively for $k\in\BB N$ let $\wt\tau_k^\ep$ and $\tau_k^\ep$ be the smallest pair of times of the form $(\wt\tau_s , \tau_s)$ such that $\tau_s - \wt \tau_s \geq \ep$ and $\wt\tau_s > \tau_{k-1}^\ep$. Equivalently, $\wt\tau_k^\ep$ and $\tau_k^\ep$ are the left and right endpoints of the $k$th excursion of $X$ away from 0 with time length at least $\ep$. By Lemmas~\ref{prop:stopping} and~\ref{prop:hit-future}, the excursions $(Z-Z_{\wt\tau_k^\ep})|_{[\wt\tau_k^\ep  , \tau_k^\ep]}$ for $k\in\BB N$ are iid and each is independent from $Z|_{[0,\tau_{k-1}^\ep]}$.
	Let $F_k^\ep$ be the event that $\eta'([ \wt\tau_k^\ep, \tau_k^\ep])$ lies to the right of $\eta^\theta$. By scale invariance and since the excursions $(Z-Z_{\wt\tau_k^\ep})|_{[ \wt\tau_k^\ep, \tau_k^\ep]}$ are iid, we find that $p:= \BB P[F_k^\ep]$ does not depend on $k$ or $\ep$. By symmetry the law of $(\wt\tau_k^\ep , \tau_k^\ep)$ is unaffected if we condition on $F_k^\ep$ so also $\BB P[F_k^\ep]$ is unaffected if we condition on the times $ \wt\tau_k^\ep  $ and $\tau_k^\ep$. By sending $\ep\rta0$, we obtain the statement of the lemma for this choice of $p$. By symmetry we must have $p = 1/2$ when $\theta=0$.
\end{proof}

\begin{proof}[Proof of Proposition~\ref{prop:local-time}]
	The proof is similar to the proof that the quantum lengths of an SLE$_\kappa$ curve on an independent $(\gamma-2/\gamma)$-quantum wedge viewed from the left and right sides agree a.s., see~\cite[Theorem 1.8]{shef-zipper}. 
	For $ u\geq 0$ let $T_u := \inf\{t\geq 0\,:\, \ell_t = u\}$. Also let $F(u) := \nu_h(\eta^\theta_{T_u})$, with $\eta^\theta_{T_u}$ as in Lemma~\ref{prop:north-msrblty} with $s = T_u$. We must show that there is a constant $c>0$ as in the statement of the lemma such that a.s.\ $F(u) = cu$ for each $u\geq 0$.
	
	Note that each $T_u$ is a stopping time for $Z$ and $\{T_u \,:\, u\geq 0\} = \{\tau_s \,:\, s\geq0\} = \mcl A$. By Lemma~\ref{prop:hit-future}, we find that (in the notation of that lemma) for each $u\geq 0$, the conditional law of the triple $(\mcl S_{T_u ,\infty}, \wh\eta'_{T_u ,\infty}, \wh\eta^\theta_{T_u ,\infty})$ given $Z|_{(-\infty ,T_u]}$ is the same as the law of $(\mcl S_{0,\infty}, \wh\eta'_{0,\infty}, \wh\eta^\theta_{0,\infty})$. Therefore $F(u)$ has stationary increments. By the Birkhoff ergodic theorem, there is a random variable $Y$ (possibly infinite) such that $\lim_{u\rta\infty} F(u)/u = Y$ a.s. The random variable $Y$ is a.s.\ determined by $(\mcl S_{T_u ,\infty} , \wh\eta'_{T_u ,\infty})$ for each $u > 0$. 
	By Lemma~\ref{prop:local-msrblty-stopping}, $Y$ is a.s.\ determined by $(Z-Z_{T_u})|_{[T_u,\infty)}$ for each $u \geq 0$. Since $T_u \rta \infty$ a.s., $Y$ is a.s.\ equal to some deterministic constant $c \geq 0$. 
	
	We will now argue that in fact a.s.\ $F(u) = c u$ for each $u \geq 0$. 
	For $C>0$, let $h^C $ and $\eta'_C$ be as in the proof of Proposition~\ref{prop:ppp}, so that $(h^C , \eta'_C) \eqD (h ,\eta')$. 
	The $\gamma$-quantum area measure and $\gamma$-quantum length measure induced by $h^C$ satisfy $\mu_{h^C} = C \mu_h$ and $\nu_{h^C} = C^{1/2} \nu_h$. Hence if we let $Z^C$ and $X^C$ be defined in the same manner as $Z$ and $X$ but with $h^C$ in place of $h$ and $\eta'_C$ in place of $\eta'$, then $Z^C_t = C^{1/2} Z_{C^{-1} t}$ and $X^C_t = C^{1/2} X_{C^{-1} t}$. Let $\ell^C$ be the local time of $X^C$ at 0 and for $ u\geq 0$ let $T_u^C := \inf\{t\geq 0\,:\, \ell_t^C = u\}$. Then $\ell_t^C = C^{1/2} \ell_{C^{-1} t}$ and hence $T_u^C = C T_{C^{-1/2} u}$ and 
	\eqbn
	\nu_{h^C}(\eta^\theta \cap \eta'_C([0,T_u^C])) = C^{1/2} \nu_h(\eta^\theta \cap \eta'([0,  T_{C^{-1/2} u}]) ) = C^{1/2} F(C^{-1/2} u) .
	\eqen
	Therefore $C^{1/2} F(C^{-1/2} u) \eqD F(u)$ for each $C > 0$, so $F(u)/u \eqD F(u')/u'$ for each $u , u' > 0$. Since $\lim_{u\rta\infty} F(u)/u = c$ a.s., it must be the case that $F(u) = c u$ a.s.\ for each $u>0$. In particular, $c>0$. Since $F(u)$ is non-decreasing, we infer that a.s.\ $F(u) = cu$ for every $u \geq 0$. 
\end{proof}

\subsection{Proof of Proposition~\ref{prop:bm-determined}}
\label{sec:bm-determined}

For the proof of Proposition~\ref{prop:bm-determined}, we will need two lemmas which will also be used in Section~\ref{sec:flowline-conv}.
Our first lemma shows that the process $W$ of~\eqref{eqn:excursion-process} encodes at most as much information as the correlated Brownian motion $Z$ but more information than the process $X$ of~\eqref{eqn:signed-bm} (we will eventually show that in fact $W$ encodes the same information as $Z$).

\begin{lem} \label{prop:excursion-msrblty}
	For each $t\geq 0$, $W|_{[0,t]}$ is a.s.\ determined by $Z|_{[0,t]}$ and $W|_{[0,t]}$ a.s.\ determines $X|_{[0,t]}$. Furthermore, $W|_{[t,\infty)}$ is conditionally independent from $Z|_{[0,t]}$ given $W|_{[0,t]}$. 
\end{lem}
\begin{proof}
	It follows from Lemma~\ref{prop:stopping} that each $\wt\tau_s$ for $s\leq t$ is a.s.\ determined by $Z|_{[0,t]}$, whence $W|_{[0,t]}$ is a.s.\ determined by $Z|_{[0,t]}$. By Proposition~\ref{prop:ppp}, if $s \in [0,t]$ such that the event $E_s$ of Proposition~\ref{prop:ppp} occurs, then $L_s$ a.s.\ hits 0 infinitely often in $[\wt\tau_s ,\wt\tau_s+\ep]$ for every $\ep > 0$; but $R_s$ is a.s.\ positive on $[\wt\tau_s , \tau_s]$. Therefore, if we let $\sigma_s$ be the largest $s' < s$ for which one of the two coordinates $W^L_{s'}$ or $W^R_{s'}$ of $W_{s'}$ is 0, then $W^L_{\sigma_s} = 0$ if and only if $E_s$ occurs. Furthermore, $\wt\tau_s$ is the supremum of the times $s' < s$ at which $W$ has a jump discontinuity and $\tau_s$ is the infimum of the times $s' > s$ at which $W$ has a jump discontinuity. Therefore $W|_{[0,t]}$ determines $E_s$, $\wt\tau_s$, $\tau_s \wedge t$, and $Z_s - Z_{\wt\tau_s}$ for $s\in [0,t]$. Hence $W|_{[0,t]}$ a.s.\ determines $X|_{[0,t]}$. 
	
	By Proposition~\ref{prop:ppp}, if we condition on $Z|_{[0,t]}$ then the conditional law of $W|_{[t, \tau_t]}$ is that of a Brownian motion with covariance matrix $\Sigma$ started from $W_t$, run until the first time its first (resp.\ second) coordinate reaches 0 if $E_t^c$ (resp.\ $E_t$) occurs. By the preceding paragraph, $E_t$ is determined by $W|_{[0,t]}$, so this conditional law depends only on $W|_{[0,t]}$. Furthermore, by Lemma~\ref{prop:hit-future} the conditional law of $W|_{[\tau_t,\infty)}$ given $Z|_{[0,\tau_t]}$ is the same as the unconditional law of $W$. Hence the conditional law of $W|_{[t,\infty)}$ given $Z|_{[0,t]}$ depends only on $W|_{[0,t]}$.
\end{proof}

The following statement is the key input in the proof of Proposition~\ref{prop:bm-determined}. 

\begin{lem} \label{prop:2bm-same-excursion}
	Suppose given a coupling $(Z , Z', W)$ of another Brownian motion $Z' \eqD Z$ with $(Z,W)$ such that a.s.\ $Z_s' - Z_{\wt\tau_s}' = W_s$ for each $s \geq 0$. Suppose also that there is a filtration $\{\mcl F_t\}_{t\geq 0}$ such that $Z$ and $Z'$ are adapted to $\mcl F_t$ and for each $t_2 \geq t_1 \geq  0$, we have
	\eqb \label{eqn:2bm-0mean}
	\BB E\left[Z_{t_2} - Z_{t_2}' \,|\, \mcl F_{t_1} \right] = Z_{t_1} - Z_{t_1}' .
	\eqe 
	Then $Z \equiv Z'$ a.s.
\end{lem}
\begin{proof}
	Our hypothesis~\eqref{eqn:2bm-0mean} implies that $Z-Z'$ is a continuous $\mcl F_t$-martingale. We will show that $Z-Z'$ a.s.\ has zero quadratic variation, so is a.s.\ constant. 
	Fix $T > 0$ and $\zeta\in (0,1/4)$. Since $Z$ and $Z'$ are both Brownian motions, $Z-Z'$ is a.s.\ locally H\"older continuous of any exponent $<1/2$, so there a.s.\ exists a random finite constant $C> 0$ such that a.s.\ $|Z_t - Z_t' - Z_s + Z_s'| \leq C |t-s|^{1/2-\zeta}$ for each $s,t \in [0,T]$. Let
	$\mcl A  $
	be as in Proposition~\ref{prop:ppp}. Our choice of coupling implies that $Z-Z'$ is a.s.\ constant on each interval of $[0,\infty)\setminus \mcl A$. 
	For $ k \in\BB N$ and $j\in [0,2^k]_{\BB Z}$, let $t_j^k := j 2^{-k} T$ and let $\mcl T^k$ be the set of $j\in [1,2^k]_{\BB Z}$ for which $[t_{j-1}^k , t_j^k] \cap \mcl A = \emptyset$. By Proposition~\ref{prop:ppp}, the set $\mcl A  = \{\tau_s , \wt\tau_s \,:\, s \geq 0\}$ a.s.\ has Minkowski dimension $1/2$, so it is a.s.\ the case that $\# \mcl T^k \leq 2^{(1/2 + o_k(1)) k }$ as $k\rta\infty$. Therefore, we a.s.\ have
	\alb
	\sum_{j=1}^{2^k} \left(Z_{t_j^k} - Z'_{t_j^k} -Z_{t_{j-1}^k} + Z'_{t_{j-1}^k}\right)^2 
	&=\sum_{j \in\mcl T^k} \left(Z_{t_j^k} - Z'_{t_j^k} -Z_{t_{j-1}^k} + Z'_{t_{j-1}^k}\right)^2  \\
	&\leq C 2^{(1/2 + o_k(1)) k} 2^{-(1 - 2\zeta) k}
	= o_k(1) .
	\ale
	That is, the quadratic variation of $Z-Z'$ is a.s.\ equal to 0. 
\end{proof}

\begin{proof}[Proof of Proposition~\ref{prop:bm-determined}]
	By the last assertion of Lemma~\ref{prop:excursion-msrblty}, it suffices to prove the statement of the proposition in the case when $t = \infty$. 
	Let $(Z , Z' , W )$ be coupled in such a way that a.s.\ $Z'_s - Z'_{\wt\tau_s} = Z_s - Z_{\wt\tau_s} = W_s$ for each $s\geq 0$, $(Z' , W) \eqD (Z, W)$, and $Z$ and $Z'$ are conditionally independent given $W$. We must show that a.s.\ $Z \equiv Z'$. We will do this by checking the condition of Lemma~\ref{prop:2bm-same-excursion} for the filtration
	\eqbn
	\mcl F_t := \sigma\left((Z' , Z  )|_{[0,t]} \right) .
	\eqen 
	By our choice of coupling, the conditional law of $Z'$ given $Z|_{[0,t]}$ and $W$ depends only on $W$. 
	By Lemma~\ref{prop:excursion-msrblty}, $W|_{[t,\infty)}$ is conditionally independent from $Z|_{[0,t]}$ given $W|_{[0,t]}$. Hence the conditional law of $Z'$ given $Z|_{[0,t]}$ is the same as its conditional law given only $W|_{[0,t]}$. In particular, the conditional law of $(Z' - Z'_t)|_{[t,\infty)}$ given $\mcl F_t$ is the same as its conditional law given only $(Z' , W)|_{[0,t]}$. Since $Z'|_{[0,t]}$ a.s.\ determines $W|_{[0,t]}$ (Lemma~\ref{prop:excursion-msrblty}), this is the same as the conditional law of $(Z' - Z'_t)|_{[t,\infty)}$ given only $Z'|_{[0,t]}$, which is the same as the unconditional law of $Z'$. By symmetry, the conditional law of $(Z-Z_t)|_{[t,\infty)}$ given $\mcl F_t$ is the same as the unconditional law of $Z$. Since $Z$ and $Z'$ have mean 0, we find that the condition of Lemma~\ref{prop:2bm-same-excursion} is satisfied, so $Z\equiv Z'$.
\end{proof}

\section{Joint convergence of one Peano curve and a single dual flow line}
\label{sec:flowline-conv}
 
 In the remainder of this paper we will restrict attention to the case when $\gamma = \sqrt{4/3}$ and $\kappa' = 12$. 
 Let $Z  =(L,R)$ be the peanosphere Brownian motion of Section~\ref{sec:peanosphere-background} with $\kappa' = 12$, so that the covariance matrix of $Z$ is given by
 \eqb \label{eqn:triple-conv-cov}
 \alpha  \left(  \begin{array}{cc}
 	1 & -1/2 \\ 
 	-1/2 & 1
 \end{array}  \right)
 \eqe 
 with $\alpha$ the constant appearing in~\eqref{eqn:correl} for $\gamma = \sqrt{4/3}$.  
 
 Define the events $E_s$ and the times $\tau_s$ and $\wt\tau_s$ for $s\geq 0$ and the random function $X$ as in Proposition~\ref{prop:ppp} with $\kappa'=12$ and $\theta=0$. Since $p=1/2$ in this case, the process $X$ has the law of $\alpha^{1/2}$ times a standard linear Brownian motion. 
 
 Recall the definition of the two-dimensional random walk $\mcl Z = (\mcl L , \mcl R)$ from Section~\ref{sec:bipolar-map} and the definition of $Z^m = (L^m , R^m)$ from Section~\ref{sec:mainres}.
 Let $\mcl X$ be the random walk from Section~\ref{sec:excursions}. Extend $\mcl X$ to $[0,\infty)$ by linear interpolation. For $m\in\BB N$ and $t\geq 0$, let  
 \eqb \label{eqn:rescaled-X}
 X_t^m := \left( \frac{3 \alpha}{2} \right)^{1/2} m^{-1/2} \mcl X_{t m}    .
 \eqe 
 In this section, we will prove the following proposition.
 
 \begin{prop} \label{prop:triple-conv}
 	In the notation above, we have $(Z^m , X^m ) \rta (Z, X)$ in law. 
 \end{prop} 
 
 It is clear from Donsker's theorem and Section~\ref{sec:excursions} that $Z^m \rta Z$ and $X^m \rta X$ in law, so the non-trivial content of Proposition~\ref{prop:triple-conv} is the convergence of the joint law. To prove this convergence, we start in Section~\ref{sec:random-walk-conv} by proving a straightforward invariance principle for a certain conditioned random walk toward the conditioned Brownian motions appearing in the statement of Proposition~\ref{prop:ppp}. In Section~\ref{sec:excursion-conv}, we use this invariance principle to prove that the pairs $(Z^m , X^m)$ converge in law along subsequences to couplings of the form $(Z' , X)$ where $Z'$ is a Brownian motion with the same law as $Z$; and that in any such coupling, the law of the excursions of $Z'$ away from the zero set of $X$ (recall from Proposition~\ref{prop:ppp} that this zero set is the same as the time set $\mcl A$) agree in law with the corresponding excursions of $Z$; see Lemma~\ref{prop:subsequence-conv} below. In Section~\ref{sec:triple-conv-proof}, we will conclude the proof of Proposition~\ref{prop:triple-conv} by showing that any such subsequential limit $(Z' , X)$ must agree in law with $(Z, X)$. This is done by means of Lemma~\ref{prop:excursion-msrblty}.

 \subsection{A scaling limit result for conditioned random walk}
 \label{sec:random-walk-conv}

Let $\Sigma$ be a symmetric positive definite $2\times 2$ matrix. Recall from Definition~\ref{def:bm-excursion} the definition of a correlated two-dimensional Brownian excursion with covariance matrix $\Sigma$ (which appear in the description of the conditional law given the time set $\mcl A$ of the excursions of $Z$ away from $\mcl A$ Proposition~\ref{prop:ppp}).
 In this subsection we will prove a simple scaling limit statement for this conditioned Brownian motion, which is needed for the proof of Proposition~\ref{prop:triple-conv}. Before stating and proving this result, we record the following fact from elementary probability theory which we will use several times throughout this section (see, e.g.,~\cite[Lemma 4.10]{gms-burger-cone} for a proof).

 \begin{lem} \label{prop:cond-law-conv}
 	Let $(A_m,B_m)$ be a sequence of pairs of random variables taking values in a product of separable metric spaces $\Omega_A\times\Omega_B$ and let $(A,B)$ be another such pair of random variables. Suppose $(A_m , B_m) \rta (A,B)$ in law. Suppose further that there is a family of probability measures $\mu_b$ on $\Omega_A$, indexed by $\Omega_B$, and a family of $\sigma(A_m)$-measurable events $E_m$ with $\lim_{m\rta\infty} \BB P[E_m] =1$ such that for each bounded continuous function $f : \Omega_A\rta\BB R $, we have
 	\eqbn
 	\BB E\left[ f(A_m) \, |\, B_m  \right] \BB 1_{E_m} \rta \BB E_{\mu_B}(f)  \quad \text{in law}.
 	\eqen
 	Then $\mu_B$ is the regular conditional law of $A$ given $B$.
 \end{lem} 
 
 \begin{lem} \label{prop:cond-walk-limit}
 	Let $\{(A_j , B_j)\}_{j\in\BB N}$ be a sequence of iid pairs of random variables, each of which takes on only countably many possible values, with finite covariance matrix $\Sigma$ satisfying $\det \Sigma \not=0$. Let $\dot\cL_0 = \dot\cR_0 = 0$ and for $m\in\BB N$, let $\dot\cL_m := \sum_{j=1}^m A_j$ and $\dot\cR_m := \sum_{j=1}^m B_j$. Extend $\dot\cL$ and $\dot\cR$ to functions $[0,\infty) \rta \BB R$ by linear interpolation. For $m \in\BB N$ and $t>0$, let
 	\eqbn
 	\wh L^m_t := m^{-1/2} A_{\lfloor m t \rfloor},\quad \wh R^m_t := m^{-1/2} B_{\lfloor m t \rfloor},\quad \wh Z^m_t := (\wh L^m_t ,\wh R^m_t ).
 	\eqen
 	Fix $T>0$, and deterministic sequences of non-negative numbers $T_m \rta T$ and $x_m \rta 0$ such that each $T_m$ is a positive integer multiple of $m^{-1}$ and each $x_m$ belongs to the support of the law of $\dot R^m_{T_m}$. For $m\in\BB N$, let 
 	\eqbn
 	E_m = E_m(x_m ,T_m) := \left\{\dot R^m_{T_m} =  -x_m   \right\} \cap \left\{\dot R^m_t \geq - x_m ,\, \forall t \in [0,T_m] \right\} .
 	\eqen
 	The conditional law of $\wh Z^m|_{[0,T]}$ given $E_m$ converges as $m\rta \infty$ (with respect to the $L^\infty$ metric) to the law of a Brownian excursion in the upper half-plane in time $T$ with covariance matrix $\Sigma$. 
 \end{lem}
 \begin{proof}
 	By applying a linear transformation, we can assume without loss of generality that the off-diagonal elements of $\Sigma$ are 0.  
 	Let $\dot Z = (\dot L,\dot R)$ be a Brownian excursion in the upper half-plane in time $T$ with covariance matrix $\Sigma$. 
 	By Donsker's theorem (resp.~\cite[Theorem 1.1]{sohier-excursion}) the conditional law given $E_m$ of $ \wh L^m|_{[0,T]}$ (resp.\ $\wh R^m|_{[0,T]}$) converges as $m\rta \infty$ to the law of $\dot L$ (resp.\ $\dot R$). 
 	By the theorems of Prokhorov and Skorokhod, for any sequence of positive integers tending to $\infty$, there exists a subsequence $m_k \rta\infty$ and a coupling of a sequence $ \dot Z^{m_k}  = (\dot L^{m_k} , \dot R^{m_k})$ of random paths with the conditional law of $\dot Z^{m_k}|_{[0,T]}$ given $E_{m_k}$ with a random continuous path $\dot Z' = (\dot L',\dot R')$ such that $\dot L' \eqD \dot L$, $\dot R'\eqD \dot R$, and $\dot Z^{m_k} \rta \dot Z'$.
 	Now let $0 < s_1 < s_2 < T$ . By the scaling limit result for random walk bridges~\cite{liggett-bridge} applied to the conditional law of $\dot Z^{m_k}|_{[s_1 , s_2]}$ given $\dot Z^{m_k}|_{[0,s_1  ]}$ and $\dot Z^{m_k}|_{[s_2 ,T]}$ together with Lemma~\ref{prop:cond-law-conv}, the regular conditional law of $\dot Z'|_{[s_1,s_2]}$ given $\dot Z'|_{[0,s_1]}$ and $\dot Z'|_{[s_2,T]}$ is that of a Brownian bridge with covariance matrix $\Sigma$ conditioned to stay in the upper half-plane. Taking a limit as $s_1 \rta 0$ and $s_2 \rta T$ shows that the conditional law of $\dot Z' $ given $\dot Z'_T = (\dot L'_T , 0)$ is what we would expect if we had $\dot Z' \eqD \dot Z$. Since $\dot Z_T' \eqD \dot Z_T$, we infer that in fact $\dot Z \eqD \dot Z'$.  
 \end{proof}

 \subsection{Scaling limit of excursions}
 \label{sec:excursion-conv}
 
In this subsection we will prove the following lemma, which is the main ingredient in the proof of Proposition~\ref{prop:triple-conv}.  
For the statement, we recall the definition of the discontinuous process $W$ from~\eqref{eqn:excursion-process}, which encodes the excursions of $Z$ away from the zero set of $X$.

 \begin{lem} \label{prop:subsequence-conv}
Suppose we are in the setting of Proposition~\ref{prop:triple-conv}. For any sequence of positive integers tending to $\infty$, we can find a subsequence $m_k \rta \infty$ and a coupling $(Z' ,  X  )$ of a random path $Z'$ in $\BB R^2$ with $X$ such that $(Z^{m_k} , X^{m_k}  ) \rta (Z' ,   X   )$ in law, $Z' \eqD Z$, and $((Z'_s  - Z_{\wt\tau_{s} }')_{s\geq 0},X) \eqD (W,X)$. 
 \end{lem}
 
To prove Lemma~\ref{prop:subsequence-conv}, we first need a scaling limit result for the discrete analogues of the times $\wt\tau_s$ and $\tau_s$ and the events $E_s$ appearing in Proposition~\ref{prop:ppp}. 

 In the notation of Section~\ref{sec:excursions}, for $m\in\BB N$ and $s\geq 0$, let $E_s^m$ be the event that $\mcl X_{\lfloor s m \rfloor} \geq 1$. If $E_s^m$ (resp.\ $(E_s^m)^c$) occurs, let $\wt\tau_s^m$ be equal to $m^{-1/2}$ times one plus the largest $i \leq \lfloor s m \rfloor$ for which $\mcl X_{\lfloor s m \rfloor}$ is equal to 0 (resp.\ 1) and let $\tau_s^m$ be equal to $m^{-1/2}$ times the smallest $i \geq \lfloor s m \rfloor+1$ for which $\mcl X_{\lfloor s m \rfloor} $ is equal to 0 (resp.\ 1). Equivalently, $[\wt\tau_s^m , \tau_s^m]$ is obtained by re-scaling the excursion interval of $\lambda'$ away from the north-going discrete flow line started from 0 which contains $\lfloor s m \rfloor$ by $m^{-1/2}$. 
 
 \begin{lem} \label{prop:hitting-time-conv}
 	Suppose we are given a sequence of positive integer $m_k \rta \infty$ such that if we couple $(X^{m_k})_{k \in\BB N}$ with $X$ in such a way that $X^{m_k} \rta X$ a.s., then it is a.s.\ the case that for each $s \geq 0$ for which $X_s \not =0$, we have
 	\eqbn
 	(\tau_s^{m_k} , \wt\tau_s^{m_k} , \BB 1_{E_s^{m_k}} ) \rta (\tau_s , \wt\tau_s , \BB 1_{E_s} ) .
 	\eqen
 \end{lem}
 \begin{proof}
 	If $X_s \not= 0$, then $E_s $ occurs if and only if $X_s > 0$ and $E_s^{m_k}$ occurs if and only if $X_s^{m_k} \geq m_k^{-1/2}$. It follows that a.s.\ $\BB 1_{E_s^{m_k}}  \rta  \BB 1_{E_s}$ for each $s \geq 0$ such that $X_s \not=0$. Now suppose that $s \geq 0$ such that $X_s > 0$. Almost surely, for each such $s$ and each $\ep > 0$ we can find $t_- \in [\wt\tau_s - \ep , \wt\tau_s]$ and $t_+ \in [\tau_s , \tau_s + \ep]$ such that $X_{t_-} < 0$ and $X_{t_+} < 0$. For large enough $m_k \in\BB N$ we have $X^{m_k}_{t_-} < 0$ and $X^{m_k}_{t_+} < 0$. For such an $m_k$ we have $\wt\tau_s^{m_k} \geq t_-$ and $\tau_s^{m_k} \leq t_+$. Since $\ep$ is arbitrary it follows that a.s.\
 	\eqbn
 	\liminf_{k \rta \infty }\wt\tau_s^{m_k} \geq \wt\tau_s \quad \op{and} \quad \limsup_{k \rta \infty } \tau_s^{m_k} \leq  \tau_s .
 	\eqen  
 	If we do not have $\tau_s^{m_k} \rta \tau_s$, then by compactness we can find $t \in (\wt\tau_s , \tau_s)$ and a sequence $k_j \rta\infty$ for which $\tau_s^{m_{k_j}} \rta t$. Then we have $X_{\tau_s^{m_{k_j}}}^{m_{k_j}} \rta X_t$, so $X_t  = 0$. It follows that $X$ attains a local minimum with value 0 at time $t$. Almost surely, there is no $t \geq 0$ for which this is the case. Hence we must have $\tau_s^{m_k} \rta \tau_s$. Similarly $\wt\tau_s^{m_k} \rta \wt\tau_s$. A similar argument shows that a.s.\ $\tau_s^{m_k} \rta \tau_s$ and $\wt\tau_s^{m_k} \rta \wt\tau_s$ whenever $X_s < 0$. 
 \end{proof}

 \begin{proof}[Proof of Lemma~\ref{prop:subsequence-conv}]
 	Since $Z^m \rta Z$ and $X^m \rta X$ in law, the joint laws of the pairs $(Z^{m } , X^{m }   )$ are tight, so for each sequence of positive integers tending to $\infty$, we can find a subsequence $m_k \rta \infty$ and a coupling $(Z' ,   X  )$ of a random path $Z'$ in $\BB R^2$ with $X$ such that $Z' \eqD Z$ and $(Z^{m_k} , X^{m_k}  ) \rta (Z' , X )$ in law. By the Skorokhod theorem we can find a coupling of $\{(Z^{m_k} , X^{m_k}  )\}_{k\in\BB N}$ with $(Z' , X )$ such that $(Z^{m_k} , X^{m_k} ) \rta (Z' , X )$ uniformly on compact intervals a.s.. By Lemma~\ref{prop:hitting-time-conv}, in any such coupling we have
 	\eqb \label{eqn:times-conv}
 	(\tau_s^{m_k} , \wt\tau_s^{m_k} , \BB 1_{E_s^{m_k}} ) \rta (\tau_s , \wt\tau_s , \BB 1_{E_s} ) ,\quad \forall s \geq 0 \: \text{with $X_s \not=0$}.
 	\eqe 
 	
 	Now fix $r \in\BB N$ and suppose given $s_1 , \dots , s_r > 0$. For $m\in\BB N$ let $\mcl G^m = \mcl G^m(s_1,\dots , s_r)$ be the $\sigma$-algebra generated by $\{(\tau_{s_j}^m , \wt\tau_{s_j}^m , \BB 1_{E_{s_j}^m}) \,:\,  j\in [1,r]_{\BB Z} \}$. Also let let $\mcl G = \mcl G(s_1,\dots , s_r)$ be the $\sigma$-algebra generated by $\{(\tau_{s_j}  , \wt\tau_{s_j}  , \BB 1_{E_{s_j}} ) \,:\, j \in [1,r]_{\BB Z}\}$. 
 	
 	In the notation of Section~\ref{sec:excursions}, for $m\in\BB N$ and $j\in [1,r]_{\BB Z}$ the time $\wt\tau_{s_j}^m$ (resp.\ $ \tau_{s_j}^m$) is equal to $m^{-1/2}$ the last (resp.\ first) time of the form $N_k^W$ or $N_k^E$ which comes before (resp.\ after) $\lfloor s_j m \rfloor$. 
 	Let $b := \left( 3\alpha/2 \right)^{1/2}$ be the constant appearing in~\eqref{eqn:rescaled-walk}. 
 	By Lemma~\ref{prop:discrete-excursions} and the strong Markov property, the conditional law given $\mcl G^m$ of each of the paths $(Z^m - Z^m_{\wt\tau_{s_j}^m} , X^m)|_{[\wt\tau_{s_j}^m , \tau_{s_j}^m]}$ for $j \in [1,r]_{\BB Z}$ such that $E_{s_j}^m$ occurs is the same as the conditional law of $(Z^m , R^m + b^{-1} m^{-1/2} ) |_{[0, \tau_s^m  - \wt\tau_s^m]}$ given that $R^m$ hits $0$ for the first time at time $\tau_s^m - \wt\tau_s^m$. The conditional law given $(Z^m - Z^m_{\wt\tau_{s_j}^m} , X^m)|_{[\wt\tau_{s_j}^m , \tau_{s_j}^m]}$ for $j \in [1,r]_{\BB Z}$ such that $(E_{s_j}^m)^c$ occurs is the same as the conditional law of $(Z^m , - L^m) |_{[0, \tau_{s_j}^m  - \wt\tau_{s_j}^m]}$ given that $L^m$ hits $-b^{-1} m^{-1/2}$ for the first time at time $\tau_{s_j}^m - \wt\tau_{s_j}^m$. Furthermore, the paths $(Z^m - Z^m_{\wt\tau_{s_j}^m} , X^m)|_{[\wt\tau_{s_j}^m , \tau_{s_j}^m]}$ for $j \in [1,r]_{\BB Z}$ such that the intervals $[\wt\tau_{s_j}^m , \tau_{s_j}^m]$ are distinct are conditionally independent given $\mcl G^m$. 
 	
 	By Proposition~\ref{prop:ppp}, the conditional law given $\mcl G$ of $(Z - Z_{\wt\tau_{s_j}} , X)|_{[\wt\tau_{s_j} , \tau_{s_j}]}$ for $j \in [1,r]_{\BB Z}$ such that $E_{s_j}$ occurs is the same as the joint law of a Brownian motion with covariance matrix as in~\eqref{eqn:triple-conv-cov} started from 0 and conditioned so that its second coordinate stays positive until time $\tau_{s_j}- \wt\tau_{s_j}$ and hits 0 for the first time at time $\tau_{s_j} - \wt\tau_{s_j}$; and the second coordinate of this Brownian motion. The conditional law given $\mcl G$ of $(Z - Z_{\wt\tau_{s_j}} , X)|_{[\wt\tau_{s_j} , \tau_{s_j}]}$ for $j \in [1,r]_{\BB Z}$ such that $E_{s_j}^c$ occurs is the same as the joint law of a Brownian motion with covariance matrix as in~\eqref{eqn:triple-conv-cov} started from 0 and conditioned so that its first coordinate stays positive until time $\tau_{s_j} - \wt\tau_{s_j}$ and hits 0 for the first time at time $\tau_{s_j} - \wt\tau_{s_j}$; and $-1$ times the first coordinate of this Brownian motion. Furthermore, the paths $(Z - Z_{\wt\tau_{s_j}} , X)|_{[\wt\tau_{s_j} , \tau_{s_j}]}$ for $j \in [1,r]_{\BB Z}$ such that the intervals $[\wt\tau_{s_j}  , \tau_{s_j} ]$ are distinct are conditionally independent given $\mcl G$.  
 	
 	It is a.s.\ the case that $X_{s_1} , \dots , X_{s_r}$ are all non-zero. By Lemma~\ref{prop:cond-walk-limit},~\eqref{eqn:times-conv}, and the above characterization of conditional laws given $\mcl G^m$ and $\mcl G$, for any coupling of $\{(Z^{m_k} , X^{m_k})\}_{k\in\BB N}$ with $(Z' , X)$ as above, it holds that the joint conditional laws of $\{(Z^{m_k} - Z^{m_k}_{\wt\tau_{s_j}^{m_k}} , X)|_{[\wt\tau_{s_j}^{m_k} , \tau_{s_j}^{m_k}]} \}_{j\in [1,r]_{\BB Z}} $ given $\mcl G^{m_k}$ converge a.s.\ as $k\rta \infty$ to the joint conditional law of $\{(W , X)|_{[\wt\tau_{s_j} , \tau_{s_j}]} \}_{j\in [1,r]_{\BB Z}}$ given $\mcl G$. 
 	By Lemma~\ref{prop:cond-law-conv}, the joint conditional laws given $\mcl G$ of
 	\eqbn
 	\left\{(Z'  - Z_{\wt\tau_{s_j} }' , X)|_{[\wt\tau_{s_j} , \tau_{s_j}]} \right\}_{j\in [1,r]_{\BB Z}} \quad \op{and} \quad \left\{(W , X)|_{[\wt\tau_{s_j} , \tau_{s_j}]} \right\}_{j\in [1,r]_{\BB Z}}  
 	\eqen
 	agree. 
 	Since $Z'$, $W$, and $X$ are right continuous (hence determined by their values at countably many times) and this equality of conditional laws holds a.s.\ for any fixed $s_1 , \dots ,s_r  > 0$, the statement of the lemma follows. 
 \end{proof}
 
 \subsection{Proof of Proposition~\ref{prop:triple-conv}}
 \label{sec:triple-conv-proof}

 Suppose given a subsequence $m_k\rta\infty$ and a coupling $(Z' ,  X  )$ of a random path $Z'$ in $\BB C$ with $X$ satisfying the conditions Lemma~\ref{prop:subsequence-conv}.  It is immediate from the conditions on $Z'$ in Lemma~\ref{prop:subsequence-conv} that we can find a coupling $(Z, Z' , X  )$ for which a.s.\
 \eqb \label{eqn:excursion-agree}
 Z'_s  - Z_{\wt\tau_{s} }' = W_s ,\quad \forall s \geq 0,
 \eqe  
 $Z$ and $Z'$ are conditionally independent given $X$ and $W$, and $(Z,X)$ has the law described in the beginning of this section. Indeed, such a coupling can be produced by first sampling $(Z' , X)$ and then sampling $Z$ from its conditional law given that $(W_s)_{s\geq 0} = (Z'_s  - Z_{\wt\tau_s}'    )_{s\geq 0}$ (under which it is a.s.\ determined by $W$ by Proposition~\ref{prop:bm-determined}). 
 We want to show that in any such coupling we have $Z = Z'$ a.s. 
 By our choice of coupling, we know that the excursions of $Z$ and $Z'$ away from the set $\mcl A$ of Proposition~\ref{prop:ppp} (which are encoded by $W$) agree, so this amounts to ruling out the possibility that these excursions are glued together in two different manners to get $Z$ and $Z'$. 
 We will accomplish this using a similar argument to that used in the proof of Proposition~\ref{prop:bm-determined}, but before we can do so we need to prove a few lemmas about the process $Z'$ in our coupling.
 
 \begin{lem} \label{prop:limit-inc}
 	Suppose we are given a sequence $m_k \rta \infty$ and a coupled pair of Brownian motions $(Z' , X )$ such that $(Z^{m_k} , X^{m_k} ) \rta (Z' , X )$ in law. Then for each $t\geq 0$ the conditional law of $(Z' - Z_t')|_{[t,\infty)}$ given $(Z' , X )|_{[0,t]}$ is the same as the law of $Z $. 
 \end{lem}
 \begin{proof} 
 	By Lemma~\ref{prop:discrete-excursions}, the one-dimensional random walk $\mcl X$ is adapted to the filtration generated by the two-dimensional random walk $\mcl Z$. Since $\mcl Z$ is a random walk with iid increments, we find that for each $t \geq 0$, the walks $(\mcl Z - \mcl Z_{\lfloor t m \rfloor})|_{[\lfloor t m \rfloor,\infty)_{\BB Z}}$ and $(\mcl Z , \mcl X)|_{[0,\lfloor t m \rfloor]_{\BB Z}}$ are independent. By passing to the scaling limit along our subsequence $m_k$, we find that for each $t\geq 0$, $(Z' - Z_t')|_{[t,\infty)}$ and $(Z' , X')|_{[0,t]}$ are independent. Since $Z' \eqD Z$, the statement of the lemma follows from the Markov property of Brownian motion.
 \end{proof}
 
 We also have some further Markov properties of the coupling $(Z , Z' , X)$ described at the beginning of this subsection, this time concerning the interaction of $Z'$ with the excursion process $W$.

 \begin{lem} \label{prop:Z'-markov}
 	Suppose given a subsequence $m_k \rta \infty$ and a coupling $(Z , Z' , X)$ as in the beginning of this subsection. Then we have the following Markov properties for each $t\geq 0$. The conditional law of $W|_{[t,\tau_t]}$ given $(Z' , W)|_{[0,t]}$ is that of a Brownian motion with covariance matrix as in~\eqref{eqn:triple-conv-cov} started from $W_t$, run until the first time its first (resp.\ second) coordinate hits 0 if $E_t^c$ (resp.\ $E_t$) occurs; and the conditional law of $(Z' - Z_{\tau_t}' , W)|_{[\tau_t , \infty)}$ given $(Z' , W )|_{[0,\tau_t]}$ is the same as the unconditional law of $(Z' , W)$.
 \end{lem} 
 \begin{proof} 
 	By the Skorokhod theorem we can find a coupling of $\{(Z^{m_k} , X^{m_k})\}_{k \in\BB N}$ with $(Z , Z' , X)$ for which $Z^{m_k} \rta Z'$ and $X^{m_k} \rta X$ uniformly a.s.\ on each compact interval. 
 	
 	Lemma~\ref{prop:discrete-excursions} implies that for each $k\in\BB N$ and $t > 0$, the conditional law of $(Z^{m_k} - Z^{m_k}_{\wt\tau_t^{m_k}} )|_{[ m_k^{-1} \lfloor t m_k  \rfloor , \tau_t^{m_k} ]}$ given $(Z^{m_k}  , X^{m_k})|_{[0, m_k^{-1} \lfloor t m_k \rfloor]}$ is the same as the law of $Z^{m_k}$ started from $Z^{m_k}_{m_k^{-1} \lfloor t m_k \rfloor} - Z^{m_k}_{\wt\tau_t^{m_k}}$ and run until the first time that its first (resp.\ second) coordinate hits $ -m_k^{-1/2}$ if $(E_t^{m_k})^c$ (resp.\ $E_t^{m_k}$) occurs. Furthermore, the conditional law of $(Z^{m_k} - Z^{m_k}_{   \tau_t^{m_k}  } , X^{m_k} - m_k^{-1/2} \BB 1_{E_t^{m_k}} )|_{[\tau_t^{m_k} , \infty)}$ given $(Z^{m_k}  , X^{m_k})|_{[0, \tau_t^{m_k} ]}$ is the same as the unconditional law of $(Z^{m_k } , X^{m_k})$. 
 	
 	By Lemma~\ref{prop:hitting-time-conv}, we a.s.\ have $(\tau_s^{m_k} , \wt\tau_s^{m_k} , \BB 1_{E_s^{m_k}} ) \rta (\tau_s , \wt\tau_s , \BB 1_{E_s} )$. It therefore follows from Lemma~\ref{prop:cond-law-conv} that the conditional law of $W |_{[t,\tau_t]}$ given $(Z' , X)|_{[0,t]}$ is that of a Brownian motion with covariance matrix as in~\eqref{eqn:triple-conv-cov} started from $W_t$, run until the first time its first (resp.\ second) coordinate hits 0 if $E_t^c$ (resp.\ $E_t$) occurs; and the conditional law of $(Z' - Z_{\tau_t}' , X)|_{[\tau_t , \infty)}$ given $(Z' , X)|_{[0,\tau_t]}$ is the same as the unconditional law of $(Z' , X)$. It is clear from our choice of coupling and the definition~\eqref{eqn:excursion-process} of $W$ that $(Z',X)|_{[0,t]}$ a.s.\ determines $(Z,W)|_{[0,t]}$, $(Z' , X)|_{[0,\tau_t]}$ a.s.\ determines $(Z' , W)|_{[0,\tau_t]}$, and $(Z' - Z_{\tau_t}' , X)|_{[\tau_t , \infty)}$ a.s.\ determines $(Z' - Z_{\tau_t}' , W)|_{[\tau_t , \infty)}$ in the same manner that $(Z' , X)$ determines $(Z' , W)$. Furthermore, the argument of Lemma~\ref{prop:excursion-msrblty} shows that $(Z' , W)|_{[0,t]}$ a.s.\ determines whether or not the event $E_t$ occurs. The statement of the lemma follows. 
 \end{proof}

 Proposition~\ref{prop:triple-conv} will follow easily from the following lemma together with Lemma~\ref{prop:2bm-same-excursion}.

 \begin{lem} \label{prop:2bm-filtration}
 	Suppose given a subsequence $m_k \rta \infty$ and a coupling $(Z , Z' , X)$ as in the beginning of this subsection. For $t\geq 0$, let
 	\eqbn
 	\mcl F_t := \sigma\left((Z' , Z  )|_{[0,t]} \right) .
 	\eqen
 	Then for each $t\geq 0$, the conditional law of $(Z-Z_t)|_{[t,\infty)}$ given $\mcl F_t$ is the same as the law of $Z$ and the conditional law of $(Z' - Z_t')|_{[t,\infty)}$ given $\mcl F_t$ is the same as the law of $Z$. In particular, for each $t_2 \geq t_1 \geq  0$, we have
 	\eqbn
 	\BB E\left[Z_{t_2} - Z_{t_2}' \,|\, \mcl F_{t_1} \right] = Z_{t_1} - Z_{t_1}' .
 	\eqen
 \end{lem}
 \begin{proof}
 	The proof is similar to that of Lemma~\ref{prop:2bm-same-excursion}, but slightly more work is needed due to the fact that we do not know $(Z' , X)\eqD (Z,X)$. 
 	
 	By Proposition~\ref{prop:bm-determined}, $Z|_{[0,t]}$ and $W|_{[0,t]}$ a.s.\ determine each other, so the conditional law of $(Z' - Z'_t)|_{[t,\infty)}$ given $\mcl F_t$ is the same as its conditional law given $(Z' , W)|_{[0,t]}$.
 	By Lemma~\ref{prop:limit-inc}, the conditional law of $(Z' - Z_t')|_{[t,\infty)}$ given $(Z' , X)|_{[0,t]}$ is the same as the unconditional law of $Z'$, so since $W|_{[0,t]}$ is determined by $(Z' , X)|_{[0,t]}$, the conditional law of $(Z' - Z_t')|_{[t,\infty)}$ given $(Z' , W)|_{[0,t]}$ is the same as the unconditional law of $Z'$. By combining this with the above, we find that the conditional law of $(Z' - Z_t')|_{[t,\infty)}$ given $\mcl F_t$ is the same as the unconditional law of $Z'$, which is the same as the law of $Z$. 
 	
 	Next we will compute the conditional law of $(Z-Z_t)|_{[t,\infty)}$ given $\mcl F_t$ via a similar argument. 
 	By Lemma~\ref{prop:Z'-markov}, the conditional law of $W|_{[t,\infty)}$ given $(Z' ,W)|_{[0,t]}$ depends only on $W|_{[0,t]}$. By our choice of coupling, the conditional law of $Z$ given $Z'|_{[0,t]}$ and $W$ depends only on $W$. Hence the conditional law of $Z $ given $(Z' , W)|_{[0,t]}$ is the same as its conditional law given only $W|_{[0,t]}$. By Lemma~\ref{prop:excursion-msrblty}, $Z|_{[0,t]}$ a.s.\ determines $W|_{[0,t]}$, so the conditional law of $(Z - Z_t)|_{[t,\infty)}$ given $\mcl F_t$ is the same as its conditional law given only $Z|_{[0,t]}$, which by the Markov property is the same as the law of $Z$. 
 \end{proof}

 \begin{proof}[Proof of Proposition~\ref{prop:triple-conv}]
 	Lemma~\ref{prop:subsequence-conv} and the discussion immediately thereafter implies that for each sequence of positive integers tending to $\infty$, we can find a subsequence $m_k\rta \infty$ and a coupling $(Z , Z' , X)$ (equivalently, $(Z  , Z' , W)$) of $(Z, X)$ with another Brownian motion $Z' \eqD Z$ such that $Z_s - Z_{\wt\tau_s} = W_s$ for each $s \geq 0$, $Z$ and $Z'$ are conditionally independent given $W$, and $(Z^{m_k} , X^{m_k}   ) \rta (Z' ,   X  )$ in law. By Lemmas~\ref{prop:2bm-filtration} and~\ref{prop:2bm-same-excursion}, we a.s.\ have $Z = Z'$, so $(Z' , X) \eqD (Z , X)$. 
 \end{proof}

\section{Joint convergence of two Peano curves}
\label{sec:peano-conv}

In this section we will deduce the statement of our main theorem from Proposition~\ref{prop:triple-conv}. Before we can do so however, we need to deal with a certain technical difficulty. Recall the discrete space-filling path $\lambda'$ which traces all of the edges of the uniform infinite bipolar-oriented triangulation $(G, \cO,\BB e)$. Let $\wt\lambda'$ be the space-filling path associated with the dual bipolar-oriented map $(\wt G, \wt\cO,\wt{\BB e})$. We need to show that for $t \in \BB R$, it is unlikely that $\wt\lambda'$ takes more than $O_m(m )$ units of time to reach the edge of $\wt G$ which crosses $\lambda'(\lfloor t m \rfloor)$. This latter amount of time can be expressed in terms of the area between two pairs of discrete north-going and south-going flow lines. We will prove tightness of this area in Section~\ref{sec:merge-tight} using a combinatorial argument before proceeding to the proof of Theorem~\ref{thm1} in Section~\ref{sec:fourtrees-conv}.

\subsection{Tightness of the area between discrete flow lines}
\label{sec:merge-tight}

Let $(G , \cO,\BB e)$ be a uniform infinite bipolar-oriented triangulation and let $(\wt G , \wt\cO,\wt{\BB e})$ be its dual.  
Recall the definition of the space-filling exploration path $\lambda'$ on $G$ (which is a function from $\BB Z$ to the edge set of $G$) from Section~\ref{sec:bipolar-background}. 
For $i\in\BB Z$, let $\lambda_i^{E}$, $\lambda_i^{W}$, $\lambda_i^{N}$, and $\lambda_i^{S}$ be the discrete east-going, west-going, north-going, and south-going, respectively, flow lines (i.e.\ infinite tree branches) started from $\lambda'(i)$ (or the edge of $\wt G$ which crosses $\lambda'(i)$).  

For $i_1,i_2 \in\BB Z$ with $i_1 < i_2 $, let $F_{i_1,i_2}^N$ be the event that the edge of $\lambda_{i_1}^{N}$ which crosses the outer boundary of 
$\lambda'((-\infty ,i_2]_{\BB Z})$ lies on $\lambda_{i_2}^{E}$. Equivalently, $F_{i_1,i_2}^N$ 
is the event that the part of $\lambda_{i_1}^{N}$ not crossed by 
$\lambda'$ before time $i_2$ lies to the east of $\lambda_{i_2}^N$. 
Symmetrically, we let $F_{i_1,i_2}^S$ be the event that the edge of $\lambda_{i_2}^{S}$ which crosses the outer 
boundary of $\lambda'([i_1 ,\infty)_{\BB Z})$ crosses an edge of $\lambda_{i_2}^{E}$. Also let $\iota_{i_1,i_2}^{N}$ 
(resp.\ $\iota_{i_1,i_2}^S$) be the time at which $\lambda'$ 
crosses the edge at which $\lambda_0^{N}$ and $\lambda_i^{N}$ 
(resp.\ $\lambda_0^{S}$ and $\lambda_i^{S}$) merge. See Figure~\ref{fig-merge-walk} for an illustration of the event $F^N_{0,i}$. 

Suppose $i_1,i_2 \in\BB Z$ with $i_1 < i_2 $. If $F_{i_1,i_2}^N$ (resp.\ $(F_{i_1,i_2}^N)^c$) occurs,
let $M_{i_1,i_2}^N$ be the set of edges of $G$ which lie in the intersection of the west (resp.\ east) side of $\lambda_{i_1}^N$, the east (resp.\ west) side of $\lambda_{i_2}^N$, and $\lambda'([i_2,\infty)_{\BB Z})$, where here, as in Section \ref{sec:one-dual-tree}, we say that an edge is to the west (resp.\ east) of a flow line if both end-points (resp.\ at least one end-point) of the edge is to the west (resp.\ east) of the flow line. We define $M_{i_1,i_2}^S$ similarly but with the roles of $i_1$ and $i_2$ interchanged; ``$N$'' in place of ``$S$'', and $\lambda'([i_1,\infty)_{\BB Z})$ in place of $\lambda'((-\infty , i_2]_{\BB Z})$. For $i\in\BB Z$ we set $M_{i,i}^N = M_{i,i}^S = \emptyset$.

\begin{figure}[ht!]
	\begin{center}
		\includegraphics{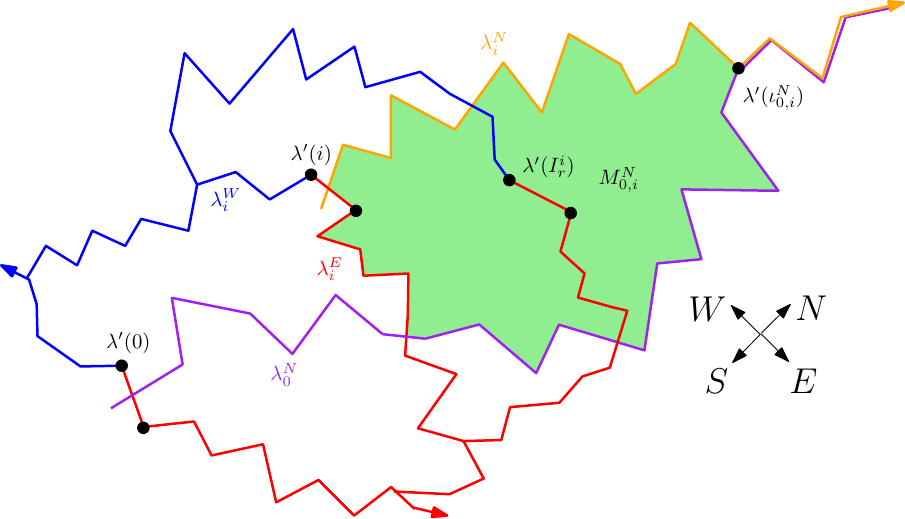} 
		\caption{An illustration of the setup and proof of Proposition~\ref{prop:merge-tight} on the event $F_{0,i}^N$. The picture on $(F_{0,i}^N)^c$ is similar, except that $\lambda_0^{N}$ crosses the left (blue) boundary of $\lambda'([0,i]_{\BB Z})$ rather than the right (red) boundary. The light-green region is bounded by the curve $\lambda_i^{E}$, $\lambda_i^{N}$, and $\lambda_0^{N}$, the latter two curves run up to the first time when they merge. It includes edges of the map $G$ which have both endpoints in the light-green region or which cross the west boundary of the green region.
			We want to prove an upper bound for the number of edges in $ M_{0,i}^N $. To this end, we consider the walk $\{\mcl Y^i_r\}_{r\in [ 0, \#M_{0,i}^N]_{\BB Z}}$ which is parameterized by the times $I_r^i$ for which $ I_r^i  \in M_{0,i}^N$ and which, at each time $r$, give the total number of edges of the outer boundary of $\lambda'((-\infty , I_r^i]_{\BB Z})$ which lie in $ M_{0,i}^N $. This walk starts from $d^i$, which is the distance between $\lambda_0^{N}$ and $\lambda'(i)$ along the outer boundary of $\lambda'((-\infty , i]_{\BB Z})$, and hits $-1$ for the first time when $\lambda'$ finishes filling in all of the edges of $ M_{0,i}^N $. We compute the law of this walk in Lemma~\ref{prop:walk-law}.}\label{fig-merge-walk}
	\end{center}
\end{figure}

In this subsection we will prove the following proposition, which is needed for the proof of Theorem~\ref{thm1}. 

\begin{prop} \label{prop:merge-tight}
	For each fixed $s_1,s_2 \in \BB R$ with $s_1 < s_2$, the laws of the random variables $\{m^{-1 } \# M_{\lfloor s_1 m \rfloor , \lfloor s_2 m\rfloor}^N\}_{m\in\BB N}$ and $\{m^{-1 } \# M_{\lfloor s_1 m \rfloor , \lfloor s_2 m\rfloor}^S\}_{m\in\BB N}$ are tight. 
\end{prop} 

The reason why we need Proposition~\ref{prop:merge-tight} is as follows. If $\wt\lambda'$ denotes the space-filling exploration path of the dual map $\wt G$, as defined in Section~\ref{sec:bipolar-infinite}, normalized so that $\wt\lambda'(0)$ is the edge which crosses $\lambda'(0)$, then the time between 0 and when $\wt\lambda'$ hits $\lambda'(i)$ is equal to the number of edges in the region disconnected from $\infty$ by the flow lines $\lambda_0^{N}$, $\lambda_0^{S}$, $\lambda_i^{N}$, and $\lambda_i^{S}$. Proposition~\ref{prop:merge-tight} will allow us to show that $\wt\lambda'$ is unlikely to take more than $O_m(m )$ units of time to hit $\lambda'(\lfloor s m\rfloor)$.

By forward/reverse symmetry of the law of $\lambda'$ and translation invariance, to prove Proposition~\ref{prop:merge-tight} it will suffice to prove tightness of $\{m^{-1 } \# M_{0 , \lfloor s   m\rfloor}^N\}_{m\in\BB N}$ for fixed $s>0$. We will do this by bounding $\# M_{0,i}^N$ in terms of the first time that a certain random walk hits 0. 
To this end, define $N_k^{i, E}$ and $N_k^{i, W}$ for $i , k \in\BB N$ in the same manner as the times $N_k^E$ and $N_k^W$ of Section~\ref{sec:excursions} but starting from $i$ instead of from 0, so that $N_k^{i,W}$ and $N_k^{i,E}$ are the times when $\lambda'$ crosses $\lambda_i^N$. Note that in this notation we have $N_k^E  = N_k^{0,E}$ and $N_k^W = N_k^{0,W}$.  
For $r \in \BB N_0$, let
\eqbn
I^i_r := \inf\left\{j \in [i  ,\infty)_{\BB Z} \,:\, \#\left( [i+1 , j]_{\BB Z} \cap (\lambda')^{-1}( M_{0,i}^N ) \right)  = r \right\}  
\eqen
be the $r$th smallest time $j\in\BB N$ at which $\lambda'(j) \in M_{0,i}^N$. 
Define
\eqbn
n_r^i :=  
\sup\left\{ j \in [i , I_r^i]_{\BB Z}  \,:\, j = N_k^{i,q} \: \text{for some $k\in\BB N$}  \right\}  ,
\eqen
where $q =E$ if $F_{0,i}^N$ occurs and $q = W$ if $(F_{0,i}^N)^c$ occurs. Also define
\eqbn
\wt n_r^i := \sup\left\{ j \in [0 , I_r^i]_{\BB Z}  \,:\, j = N_k^{0,\wt q} \: \text{for some $k\in\BB N$}  \right\}  ,
\eqen
where $\wt q =W$ if $F_{0,i}^N$ occurs and $\wt q = E$ if $(F_{0,i}^N)^c$ occurs. Since each edge $\lambda'(I_r^i)$ belongs to $M_{0,i}^N$, $n_r^i$ (resp.\ $\wt n_r^i$) is the last time before $i$ at which $\lambda'$ crosses $\lambda_i^{N}$ (resp.\ $\lambda_0^{N}$). 
If $I_r^i < \infty$, let
\eqbn
\mcl Y_r^i := 
\begin{cases}
	&\mcl L_{I_r^i} - \mcl L_{n_r^i}  +\mcl R_{I_r^i} - \mcl R_{\wt n_r^i}  ,\quad \text{if $F_{0,i}^N$ occurs}\\
	&\mcl R_{I_r^i} - \mcl R_{n_r^i}  +\mcl L_{I_r^i} - \mcl L_{\wt n_r^i} ,\quad \text{if $(F_{0,i}^N)^c$ occurs}\\
\end{cases}
\eqen
and otherwise let $\mcl Y_r^i = -1$. 
Then $\mcl Y_r^i$ is equal to the total number of edges on the outer boundary of $\lambda'(I_r^i)$ which are contained in $M_{0,i}^N$.  

Let $d^i$ be the number of edges on the outer boundary of $\lambda'((-\infty ,i]_{\BB Z})$ between $\lambda'(i)$ and the edge on the outer boundary of $\lambda'((-\infty ,i]_{\BB Z})$ crossed by $\lambda_0^{N}$.

\begin{lem} \label{prop:walk-law}
	The conditional law of $\mcl Y^i$ given $\mcl Z|_{(-\infty , i]}$ is that of a random walk with iid increments, each uniformly distributed on $\{-1,0,1\}$, started from $d^i$ and stopped at the first time it hits $-1$. In particular, the conditional law of $\# M_{0,i}^N$ given $\mcl Z|_{(-\infty , i]}$ is that of the first time  such a random walk hits $-1$. 
\end{lem}
\begin{proof}
	Let $\{\mcl F_j\}_{j\in\BB Z}$ be the filtration generated by $\mcl Z$. It follows from Lemma~\ref{prop:discrete-excursions} that each of the times $I_r^i$ is a stopping time for $\{\mcl F_j\}_{j\in\BB Z}$. Furthermore, the event $F_{0,i}^N$ is measurable with respect to $\mcl F_i$ (it is determined by the sign of $\mcl X_i$ with $\mcl X$ as in Lemma~\ref{prop:discrete-excursions}). Consequently, $\mcl Y^i$ is adapted to the filtration $\{\mcl F_{I_r^i}\}_{r\in\BB N_0}$. 
	
	We will prove that the conditional law of $\mcl Y_r^i$ given $\mcl F_i$ is as described on the event $F_{0,i}^N$. The case when $(F_{0,i}^N)^c$ occurs is treated similarly. Note that on $F_{0,i}^N$, we have 
	\eqbn
	d^i = \mcl R_i - \mcl R_{n_0^i} = \mcl Y_0^i .
	\eqen
	If we condition on a realization of $\mcl F_{I_r^i}$ for $r\in\BB N$ such that $F_{0,i}^N$ occurs, then the conditional law of $\mcl Z^i_{I_r^i + 1} - \mcl Z^i_{I_r^i}$ is uniform on $\{(1,0) , (0,-1) , (-1,1)  \}$. To determine the conditional law of $\mcl Y_{r+1}^i -\mcl Y_r^i$ we consider four cases, according to whether $\mcl L^i_{I_r^i} - \mcl L^i_{n_r^i}$ (resp.\ $\mcl R_{I_r^i} - \mcl R_{\wt n_r^i}$) is zero or non-zero (equivalently whether the edge $\lambda'(I_r^i)$ is or is not crossed by an edge of $\lambda_i^{N}$ and whether $\lambda'(I_r^i)$ is or is not adjacent to an edge of $\lambda'((-\infty , i]_{\BB Z})$ which is crossed by an edge of $\lambda_0^{N}$). In the case when both quantities are non-zero, $I_{r+1}^i = I_r^i + 1$, and it is clear that the conditional law of $\mcl Y_{r+1}^i - \mcl Y_r^i$ is uniform on $\{-1,0,1\}$. In the two cases when one coordinate is zero but the other is not, it holds with conditional probability $1/3$ that $\lambda'(I_r^i + 1) \notin  M_{0,i}^N$, in which case $I_{r+1}^i$ is equal to the next time after $I_r^i+1$ of the form $N_k^{i,W}$ or $N_k^{0,E}$ for $k\in\BB N$ (depending on which of the two quantities is zero). It is easy to see from Lemma~\ref{prop:discrete-excursions} that the conditional law of $\mcl Y_{r+1}^i - \mcl Y_r^i$ is again uniform on $\{-1,0,1\}$ in each of these two cases. 
	
	In the case when $\mcl L^i_{I_r^i} - \mcl L^i_{n_r^i} = \mcl R_{I_r^i} - \mcl R_{\wt n_r^i} = 0$, we have $I_r^i = I_{r+1}^i$ and $\mcl Y_{r+1}^i - \mcl Y_r^i = 1$ if $\mcl Z^i_{I_r^i + 1} - \mcl Z^i_{I_r^i} = (1,0)$ and $\lambda'(I_r^i + 1) \notin M_{0,i}^N$ otherwise. If $\mcl Z^i_{I_r^i + 1} - \mcl Z^i_{I_r^i} = (-1,1)$, then it follows from Lemma~\ref{prop:discrete-excursions} (recall the asymmetry in the definitions of $N_k^{0,W}$ and $N_k^{0,E}$) that $\lambda'$ enters the region strictly to the left of $\lambda_i^N$ at time $I_r^i + 1$, the time $I_{r+1}^i$ is finite, the edge $\lambda'(I_{r+1}^i)$ does not cross $\lambda_0^N$, and $\mcl Y_{r+1}^i - \mcl Y_r = 0$. If $\mcl Z^i_{I_r^i + 1} - \mcl Z^i_{I_r^i} = (0,-1)$, then the edge $\lambda'(I_r^i + 1)$ crosses the flow line $\lambda_0^N$. Furthermore, at the next time after $I_r^i+1$ when $\lambda'$ lies to the west of $\lambda_0^N$, it also lies to the west of $\lambda_i^N$. In other words, $\lambda'$ crosses both flow lines simultaneously before re-entering $M_{0,i}^N$, which implies $I_{r+1}^i = \infty$ and $\mcl Y_{r+1}^i - \mcl Y_r^i = -1$. 
\end{proof}

\begin{proof}[Proof of Proposition~\ref{prop:merge-tight}]
	By symmetry we only need to consider the sets $M_{0,\lfloor s m \rfloor}^N$ for $m\in\BB N$ when $s > 0$. Henceforth fix such an $s$ and an $\ep > 0$. For $m\in\BB N$, let $d^{\lfloor s m\rfloor}$ be as defined just above Lemma~\ref{prop:walk-law} with $i = \lfloor s m \rfloor$. Then $d^{\lfloor s m \rfloor}$ is at most the number of edges which lie on the outer boundary of $\lambda'((-\infty , \lfloor s m \rfloor]_{\BB Z})$ but not on the outer boundary of $\lambda'((-\infty , 0]_{\BB Z})$. Therefore $d^{\lfloor s m \rfloor} \leq |\mcl Z_{\lfloor s m\rfloor} |$. Since $m^{-1/2} |\mcl Z_{\lfloor s m \rfloor} |$ converges in law to $|Z_s|$, we can find a constant $  \wt C_\ep  > 0$ (depending on $\ep$ and $s$ but not on $m$) such that $\BB P\left[d^{\lfloor s m \rfloor} > \wt C_\ep m^{1/2} \right] \leq \ep$ for each $m\in\BB N$. By Lemma~\ref{prop:walk-law}, on the event $\{d^{\lfloor s m \rfloor} \leq \wt C_\ep m^{1/2}\}$, the conditional law of $\# M_{0,\lfloor s m \rfloor}^N$ given $\mcl Z|_{(-\infty , \lfloor s m \rfloor]}$ is that of the first time that a random walk with iid steps uniformly distributed on $\{-1,0,1\}$ started from 0 hits $- \wt C_\ep m^{1/2}-1$. By Donsker's theorem, we can find $C_\ep > 0$ depending only on $\wt C_\ep$ such that the probability that this hitting time is $> C_\ep m$ is at most $\ep$. Consequently, 
	\eqbn
	\BB P\left[ \# M_{0,\lfloor s m \rfloor}^N > C_\ep m \right] \leq 2 \ep .
	\eqen
	This proves the desired tightness.
\end{proof}

\subsection{Convergence of two Peano curves}
\label{sec:fourtrees-conv}

Recall the process $\cl$ from Section \ref{sec:excursions}. For $m\in\N$ let $\ell^m=(\ell^m_t)_{t\geq 0}$ be the renormalized version of $\cl$, which we define as follows for discrete times 
\eqb
\ell^m_{t}=(3\alpha/2)^{1/2}m^{-1/2}\cl_{tm},\qquad tm\in\Nz,
\label{eqn:ellm}
\eqe
and by linear interpolation for $tm\not\in\Nz$, where $\alpha>0$ is as in~\eqref{eqn:correl}. Recall also the definition of $X^m$ from~\eqref{eqn:rescaled-X} and the definition of $X$ from~\eqref{eqn:signed-bm}. Let $\ell$ be the local time of $X$ at 0, as in Section~\ref{sec:continuum-decomp}. 

\begin{lem}
	In the notation above, the following convergence holds in law for the topology of uniform convergence on compact sets:
	\eqbn
	(X^m,\ell^m) \rtaD \left(X,\frac 32 \ell\right).
	\eqen 
	\label{prop10}
\end{lem}
\begin{proof}
	Define $\wh X^m:=a^{-1} X^m$ and $\wh \ell^m:=a^{-1} \ell^m$, where $a:=(3\alpha/2)^{1/2}$. Then $\wh X^m$ converges in law to $a^{-1} X$, which has the law of a standard Brownian motion multiplied by $(2/3)^{1/2}$. By \cite[Theorem 1.1]{bor81-loctime} the pair $(\wh X^m,\wh \ell^m)$ converges in law towards $(a^{-1} X, \frac 32 \ell')$, where $\ell'$ is the local time at zero of $a^{-1}X$, and the factor $\frac 32$ is added since \cite{bor81-loctime} uses another definition of local time than us: Given a process $Y$ with the law of a constant multiple of a standard Brownian motion, Borodin defines its local time at zero by $t\mapsto \lim_{\ep\rta 0}\int_0^t \1_{0\leq Y_s< \ep}\,ds$, while we use the definition applied in e.g.\ \cite{revuz-yor}, i.e.\ $t\mapsto \lim_{\ep\rta 0}\int_0^t \1_{0\leq Y_s< \ep}\,d\langle Y,Y \rangle_s$. With our definition of local time $\ell'=a^{-1}\ell$. It follows that $(X^m,\ell^m)=(a\wh X^m,a\wh \ell^m)$ converges in law to $(X,\frac 32 \ell)$.
\end{proof}

For $q\in\Q$ define $X^q$ and $\ell^q$ just as we defined $X$ and $\ell$, except that we consider the recentered quantum cone and SLE pair $(h(\cdot+\eta'(q)), \eta'(\cdot+q)-\eta'(q))$ in place of $(h ,\eta')$. The pair $(X^q,\ell^q)$ is equal in law to $(X,\ell)$ by recentering invariance of the SLE-decorated quantum cone~\cite[Theorem 1.9]{wedges}, and the pair encodes information about the north-going flow line started from $\eta'(q)$ of the whole-plane GFF used to generate $\eta'$. We extend the processes $X^q$ and $\ell^q$ to negative times by letting $(X^q_{-t})_{t\geq 0}$ and $(\ell^q_{-t})_{t\geq 0}$ be the analogues of $X^q$ and $\ell^q$ for the south-going flow line started from $\eta'(q)$. More precisely, $(X^q_{-t},\ell^q_{-t})_{t\geq 0}=(\wh X^{-q}_{t},\wh\ell^{-q}_{t})_{t\geq 0}$, where $(X^{\wh q},\wh \ell^{\wh q})$ for $\wh q\in\Q$ is defined such that $(\wh L,\wh R,\wh X^{\wh q},\wh \ell^{\wh q})\eqD(L,R,X^{\wh q},\ell^{\wh q})$ for $(\wh L,\wh R):=(L_{-t},R_{-t})_{t\in\R}$.

For $m\in\N$ and $q\in\Q$ define $\cX^{q,m}=(\cX_n^{q,m})_{n\in\N}$ (resp., $\cl^{q,m}=(\cl_n^{q,m})_{n\in\N}$, $X^{q,m}=(X_t^{q,m})_{t\geq 0}$, $\ell^{q,m}=(\ell_t^{q,m})_{t\geq 0}$) just as we defined $\cX$ (resp., $\cl$, $X^m$, $\ell^m$), except that we consider the bipolar-oriented map encoded by the recentered random walk $(Z_{n+\lfloor m q \rfloor})_{n\in\N}$. Note that the law of $\cX^{q,m}$ (resp., $\cl^{q,m}$, $X^{q,m}$, $\ell^{q,m}$) is independent of the value of $q$ by invariance in law of the infinite volume bipolar-oriented map under recentering. We extend the definitions of $\cX^{q,m},\cl^{q,m},X^q,$ and $\ell^q$ to negative times by letting $(\cX^{q,m}_{-n})_{n\in\BB N_0}$ and $(\cl^{q,m}_{-n})_{n\in\BB N_0}$ be the analogues of $\cX^{q,m}$ and $\cl^{q,m}$ with the discrete south-going flow line started from $\lambda'(\lfloor m q \rfloor)$ in place of the discrete north-going flow line started from $\lambda'(0)$.

\begin{prop}
	With the notation introduced above the following convergence holds in law for the topology of uniform convergence on compact intervals for every finite collection of times $q\in \BB Q$
	\eqb
	\begin{split}
		(Z^m,(X^{q,m})_{q\in\Q},(\ell^{q,m})_{q\in\Q}) \rtaD 
		\left(Z,(X^q)_{q\in\Q},\left(\frac 32 \ell^q\right)_{q\in\Q}\right).
	\end{split}
	\label{fourtrees-eq1}
	\eqe
	\label{prop:multiplefl}
\end{prop}
\begin{proof}
	Define $Y^m:=(Z^m,(X^{q,m})_{q\in\Q},(\ell^{q,m})_{q\in\Q})$ to be the random infinite collection of functions on the left-hand side of \eqref{fourtrees-eq1}. The laws of $Y^m$ for $m\in\N$ are tight: tightness of the laws of $Z^m$ and $X^{q,m}$ follows from Donsker's theorem; tightness of the laws of $\ell^{q,m}$ follows from Proposition \ref{prop10} and invariance under re-rooting of the law of the infinite volume bipolar-oriented random planar map. Consider a subsequence along which $Y^m$ converges in law. By the Skorokhod representation theorem we may choose a coupling such that the convergence along this subsequence is almost sure, i.e., we can find random variables such that the following convergence holds almost surely along this subsequence for the topology of uniform convergence on compact intervals for every finite collection of times $q\in \BB Q$:
	\eqb
	(Z^m,(X^{q,m})_{q\in\Q},(\ell^{q,m})_{q\in\Q})
	\rta
	\left(Z,(\wh X^{q})_{q\in\Q},\left(\frac 32 \wh\ell^{q}\right)_{q\in\Q}\right).
	\eqe
	Recall from Section~\ref{sec:continuum-decomp} that the random variables $X^q,\ell^q$ for $q\in\Q$ are a.s.\ determined by $Z$. We need to show that $\wh X^q=X^q$ and $\wh\ell^q=\ell^q$ a.s.\ for every $q\in\Q$. This result holds for $q=0$ by Propositions \ref{prop:triple-conv} and \ref{prop10}. By Proposition \ref{prop10} it is sufficient to prove that $\wh X^{q}=X^q$ a.s.\ for each $q\in\Q$, since $(\wh X^q , \wh\ell^q) \eqD (X^q , \ell^q)$ so $\wh X^{q}=X^q$ implies that $\wh \ell^{q}=\ell^q$.  
	By invariance in law under recentering of the infinite volume bipolar-oriented map and its continuum analogue and by Proposition~\ref{prop:triple-conv}, it holds that $(Z^m_{\lfloor\cdot+q\rfloor}-Z_q^m,X^{q,m})\rtaD (Z,X)\eqD(Z_{\cdot+q}-Z_q,X^q)$. Therefore $(Z  , \wh X^q) \eqD (Z , X^q)$. Since $X^q$ is a.s.\ determined by $Z$ we infer that $\wh X^q = X^q$ a.s. 
\end{proof}

For $q\in\Q$ let $t_q\in\R$ denote the time when the space-filling SLE$_{12}$ counterflow line $\wt\eta'$, which travels in west direction, hits $\eta'(q)$. Let $n_q^{m}$ denote the time when the dual space-filling path $\wt\lambda'$ visits the edge which was hit by the primal discrete space-filling path $\lambda$ at time $\lfloor q m \rfloor$, and define $t_q{m}=m^{-1}n_q^{m}$. Let $s_q^{N}$ (resp.\ $s_q^{S}$) denote the time when $\eta'$ hits the point where the north-going (resp.\ south-going) flowline from 0 and $\eta'(q)$ merge, and define $\BB s_q:=(s_q^{N},s_q^{S})$. Let $\iota_{0 , \lfloor q m \rfloor}^{ N}$ (resp.\ $\iota_{0 , \lfloor q m \rfloor}^{ S}$) denote the time when $\lambda'$ visits the edge where the discrete north-going (resp.\ south-going) flow lines from $\lambda'(0)$ and $\lambda'(\lfloor qm \rfloor)$ merge, as in Section~\ref{sec:merge-tight}; and define $s_q^{m,N}=m^{-1} \iota_{0 , \lfloor q m \rfloor}^{ N}$ (resp.\ $s_q^{m,S}=m^{-1} \iota_{0 , \lfloor q m \rfloor}^{ S}$). Also define $\BB s_q^{m}:=(s_q^{m,N},s_q^{m,S})$. See Figure \ref{fig:fourtrees} for an illustration.

\begin{figure}[ht!]
	\begin{center}
		\includegraphics[scale=1]{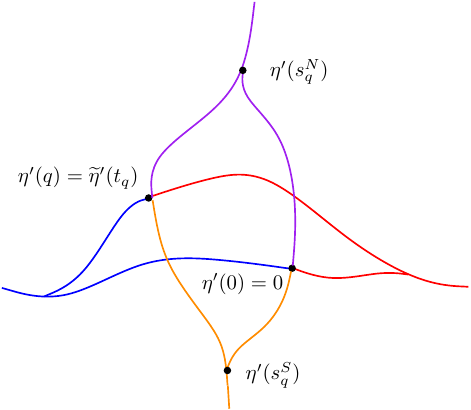}
	\end{center}
	\caption{Illustration of the proof of Theorem \ref{thm1}. The curves in red (resp.\ blue, purple, orange) are flow lines in east (resp.\ west, north, south) direction of the Gaussian free field from which $\eta'$ is generated.} \label{fig:fourtrees}
\end{figure}

\begin{prop}
	For each fixed $q\in\Q$ the laws of the random variables $t_q$ and $\BB s_q^{m}$ defined above are tight.
	\label{prop:tight-ts}
\end{prop}
\begin{proof}
	Tightness of $t_q^{m}$ is immediate from Proposition \ref{prop:merge-tight}. The remainder of the proof will consist of proving that $s_q^{m,N}$ is tight; the proof for $s_q^{m,S}$ is identical. Since $\ell^m_{s_q^{m,N}}=\inf_{0\leq t\leq t_q^{m}} \wt R^m_t$ by definition of $\ell^m,s_q^{m,N},\wt R^m$, the laws of the random variable $\ell^m_{s_q^{m,N}}$ are tight by tightness of $t_q^m$ and tightness of $\wt R^m$ (which follows from Proposition \ref{prop:twoNtrees}). 
	Since $\ell^{m}\rtaD \ell$ and $\ell|_{[0,\infty)}$ is a.s.\ an increasing function converging to $\infty$ it holds that the laws of $s_q^{m,N}$ are tight.
\end{proof}

\begin{proof}[Proof of Theorem \ref{thm1}]
	Consider the infinite collections of random functions
	\eqb
	(Z^m,(X^{q,m})_{q\in\Q},(\ell^{q,m})_{q\in\Q},(t_q^{m})_{q\in\Q},(\BB s_q^{m})_{q\in\Q},\wt Z^m) .
	\label{eq10}
	\eqe 
	The laws of these collections are tight: tightness of $(Z^m,(X^{q,m})_{q\in\Q},(\ell^{q,m})_{q\in\Q})$ follows from Proposition \ref{prop:multiplefl}; 
	tightness of $t_q^{m}$ and $\BB s_q^{m}$ follow from Proposition \ref{prop:tight-ts};
	tightness of $\wt Z^m$ follows from Proposition \ref{prop:twoNtrees}. Consider a subsequence along which the random variable \eqref{eq10} converges in law. 
	By the Skorokhod representation theorem we may choose a coupling such that the convergence along this subsequence is almost sure, i.e., we can find random variables such that the following convergence holds almost surely along a subsequence for the topology of uniform convergence on compact intervals for every finite set of $q\in\Q$:
	\eqb
	\begin{split}
		(Z^m,(X^{q,m})_{q\in\Q},&(\ell^{q,m})_{q\in\Q},(t_q^{m})_{q\in\Q},(\BB s_q^{m})_{q\in\Q},\wt Z^m)\\
		&\rta
		\left(Z,(\wh X^{q})_{q\in\Q},(\wh\ell^{q})_{q\in\Q},(\wh t_{q})_{q\in\Q},(\wh {\BB s}_{q})_{q\in\Q}, 3\wh Z\right).
	\end{split}
	\label{eq11}
	\eqe
	We want to show that $\wh Z= \wt Z$ a.s., where $Z$ and $\wt Z$ are coupled together (and measurable with respect to each other) as described in Section \ref{sec:mainres}. Recall that the collection of random objects $(Z,(X^{q})_{q\in\Q},(\ell^{q})_{q\in\Q},(t_{q})_{q\in\Q},(\BB s_{q})_{q\in\Q},\wt Z)$ is a.s.\ determined by $Z$, hence we would expect (but have not proved) that the right-hand side of \eqref{eq11} is equal to this collection of random objects. However, it is immediate from Proposition \ref{prop:multiplefl} that $((\wh X^{q})_{q\in\Q},(\frac 32 \wh\ell^{q})_{q\in\Q})=((X^{q})_{q\in\Q},(\frac 32 \ell^{q})_{q\in\Q})$.
	
	First we will argue that $\{t_q\,:\,q\in\Q\}$ is a dense subset of $\R$. Observe that $\{\eta'(q)\,:\,q\in\Q\}$ is a dense subset of $\C$, since for any $\ep>0$ and $z\in\C$ the pre-image $(\eta')^{-1}(B_\ep(z))$ contains an interval. Next note that if $\{t_q\,:\,q\in\Q\}$ was not dense there would be an interval $I\subset\R$ such that $\eta'(I)\cap \{\eta'(q)\,:\,q\in\Q\}=\emptyset$, which is a contradiction to our first observation, since for each $I\subset\R$ there is an $\ep>0$ and a $z\in\C$ such that $B_\ep(z)\subset\eta'(I)$. It follows that $\{t_q\,:\,q\in\Q\}$ is a dense subset of $\R$.
	
	Since $\{t_q\,:\,q\in\Q\}$ is dense in $\R$, since $\wh Z$ is continuous, and by left/right symmetry, in order to prove that $\wh Z = 3\wt Z$ a.s.\ it is sufficient to prove that for any $q\in\Q$ it holds a.s.\ that $\wh R_{t_q}=\wt R_{t_q}$. We claim that it is sufficient to establish the following two results in order to prove that $\wh R_{t_q}=\wt R_{t_q}$: (i) a.s.-$\lim_{m\rta\infty}t_q^{m} = t_q$, and (ii) a.s.-$\lim_{m\rta\infty}\wt R^m_{t_q^{m}} = 3\wt R_{t_q}$. The claim follows by observing that
	\eqbn
	|\wh R_{t_q} - \wt R_{t_q}|
	\leq \left|\wh R_{t_q} -  \frac 13 \wt R^m_{t_q^{m}}\right|
	+ \left| \frac 13 \wt R^m_{t_q^{m}} - \wt R_{t_q}\right|,
	\eqen
	where the first term on the right-hand side converges to zero by (i) and the definition of $\wh R$, and the second term on the right-hand side converges to zero by (ii).
	
	First we will prove (i). Consider the two pairs of north-going and south-going flow lines started from $\eta'(q)$ and 0, respectively. Define $\mcl T\subset\R$ (resp.\ $\cT_q\subset\R$) to be the set of times $t\in\R$ when $\eta'(t)$ is to the east of the two flow lines started from $\eta'(q)$ and 0, respectively, and let $\cT^m\subset\R$ (resp.\ $\cT_q^{m}\subset\R$) be a normalized discrete analogue of this set. More precisely, we define $\mcl T^{m}$ (resp.\ $\cT_q^{m}$) to be the union of intervals $m^{-1}[n,n+1]$, $n\in\Z$, such that $\lambda'(n)$ is to the east of the union of the discrete north-going and south-going flow lines started from 0 (resp.\ $\lambda'(\lfloor m q \rfloor)$). Letting $\mu_0$ denote Lebesgue measure on $\BB R$ it holds that 
	\eqbn
	\begin{split}
		m^{-1}n_q^{m}&=\mu_0(\cT_q^{m}\backslash \cT^{m})- \mu_0(\cT^{m}\backslash \cT_q^{m}),\\
		t_q &= \mu_0(\cT_q\backslash \cT) - \mu_0(\cT\backslash \cT_q).
	\end{split}
	\eqen
	Note that one of $\cT_q\backslash \cT$ or $\cT\backslash \cT_q$ is empty, depending on whether $\eta'(q)$ lies to the east or west of the flow lines started from 0. In order to complete the proof of (i), it is therefore sufficient to prove that 
	$\lim_{m\rta\infty}\mu_0((\cT\Delta \cT^m))= 0$ and
	$\lim_{m\rta\infty}\mu_0((\cT_q\Delta \cT_q^{m}))= 0$ a.s.,
	where $\Delta$ denotes symmetric difference. The proof of these two limits are done similarly, so we will only prove that $\lim_{m\rta\infty}\mu_0((\cT\Delta \cT^m))= 0$ a.s. Almost surely, for each sufficiently large $m\in\N$ it holds that $\cT\Delta \cT^m\subset[-2s_q^{S},2s_q^{N}]$. Hence uniform convergence of $(L^m,R^m,X^{m})$ to $(L,R,X)$ on compact intervals implies that $\mu_0((\cT\Delta \cT^m))\rta 0$, since an interval $I=[t_1,t_2]$ is contained in an east excursion from the north-going flow line started from 0 iff $(X-X_{t_1})|_I=-(L-L_{t_1})|_I$. 
	
	Now we will prove (ii). First we prove that $\wh s_q^{N}\geq s_q^{N}$ (note that we will not prove equality). The function $s\mapsto \ell^q_{s-q}-\ell^0_s$ is constant for $s\geq \wh s_q^{N}$, since $s\mapsto \ell^{q,m}_{s-q} - \ell^{0,m}_{s}$ is constant for $s\geq s_q^{m,N}$, and $(\ell^{q,m}_{s-q} - \ell^{0,m}_{s})_{s\in\R}\rta(\ell^{q}_{s-q} - \ell^{0}_{s})_{s\in\R}$. On the other hand, the function $s\mapsto \ell^q_{s-q}-\ell^0_s$ is a.s.\ not constant on any interval contained in $[0,s_q^{N}]$, i.e., on any interval corresponding to times before the flow lines merge, hence $\wh s_q^{N}\geq s_q^{N}$.
	
	By definition of $\wt R^{m}$ it holds that $\wt R^m_{t_q^{m}} = \ell^{q,m}_{s_q^{m,N}-q} - \ell^{0,m}_{s_q^{m,N}}$, where $a:=(3\alpha/2)^{1/2}$. By continuity of $\ell^q$ and $\ell^0$, the definition of $\wh {\BB s}_q$, and since $\wh\ell^q=\ell^q$ and $\wh\ell^0=\ell^0$ this implies that a.s.\ $\wt R^m_{t_q^{m}} \rta 
	\frac 32 (\ell^{q}_{\wh s_q^{N}-q} - \ell^{0}_{\wh s_q^{N}})$.  By Proposition \ref{prop:local-time}, the definition of $\wt R$, and since $\wh s_q^{N}\geq s_q^{N}$ implies that $(\ell^{q}_{\wh s_q^{N}-q} - \ell^{0}_{\wh s_q^{N}})=(\ell^{q}_{s_q^{N}-q} - \ell^{0}_{s_q^{N}})$, it follows that 
	\eqbn
	\wt R^m_{t_q^{m}} 
	\rta 
	\frac 32 (\ell^{q}_{s_q^{N}-q} - \ell^{0}_{s_q^{N}}) =
	\frac 32 c \wt R_{t_q} ,
	\eqen
	where $c$ is the constant from Proposition~\ref{prop:local-time} with $\kappa' = 12$ and $\theta=0$. By Proposition \ref{prop:twoNtrees}  it holds that $\Var(\wt R^m_t)=9\alpha t^2+o_m(1)$ for any $t \in\R$, and since $\wt R\eqD R$ it holds that $\Var(\wt R_t)=\alpha t^2$. Hence we must have $c=2$, and $\wt R^m_{t_q^{m}} \rta 3 \wt R_{t^q}$.
\end{proof}

We made the following observation in the last paragraph of the above proof.
\begin{cor} \label{cor:c-value}
	In the context of Proposition \ref{prop:local-time} it holds that $c=2$ if $\theta=0$ and $\kappa'=12$.
\end{cor}

\bibliography{ref,ref2,cibib,ref_thispaper}
\bibliographystyle{hmralphaabbrv}
\end{document}